\documentclass[a4paper,10pt]{article}
\usepackage{makeidx}
\usepackage[mathscr]{eucal}
\usepackage{amssymb, amsfonts, amsmath, amsthm, graphicx, cases, color, tikz} 
\usepackage{here}
\usepackage{tasks}
\usepackage{mathtools}

\newtheorem{proposition}{Proposition}[section]
\newtheorem{theorem}{Theorem}[section]

\newtheorem{lemma}[proposition]{Lemma}
\newtheorem{remark}{Remark}[section]
\newtheorem{corollary}[theorem]{Corollary}
\newtheorem*{thank}{Acknowledgments}

\newcommand{\N}{\mathbb{N}}

\newcommand{\R}{\mathbb{R}}

\numberwithin{equation}{section}
\newcommand{\e}{\varepsilon}

\numberwithin{equation}{section}

\author{}
\date{}
\title{Singular solutions and bifurcation 
diagram of semilinear elliptic equations with 
general nonlinearity in two dimensions}

\author{Hiroaki Kikuchi, Kenta Kumagai}
\date{}
\begin{document}
\maketitle

\begin{abstract}
In this paper, we investigate semilinear elliptic equations with general exponential-type nonlinearities
in two dimensions. For such nonlinearities, 
we establish two main results. The first is the construction of a singular solution. Recently, Fujishima, Ioku, Ruf, and Terraneo~\cite{MR4876902} proved the existence of singular solutions under certain assumptions for nonlinearities.
We succeed in relaxing these conditions by providing the precise asymptotic form of a singular solution.
Our second result concerns the bifurcation diagram of regular solutions. While the bifurcation structure has been extensively studied in three or higher dimensions, comparatively little was known in two dimensions until recently. In \cite{kuma2025}, the second author proved that the bifurcation curve possesses infinitely many turning points for supercritical analytic nonlinearities. In the present work, we refine this analysis 
by showing the bifurcation curve oscillates infinitely many times 
\textit{around some point}, without assuming analyticity of the nonlinearities. The novelty of our approach lies in the introduction of a generalized Emden-type transformation.




\end{abstract}
\section{Introduction}
\label{intro}
In this paper, we consider a radial singular solution of the following 
semilinear elliptic equation with $d=2$:
\begin{equation}
\label{singulareq-expo}
-\Delta u =f(u) \qquad \text{in $\mathbb{R}^d$}. 
\end{equation}
Throughout this paper, we assume that 
the nonlinearity $f$ satisfies
    \begin{equation}
    \label{zentei-intro}
    f \in C^1[0, \infty), \hspace{4mm} 
    f > 0 \quad\text{on $[0, \infty)$}. 
    \end{equation}
Here, we say that $U(x)=U(r), r = |x|$ 
is a \textit{radial singular solution} of \eqref{singulareq-expo} if $U 
\in C^{2}(0, \infty)$ satisfies \eqref{singulareq-expo} except for 
$x = 0$
and $U(r)\to \infty$ as $r\to 0$. 
\subsection{Singular solution 
in higher dimensional cases}
The first goal of this paper is to study the existence of singular solutions to \eqref{singulareq-expo} for general 
exponential type nonlinearities $f$.
For higher dimensional cases 
$d\ge 3$, the existence and multiplicity of singular solutions have been studied intensively.
For the case $f(s) =(1+s)^p$, 
so called \textit{Serrin 
exponent} plays an important 
role. 

Indeed, it is known \cite{MR605060} that
in $1<p<\frac{d}{d-2}$, any singular solution behaves like 
the fundamental solution
while in $p>\frac{d}{d-2}$, \eqref{singulareq-expo} has the explicit radial singular solution $U_{\infty}(r)=r^{\frac{-2}{p-1}}-1$. On the other hand, the multiplicity of singular solutions is determined by the Sobolev sub/super-criticality. Indeed, when $\frac{d}{d-2}<p \leq   \frac{d + 2}{d - 2}$, it is shown by 
\cite{MR982351, MR768731} 
that there exist infinitely many radial singular solutions. 
On the other hand, Serrin and Zou \cite{MR1272885} proved the uniqueness of radial singular solutions 
if $p > \frac{d + 2}{d - 2}$. 
For exponential-type nonlinearities, a number of studies for singular solutions have been done such as $f=e^u+o(u)$ in \cite{Mi2015, Mi2020}, $f=e^{u^p}$ in \cite{KiWei}, $f=e^{e^u}$ in \cite{Marius} and so on.

Recently, Fujishima and Ioku \cite{FI1}, 
Miyamoto~\cite{Mi2018}
and Miyamoto and Naito \cite{Mi2023} generalized the above results and clarified the existence, uniqueness/multiplicity and
asymptotic behavior of singular solutions for general nonlinearities. More precisely, in \cite{FI1, Mi2023} the following assumptions are considered:
\begin{equation}
\label{asf1}
F(s):=\int_{s}^{\infty}\frac{dt}{f(t)} <\infty \hspace{4mm}\text{and}\hspace{4mm}
\frac{1}{q}:=\lim_{s \to\infty}
H[f](s) \hspace{4mm}\text{exists}, 
\end{equation}
where 
    \begin{equation*} 
    H[f](s) :=\frac{f(s) f''(s)}{(f'(s))^{2}}.
\end{equation*}
If  $f\in C^2[0,\infty)$ satisfies $f>0$ in $[0,\infty)$
and \eqref{asf1} hold, we deduce that $q\ge 1$, $f'(u), f''(u)>0$ for $u>u_0$ with some $u_0>0$ 
(see e.g. Lemma \ref{lem2-2} below), and 
\begin{equation*}
   \lim_{s \to\infty} F(s)f'(s)=\lim_{s\to\infty} \frac{(F(s))'}{(1/f'(s))'}=q. 
\end{equation*}
Moreover, we mention that the H\"older conjugate of the ratio $q$ measures a variant of the growth rate of $f$. Indeed, we can easily confirm that $q=\frac{p}{p-1}$ if $f=(1+u)^p$ and $q=1$ if $f=e^{u}$.

Fujishima and Ioku \cite{FI1} constructed infinitely many number of radial singular solutions and obtained the asymptotic behavior of them when $\frac{d+2}{4}<q<\frac{d}{2}$.~\footnote{Here, we note 
that $q = \frac{d}{2}$ 
if $f(s) = s^{\frac{d}{d -2}}$
and 
$q = \frac{d + 2}{4}$ if 
$f(s) = s^{\frac{d + 2}{d - 2}}$, respectively.} 
On the other hand, Miyamoto and Naito \cite{Mi2023} proved the uniqueness of singular solutions and obtained the asymptotic behavior of the singular solution when $q<\frac{d+2}{4}$.

\subsection{Singular solution in 
two dimensions}
Now, we pay our attention to the singular solution to \eqref{singulareq-expo} in two dimensions. 
Compared with the higher dimensional 
case $d \geq 3$, it is not 
well understood until recently.

In this case, it is known that $f(s) = e^{s}$ 
corresponds to the Serrin 
critical nonlinearity 
$f(s) = s^{\frac{d}{d - 2}}$ for higher dimensional case. 
Indeed, if the growth 
rate of $f(s)$ is lower than or equal to that of $e^{s}$, then 
any singular solution has 
the same singularity as 
the fundamental solution 
$- \log |x|$ 
(see \cite{MR2538571, MR3944340}). 
On the other hand, 
it is known in 
\cite{MR1338473, MR4876902, MR3944340}
that for 
\begin{equation}
\label{ex1-1}
 f_{1}(s) := \frac{4}{B B'} 
    \frac{e^{s^{B'}}}{s^{2B' -1}}, 
    \qquad 
    f_{2}(s) := 4 
    \frac{e^{e^{s}}}{
    e^{2s}}
    \end{equation}
which satisfy 
$\lim_{s \to \infty}f(s)/e^{s} = \infty$, 
the following explicit 
singular solutions 
    \begin{equation} \label{ex1-2}
  u_{1, \infty}(r) 
  = (- 2 \log r)^{1/B'}, 
    \qquad 
    u_{2, \infty}(r) 
    = \log (- 2 \log r)
    \end{equation}
were founded. 
Note that $\lim_{r \to 0} 
u_{i, \infty}(r)/(- \log r) 
= 0$. 
See also 
\cite{MR3944340, MR4263686} for the existence 
of a singular solution in two dimensions. 

Based on the above 
explicit singular solutions, 
Fujishima, Ioku, Ruf and 
Terraneo~\cite{MR4876902} 
extend above results to a
wider class of nonlinearities. 
To state their results, we prepare several notations. 
Define 
  \[
  \frac{1}{B_{1}[f](s)}:= 
  (- \log F(s)) [1 - f^{\prime}(s) F(s)], 
  \hspace{4mm} 
   \frac{1}{B_{2}[f](s)} := 
   f^{\prime}(s) F(s) \left(- \log F(s) \right)^{2} 
   \left[\frac{f(s) f^{\prime \prime}(s) F(s)}{f^{\prime}(s)} 
   - 1\right]. 
   \]

Then, they assume the following: 
\begin{enumerate}
  \item[(F1)] 
  $f \in C^{2}(s_{0}, \infty), f(s) > 0, F(s) < \infty$ and 
  $f^{\prime}(s) > 0$ for $s > s_{0}$ with some $s_{0} > 0$, 
  where $F$ is defined by \eqref{asf1}. 
  \item[(F2)] 
  The limits $\lim_{s \to \infty} f^{\prime}(s) F(s)$ 
  and $B:=\lim_{s \to \infty} B_{2}[f](s)$ exist and satisfy 
  \[
  \lim_{s \to \infty} f^{\prime}(s) F(s) = 1. 
  \]
\end{enumerate}
In addition, they define 
 \begin{equation} \label{ex1-3}
  \widetilde{u}(r) :=F^{-1} 
  \left[\Phi(v(r))\right]=F^{-1}(\frac{B}{4}r^2(-2\log r+1)).
  \end{equation}
Here, the functions 
$v, \phi$ and $\Phi$ 
are defined by 
  \begin{equation*} 
  v(x) := u_{1,\infty}(x), 
  \quad 
  \phi(s) := f_1(s), 
  \quad 
  \Phi(s) := \int_{s}^{\infty} \frac{d \tau}{f_1(\tau)} 
  = \frac{B}{4} \frac{s^{B^{\prime}} + 1}{e^{s^{B^{\prime}}}}, 
  \end{equation*}
if $B > 1$, and 
  \begin{equation*} 
  v(x) := u_{2,\infty}(x), \quad 
  \phi(s) := f_2(s), \quad 
  \Phi(s) := \int_{s}^{\infty} \frac{d \tau}{f_2(\tau)} 
  = \frac{1}{4} \frac{e^{e^{s}} + 1}{e^{s}} 
  \end{equation*}
if $B = 1$, 
where $u_{i, \infty}$ and 
$f_{i}(s) \; (i = 1, 2)$ are 
defined by \eqref{ex1-1} and 
\eqref{ex1-2}. 
Note that since $v$ satisfies 
  $- \Delta v = \phi(v)$, 
  $\widetilde{u}$ becomes a solution to 
   \[
    - \Delta \widetilde{u} 
    = f(\widetilde{u}) + 
    \frac{|\nabla \widetilde{u}|^2}
    {f(\widetilde{u}) F(\widetilde{u})}
    \left[
    \phi'(v) \Phi(v) - f'(\widetilde{u}) 
    F(\widetilde{u})
    \right]. 
   \] 
In order to regard the second term of the right-hand side of the above 
equation as a reminder term, they additionally assumed the following
\begin{enumerate}
  \item[(F3)]
  \[
  \lim_{x \to 0} (- \log |x|)^{\frac{1}{2}} 
  [R_{1}(x) + R_{2}(x)] = 0,  
  \]
where 
  \[
  R_{1}(x) := \biggl| 
  \frac{1}{B_{1}[f] (\widetilde{u}(x))} - 
  \frac{1}{B_{1}[\phi] (v(x))}
  \biggl|, \hspace{4mm} 
   R_{2}(x) := \biggl| 
  \frac{1}{B_{2}[f] (\widetilde{u}(x))} - 
  \frac{1}{B_{2}[\phi] (v(x))}
  \biggl|. 
  \]
  \end{enumerate}
In \cite{MR4876902}, it was shown that 
under conditions (F1)--(F3),
\eqref{singulareq-expo} has a singular solution 
satisfying 
\begin{equation}
\label{eq1-FI}
  u(x) = \widetilde{u}(x) + O(
f(\widetilde{u}(x)) F(\widetilde{u}(x)) R(x)) 
\hspace{4mm} \mbox{as $x \to 0$}\hspace{4mm}\text{with $R(x)=\sup_{|y| \leq |x|} (R_{1}(y) + R_{2}(y)) $}. 
\end{equation}
We remark that 
    \begin{equation} \label{eq1-1}
f(s) = s ^{r} e^{s^{p}}\; (p > 1, r \in R), 
    \hspace{4mm} 
f(s) = e^{s^{p} + s^{r}}\; 
(\mbox{$p > 1, p/2 > r > 0$ or 
$1 < p < 4, r > p - 1$}), 
\hspace{4mm} 
f(s) = e^{e^{s}}
    \end{equation}
satisfy 
conditions (F1)--(F3).  
However, from their result, 
we do not know whether there 
exists a singular solution or not for 
the following nonlinearities 
    \begin{equation} \label{non1}
   f(s) = e^{s^{p} (\log s)^r}\quad 
    (p > 1, r \in \R \setminus \{0\}), 
    \hspace{4mm} 
    f(s) = 
    \exp[\exp[\exp[s]]] 
    \quad \mbox{(triple 
    exponential)}
    \end{equation}
and 
    \begin{equation} \label{non2}
    f(s) = e^{s^{p} + s^{r}}\quad  
\mbox{($(p, r)$ which does not satisfy 
the conditions in \eqref{eq1-1})}. 
    \end{equation}
See Remark 2.2 of 
\cite{MR4876902}. 

Our first result is concerned with 
the construction of a singular solution 
for general nonlinearities 
in which 
the above nonlinearities 
\eqref{non1} and \eqref{non2}
are included. 
In order to state our result, 
we need preparations. 
We put $g(s) := \log f(s)$ and 
    \[
    H(s) := H[g](s) = 
    \frac{g(s)g''(s)}{(g'(s))^2}. 
    \]
Throughout this paper, we always assume the following conditions: 
    \begin{enumerate}
        \item[(G1)]
        The function $g >0$ satisfies 
            \[\text{$ 
        \int_{s}^{\infty}\frac{\,du}
        {g(u)}<\infty$ for all $s$ 
        large, and there exists}\; 
        \frac{1}{q} = 
          \lim_{s \to \infty} H(s).
            \]
    \item[(G2)]
There exists $C>0$ and $s_0 > 0$ such that
$g \in C^{5}(s_{0}, \infty)$ and 
$|H^{(k)}(s)|<C(\frac{g'}{g})^k(s)$ for any $k\in \{1,2,3\}$ and $s>s_0$. 
    \end{enumerate}
Then, we obtain the following:
    \begin{theorem}\label{thm-sing}
    Let $d = 2$. 
    Assume that (G1) and (G2). For any $\e\in (0,\frac{1}{2}]$,  
        \eqref{singulareq-expo} has a radial singular 
        solution $U_{\infty}$
      satisfying       
        \begin{equation} 
        \label{eq-asy}
        \begin{split}
        g(U_{\infty}(r)) 
    = & -2 \log r 
    -2\log (-2 \log r) + 
    \log [\frac{g}{g'}(g^{-1}(-2 \log r))] 
    +\log 4 - \log q\\
    &  + \log qH(g^{-1}(-2 \log r)) 
    + o((- \log r)^{-1+\varepsilon}) 
    \hspace{4mm} \mbox{as $r \to 0$}. 
        \end{split}
        \end{equation} In addition, $U_{\infty}$ has a zero at some $r=r_0$.
        \end{theorem}
\begin{remark}
\rm{
\begin{enumerate}
\item The condition (G2) is not so strong assumption. Indeed, typical nonlinearities (stated in Subsection 2.1) satisfy (G1), (G2) and \eqref{intro-eq}.
\begin{equation}
\label{intro-eq}
H^{(k)}=o((g'/g)^{k}) \qquad\text{for all $k\in\{1,2\}$}.
\end{equation}
In addition, if (G1), (G2) and \eqref{intro-eq} are satisfied, we obtain $(F1)$ and $(F2)$ with $B=q$ (see Lemma \ref{bnolem}). On the other hand, under the same assumptions, $(F3)$ holds if and only if 
\begin{equation}
\label{doutiF3}
    g^{1/2}(s)(H(s)-\frac{1}{q})=o(1)\quad\text{and}\quad g^{1/2}(s)
    \frac{H'g}{g'}(s)=o(1) 
    \qquad \mbox{as 
    $s \to \infty$}.
\end{equation}
We remark that $f=e^{e^{e^u}}$ and $f=e^{u^p(\log u)^q}$ does not satisfy \eqref{doutiF3} (see Subsection 2.1).
\item Their approach in \cite{MR4876902} 
is based on the 
expectation that the asymptotic 
behavior of the singular solution is 
sufficiently close to $\Tilde{u}$ 
defined in \eqref{ex1-3}, and on constructing the singular solution 
$u$, which satisfies 
    \[
g(u)=g(\Tilde{u})+O(R(-2\log r))=g(\Tilde{u})+o((-\log r)^{-1/2}), \qquad\text{(see Lemma \ref{apenlem})}
    \]
by solving the remainder equation of 
$u-\Tilde{u}$
via the contraction mapping theorem. 
It is worth mentioning that their singular solution coincides with our singular solution $U_{\infty}$ 
under condition $(F3)$;
moreover, Theorem \ref{thm-sing} 
provides the following finer asymptotic behavior 
(by letting $\e=1/2$)
    \begin{equation} \label{eq-asym1}
g(U_{\infty})=g(\Tilde{u})+\log q H(g^{-1}(-2\log r))+o((-\log r)^{-1/2}). 
    \end{equation}
See Lemma \ref{apenlem} and Remark \ref{apenrem}. 
We point out that the second term 
$\log q H(g^{-1}(-2\log r))$ of 
\eqref{eq-asym1} 
plays a key role in the construction of the singular solution via a contraction argument. 
Indeed, under condition (F3), 
the second term satisfies 
$\log q H(g^{-1}(-2\log r)) = o(
(-\log r)^{-1/2})$, which can be regarded as the remainder term. 
Hence, the singular solution can be constructed (through the contraction argument) based solely on $\Tilde{u}$. In particular, their singular solution 
$u$ corresponds to our solution 
$U_{\infty}$. On the other hand, if 
$H(s)-1/q$ converges to $0$ slowly, the second term $\log q H(g^{-1}(-2\log r))$ 
can no longer be treated as a remainder term, and thus their approach fails.


\item Our approach does not rely on the explicit function $\Tilde{u}$, 
but instead provides the first and second asymptotics explicitly, as in \cite{MR4263686}. These finer asymptotics enable us to apply the contraction argument, thereby allowing us to remove the technical condition (F3). In particular, our result covers the triple exponential nonlinearity 
$f(s) = \exp[\exp[\exp[s]]]$, among others, which were not treated in previous works (see Subsection 
\ref{ex-non}).

\item The second author \cite{kuma2025} showed the existence of singular solution if $f$ satisfies the supercritical condition \eqref{supercritical} below in the sense of Trudinger-Moser embedding. However, there are little information about the singular solution.

\item 
From the results in the higher-dimensional case, we expect that 
\eqref{singulareq-expo} 
admits infinitely many singular solutions if the nonlinearity 
$f$ is subcritical in the sense of the Trudinger–Moser embedding, whereas 
\eqref{singulareq-expo} possesses a unique singular solution if 
$f$ is supercritical. To the best of the authors’ knowledge, this problem remains open. 
Here, we stress that 
the situation in two dimensional 
case is different from 
that in the higher dimensional
case and it has own 
difficulty. For the uniqueness problem in the higher-dimensional case with $f=e^u$, 
the Emden transformation, which transfers any singular solution $u$ to the difference between 
$u$ and the specific singular solution $u_*$ to \eqref{singulareq-expo} with $f(u) = e^u$, plays a key role. Indeed, the difference 
$u - u_*$ satisfies a second-order differential equation that admits a Lyapunov function, thereby allowing the use of phase plane analysis 
(see (2.21) in Miyamoto-Naito \cite{Mi2020}). As a consequence, the uniqueness of singular solutions can be established. We remark that for the general exponential type nonlinearities, the generalized Cole–Hopf transformation 
(see \cite{Mi2018,Mi2023}) 
reduces the analysis to the exponential nonlinearity. 
On the other hand, this approach appears to fail when the equation for the difference $u - u_*$ does not admit a Lyapunov function due to the absence of the first derivative term.

\end{enumerate}
}
\end{remark}

\subsection{Application: bifurcation diagram for supercritical nonlinearities}
As an application, we clarify the bifurcation diagram of 
regular (classical) solution to the following
\begin{equation}\label{eq-expo}
    \begin{cases}
    - \Delta u = \lambda f(u) 
    & \hspace{4mm} \mbox{in $B_{1}$}, \\
    u = 0 & \hspace{4mm} \mbox{on $B_{1}$},  
    \end{cases}
    \end{equation}
    where $\lambda > 0$ is a parameter and 
$B_{1}$ is the unit ball in $\R^{d}$ with $d=2$.
By the maximum principle, the solution 
to \eqref{eq-expo} becomes positive.  
In addition, it follows from the result 
of Gidas, Ni and Nirenberg~\cite{Gidas} that 
the solution to \eqref{eq-expo} automatically 
radially symmetric and decreasing. Thus our equation in \eqref{eq-expo} is reduced 
to an ordinary differential equation. Moreover, we note that if $v$ is a regular/singular solution of \eqref{singulareq-expo} and $r_0$ is the first zero of $v$, $(\lambda,u)=(r_{0}^2, v(r_0r))$ is a regular/singular solution of \eqref{eq-expo}. From this fact, we can verify that \eqref{eq-expo} has 
an unbounded $C^{1}$ curve emanating from $(0,0)$ and
represented by
$\{(\lambda(\alpha), u(r, \alpha))\colon 
\alpha > 0\}$, where $u(r, \alpha)$ is 
 a solution satisfying 
 $\|u(\cdot, \alpha)\|_{L^{\infty}}=u(0, \alpha) = \alpha$ (see \cite{korman, Mi2015}).

First, we consider  
the case where 
$3\le d\le 9$ and $f$ is supercritical,
that is, 
the exponent $q$ 
defined by 
\eqref{asf1} satisfies 
$q < \frac{d + 2}
{4}$. 
In this case, 
the 
bifurcation structure is extensively studied (see e.g. \cite{Nor, BD, Dan, Dancer-e, 
KiWei, korman, Marius, Merle, Mi2014, Mi2015, Mi2018-2, Mi2023} and references therein); and finally, Miyamoto \cite{Mi2018} and Miyamoto and Naito \cite{Mi2023} 
proved that the bifurcation curve oscillates infinitely often around $\lambda=\lambda_{*}$
converges to the unique singular solution $(\lambda_{*}, U_{*})$. In particular, there exist infinitely many solutions for $\lambda=\lambda_{*}$. 

Our second purpose is to show that the bifurcation curve have infinitely turning points around some 
$\lambda = \lambda_{*}$ for even when $d = 2$ and $f$ is supercritical in the sense of Trudinger-Moser inequality. 
Note that Atkinson and Peletier
~\cite{MR823114} 
proved that 
the bifurcation curve $(\lambda(\alpha), 
u(r, \alpha))$ 
has a unique turning point and converges to  
$(\lambda, u) = (0, \infty)$ as $\alpha 
\to \infty$ for the subcritical nonlinearity $f=e^{u^p}$ with $1\le p<2$
(see also Volker~\cite{MR1319012}). 
On the other hand, until recently, there have been few studies trying to prove these properties for supercritical nonlinearities in two dimensions. 
 The second author \cite{kuma2025}  proved that the bifurcation curve has infinitely many turning points for any analytic nonlinearities satisfying the supercritical 
 condition \eqref{supercritical} 
 below. 
 Note that \eqref{supercritical} 
 holds if and only if $q < 2$ 
 (see Lemma \ref{lem2.6}). 
 However, it is not 
 clear whether the bifurcation curve oscillates around some $\lambda_{*}$ 
 or not. 
 Recently, Naimen \cite{Naimen-accept,naimen2025-1, naimen2025-2} established
for a class of nonlinearities including $f(s)=e^{s^p}$ with $p>2$ and $f=e^{e^s}$, by analyzing the concentration phenomenon of the solutions, the bifurcation curve oscillates infinitely many times around some $\lambda_{*}$. We remark that the approach needs  
conditions on the existence and the asymptotic behavior of the singular solution, 
which are ensured by Fujishima 
and Ioku~\cite{MR4876902}  
for the nonlinearities 
$f(s) = e^{s^p}, 
e^{e^s}$. 
However, we 
do not know whether 
the condition holds 
for general supercritical 
nonlinearity 
(see Page 22 of 
Naimen~\cite{naimen2025-2}). 

Thanks to Theorem \ref{thm-sing}, we obtain the following result for general supercritical nonlinearities:

\begin{theorem}\label{thm-bi}
Let $d = 2$. 
    Assume that \eqref{zentei-intro}, (G1), (G2) 
    and $1 \leq q < 2$. We 
    define $\lambda_{\infty}=r_{0}^2$, where $r_0$ is the first zero of the singular solution $U_{\infty}$ obtained in 
    Theorem \ref{thm-sing}. 
    Then, the bifurcation 
curve of \eqref{eq-expo} has 
    infinitely many turning point around 
    $\lambda = \lambda_{\infty}$. In particular, \eqref{eq-expo} has infinitely many solutions at $\lambda 
    = \lambda_{\infty}$.
    \end{theorem}
\begin{remark}
\rm{
\begin{enumerate}
        \item We do not assume the 
        analyticity on the nonlinearity 
        $f$. 
        \item 
        While we were preparing this paper, 
        Naimen 
        posted related papers~\cite{naimen2025-1, naimen2025-2} on arXiv. 
        In fact, once Theorem \ref{thm-sing} 
        is established and some additional conditions of $f$ are imposed 
        (see (H1) 
        (\textrm{ii}) 
        in \cite{naimen2025-2}), Theorem \ref{thm-bi} 
        follows by applying the result 
        of Naimen~\cite{naimen2025-2}.
        In this paper, we relax the condition of the nonlinearity and give a proof of Theorem 
        \ref{thm-bi} by 
        assuming only 
        (G1) and (G2). 
        Our approach is to study the behavior of \lq\lq general" 
        singular solutions to 
        \eqref{singulareq-expo} near the origin by means of a generalized Emden-type transformation, which plays an important role in the bifurcation analysis in the higher-dimensional case. 
        It seems that our approach is 
        rather simple and short. 
     \end{enumerate} 
     }
    \end{remark}
\begin{remark}[Strategy of the proof of Theorem \ref{thm-bi}]
 \rm{
As mentioned in Remark 1.1, the approach focusing on the standard difference between $U_{\infty}$ and other singular solutions $V_{\infty}$ does not allow us to study the behavior of singular solution in two dimensions. Instead, we find the suitable energy function $\mathcal{F}$ so that $\mathcal{F}(U_{\infty})+2\log r =o(1)$ as $r \to 0$ as constructed by Mcleod-Mcleod \cite{MR952325} 
for special nonlinearities. The function $w=h(V_{\infty})+2\log r$ measures the difference between the solution $U_{\infty}$ and $V_{\infty}$. Moreover, through the specific change of variables, $w$ satisfies a second-order differential equation that admits a Lyapunov function. Therefore, by the Lyapunov-functional method, we can estimate the behavior of the singular solution $V_{\infty}$ near $r=0$. As a result, we can obtain the 
interaction number between the singular 
solution $U_{\infty}$ and regular solutions to 
\eqref{singulareq-expo}; and prove the oscillation property of the bifurcation curve around $\lambda_{*}$. 
The authors believe that this approach plays 
a key role in order to study the 
uniqueness of the singular solution in two dimensions, 
which is a challenging problem as 
we mentioned above.


}
\end{remark}

The rest of this paper is organized as 
follows. 
In Section \ref{sec-pre}, we prove several estimates on nonlinearity 
$f$, which are needed later. 
In Section \ref{sec-sing}, 
we construct a singular solution and 
give the proof of Theorem \ref{thm-sing}. 
In Section \ref{sec-emden}, 
we study several properties of 
singular solutions to \eqref{eq-expo} 
and obtain a result which play an important 
role for the proof of Theorem \ref{thm-bi}. 
Finally, in Section \ref{sec-bif}, 
we obtain an intersection result and 
prove Theorem \ref{thm-bi}.

\section{Preliminaries} \label{sec-pre}
In this section, we will show 
several nonlinear estimates, 
which are needed later 
and give some examples of 
the computation of $H(s)$ 
and $\frac{1}{q} = \lim_{s \to \infty} 
H(s)$. 
Let $\alpha \in \R$ and $\psi$ be a function such that 
$\lim_{s \to \infty} \psi(s) = \infty$. 
we say that 
$\varphi(s)=\psi^{\alpha+ o(1)}(s)$ for any 
$\varepsilon>0$, there exist $s_0$ such that 
$\psi^{\alpha - \varepsilon}(s) < \varphi(s) < 
\psi^{\alpha +\e}(s)$ for all $s>s_0$. 
From \cite[Lemma 2.1]{Mi2023}, we can easily obtain the following: 
\begin{lemma}\label{lem2-2}
Assume (G1) and (G2). 
Then, we have $q\ge 1$. Moreover,
$g(s)$ is strictly monotone increasing and  
convex for 
$s \gg 1$ satisfying 
$\lim_{s \to \infty}g(s) = \infty$. 
\end{lemma}

Then, by the L’Hopital rule, we obtain the 
following lemma:
\begin{lemma}\label{lem2-1}
Assume (G1) and (G2). 
The followings hold: 
\begin{equation} \label{eq-y13}
    \frac{g(s)}{sg'(s)}=\frac{g(s)/g'(s)}{s}=1-\frac{1}{q}+o(1) \hspace{4mm}\text{as $s \to \infty$},
\end{equation}
\begin{equation} \label{eq-y41}
    \frac{\log s}{\log g(s)}=\dfrac{1/s}{g'(s)/g(s)}
    = \frac{g(s)}{s g'(s)}
    =1-\frac{1}{q}+o(1) \hspace{4mm}\text{as $s \to \infty$},
\end{equation}
\begin{equation}
    \label{answer-tec0}
    \frac{\log g(s)}{\log g'(s)}=\frac{g'(s)/g(s)}{g''(s)/g'(s)}= 
    q + o(1) \hspace{4mm}\text{as 
    $s \to \infty$}.
\end{equation}
Moreover, in the case of 
$q > 1$, for any $\e>0$, there exists $s_0>0$ such that
\begin{equation} \label{eq-y29}
    \left(\frac{s}{g(s)^{1-\frac{1}{q}+\varepsilon}}\right)'=g^{-2+\frac{1}{q}-\e}(s)(g(s)-(1-\frac{1}{q}+\e)g'(s)s)\le 0\quad\text{for any $s>s_0$.}
\end{equation}
\end{lemma}

%



\begin{lemma} \label{lem2-3} 
Assume that (G1) and (G2). 
We take $0<t<1$ arbitrarily and fix it. 
Then, we have 
\begin{equation}
    \label{tec2}
    g(t \rho) \le t g(\rho)
\end{equation}
for sufficiently large $\rho > 0$. 
\end{lemma}
\begin{proof}
For each $s > t$, we put 
    \[
    \varphi(s) := \frac{g(s \rho)}{s}. 
    \]
    We see from \eqref{eq-y13} and $q \geq 1$ 
    that
    \[
    \varphi'(s) 
    = \frac{g'(s \rho) s \rho 
    - g(s \rho)}{s^{2}} 
    \geq 
    \frac{1}{2q - 1} \frac{g(s \rho)}{s^2} > 0
    \]
for $t \leq s \leq 1$. 
This implies that 
    \[
    \frac{g(t \rho)}{t} = \varphi(t) 
    \leq \varphi(1) = g(\rho).
    \]
Thus, we obtain the claim. 
\end{proof}
\begin{lemma} \label{lem3-2}
Assume that (G1) and (G2). 
In the case of  
$q > 1$, we obtain
\begin{equation}\label{eq-y52}
    g(s)=s^{\frac{1}{1-\frac{1}{q}+o(1)}} \hspace{4mm}\text{as $s\to\infty$}.
\end{equation}
In the case $q=1$, for any $M>0$, there exists $s_0$ such that 
\begin{equation}\label{eq-y52-1}
    g(s)>s^{M} \hspace{4mm}\text{for all $s>s_0$}.
\end{equation}
In addition, for any $q \geq 1$, it follows that 
\begin{equation}\label{eq-y52-2}
\frac{g(s)}{g'(s)}
=
\begin{cases}
(1 - \frac{1}{q} +o(1)) g(s)^{1-\frac{1}{q} +o(1)}
& \qquad \mbox{if $q > 1$,} \\
g(s)^{1-\frac{1}{q} +o(1)} 
& \qquad \mbox{if $q = 1$}
\end{cases}
\qquad\mbox{as $s \to \infty$}. 
\end{equation}
\end{lemma}
\begin{proof}
We first remark that \eqref{eq-y52} and \eqref{eq-y52-1}
follow from \eqref{eq-y41}. When $q > 1$, since \eqref{eq-y52} can be rewritten by $s=g(s)^{1-\frac{1}{q}+o(1)}$ as $s\to\infty$, \eqref{eq-y52-2} follows from \eqref{eq-y13}. When $q=1$, we deduce from \eqref{answer-tec0} that $g'(s)=g(s)^{1+o(1)}$ as $s\to\infty$. Thus, we have \eqref{eq-y52-2}.
\end{proof}
\begin{lemma}\label{lem2.6}
Assume that (G1) and (G2). 
We obtain that 
\begin{equation*}
\lim_{s \to\infty}(\frac{F(s) 
\log F(s)}{f(s)})'=1-\frac{1}{q}, \hspace{2mm}\text{where}\hspace{2mm} F(s)=\int_{0}^{s}f(u)\,du.
\end{equation*}
In particular, if $1\le q< 2$, we have the following supercritical condition
\begin{equation}\label{supercritical}
    \limsup_{s \to\infty}(\frac{F(s) 
\log F(s)}{f(s)})'<\frac{1}{2}.
\end{equation}
\end{lemma}
\begin{proof}
We fix $\e>0$. We confirm from (G2) that
\begin{equation}\label{H-eq1}
    |(\frac{H(s)}{g(s)g'(s)})'| 
    =|
    \frac{H'(s)}{g(s)g'(s)} - 
    \frac{H(s)}{g^2(s)} - \frac{H^2(s)}
    {g^2(s)}|\le C|\frac{H(s)}
    {g^2(s)}|=o(H(s)/g(s)) 
\end{equation}
as $s \to\infty$. Therefore, since
\begin{equation*}
\int_{s_{0}}^{s} e^{g(u)} 
\frac{H(u)}{g(u)}\,du = 
e^{g(s)}\frac{H(s)}{g(s)g'(s)} - C(s_0)
-\int_{s_0}^{s}e^{g(s)}(\frac{H(u)}
{g(u)g'(u)})'\,du,
\end{equation*}
it follows from \eqref{H-eq1} that there exists a large $s_0$ depending only on $\e$ so that for $s>s_0$, we have
\begin{equation} \label{H-eq2}
  e^{g(s)}\frac{H(s)}{g(s)g'(s)}-C(s_0)\le \int_{s_{0}}^{s} e^{g(u)} 
    \frac{H(u)}{g(u)}\,du \le (1+\e) e^{g(s)}\frac{H(s)}{g(s)g'(s)}-C(s_0).
\end{equation}
Here, it follows from $g(s) = 
\log f(s)$, integral by parts that
\begin{align*}
F(s)=C(s_0)+\int_{s_{0}}^{s} e^{g(u)}\,du 
= \frac{e^{g(s)}}{g'(s)}+ 
\int_{s_{0}}^{s} e^{g(u)} \frac{H(u)}{g(u)}\, du 
+ C(s_0).
\end{align*}
Therefore, by using \eqref{H-eq2}, we have
\begin{align*}
  \frac{e^{g(s)}}{g'(s)}\left(1 
+ \frac{H(s)}{g(s)}+O(g'(s) e^{- g(s)}) \right)\le F(s)\le \frac{e^{g(s)}}{g'(s)}\left(1 
+ (1+\e )\frac{H(s)}{g(s)}+O(g'(s) e^{- g(s)}) \right) \qquad\text{as $s \to\infty$}.
\end{align*}
By the above computation,
\begin{equation*}
     \frac{F(s)f'(s)}{f^2(s)} 
    = \frac{e^{g(s)}g'(s)F(s)}{e^{2g(s)}}=e^{-g(s)}g'(s)F(s)
\end{equation*}
and
\begin{equation*}
\log F(s) 
    = g(s) - \log g'(s) + o(1) 
    = g(s) \left(1 
    + o(1) \right),  
\end{equation*}
we see that
    \[
     -(1+\e+o(1))H(s)-O(g'(s)e^{-g(s)}g(s))\le \log F(s)(1-\frac{F(s)f'(s)}{f^2(s)}) 
    \le -(1 +o(1))H(s)-O(g'(s)g(s) e^{- g(s)})
    \]
as $s \to\infty$. Since 
\begin{equation*}
     (\frac{F(s)\log F(s)}{f(s)})' 
    = 1+(1-\frac{F(s)f'(s)}{f^2(s)}) \log F(s) 
\end{equation*}
and $\e>0$ is arbitrary, we obtain the result.
\end{proof}
\begin{lemma}\label{lem2.7}
Assume that (G1) and (G2). 
Then, there exists $C > 0$ such that 
\begin{equation*}
    |g'''(s)|\le C 
    \frac{g'^3(s)}{g^2(s)}
\end{equation*}
for $s > s_{0}$. 
\end{lemma}
\begin{proof}
By a direct computation, we have
\begin{equation*}
    H'(s)=(\frac{g(s)g''(s)}{g'^2(s)})'=\frac{g'''(s)g(s)}{g'^2(s)}+\frac{g''(s)}{g'(s)}
    -2\frac{g(s)g''^2(s)}{g'^3(s)}.
\end{equation*}
Thus, by the assumption (G2), 
we obtain 
\begin{align*}
|g'''(s)|\le \frac{g'(s)g''(s)}
{g}+2\frac{g''^2(s)}
{g'(s)}+\frac{g'^2(s)}{g(s)}|H'(s)| 
\le 
\frac{H(s)g'^3(s)}
{g^2(s)}+2\frac{H^2(s)g'^3(s)}
{g^2(s)}+C\frac{g'^3(s)}{g^2(s)} 
\le C\frac{g'^3(s)}{g^2(s)}. 
\end{align*}
This completes the proof. 
\end{proof}

\section{Existence of Singular solutions}
\label{sec-sing}
In this section, we give the proof of Theorem 
\ref{thm-sing}. 
We pay our attention to a radially symmetric solution 
to \eqref{eq-expo}. 
Then, the equation in \eqref{eq-expo} is 
reduced to the following ordinary differential equations: 
\begin{equation} \label{eq-me-o}
- \frac{d^{2} u}{d r^{2}} - 
\frac{1}{r}\frac{d u}{d r} = f(u). 
\end{equation}
Here, we employ the following Emden-Fowler transformation: 
	\[
	y(\rho) = u(r), \hspace{4mm} \rho = -2 \log r. 
	\]
Then, we see that
the equation in \eqref{eq-me-o} is 
equivalent to the following: 
	\begin{equation} \label{ODE} 
	- \frac{d^{2} y}{d \rho^{2}} 
	= \frac{1}{4} f(y) e^{-\rho}. 
	\end{equation}
We will show the following: 
\begin{proposition}\label{prop-es}
There exists $\Lambda_{\infty} > 0$ 
and a solution $y_{\infty}(\rho)$ to \eqref{ODE} 
for $\rho \in [\Lambda_{\infty}, \infty)$ satisfying 
\begin{align} \label{eq-yhonto}
    y_{\infty} 
    = g^{-1}(\rho)+\frac{1}{g'(g^{-1}(\rho))}\left(-2\log \rho-\log \frac{g'(g^{-1}(\rho))}{\rho}+\log H(g^{-1}(\rho))+\log 4\right)+o\left(\frac{1}{\rho^{1-\e}g'(g^{-1}(\rho))}\right) 
\end{align}
as $\rho \to \infty$, 
where $g^{-1}$ is the inverse function of $g$ and $\e\in (0,\frac{1}{2}]$ is an arbitrary number. Moreover, the solution of \eqref{ODE} satisfying \eqref{eq-yhonto} with $\e=\frac{1}{2}$ is unique.
\end{proposition}
We remark that the following corollary follows from
Proposition \ref{prop-es}.
\begin{corollary}
\label{cor}
The solution $y_{\infty}(\rho)$ obtained in Proposition \ref{prop-es} satisfies
\begin{align} \label{eq-y24}
    g(y_{\infty}) 
    = \rho-2\log\rho+\log [\frac{g}{g'}(g^{-1}(\rho))]+\log 4H(g^{-1}(\rho)) 
    + o(\rho^{-1+\e}) 
    \hspace{4mm} \mbox{as $\rho \to \infty$}. 
\end{align}
\end{corollary}

\subsection{
Construction of a local solution near the origin}
\subsubsection{First and second orders of asymptotic 
of the singular solution}
To prove Proposition \ref{prop-es}, 
we determine an asymptotic of a solution to 
\eqref{ODE}. 
We choose a first-order asymptotic $y_{1}$
of a singular solution so that 
  \begin{equation} \label{eq-y4} 
  g(y_{1}) - \rho = 0.  
  \end{equation}
Namely, $y_{1} = g^{-1} (\rho)$ 
(note that it follows from Lemma \ref{lem2-2} 
that the function $g(s)$ has the inverse function 
for sufficiently large $\rho > 0$). 
From \eqref{eq-y4} and $g(s) = \log f(s)$, we have 
  \begin{equation} \label{eq-y76}
  1 = g^{\prime}(y_{1}) \frac{d y_1}{d \rho} 
  = \frac{f^{\prime}(y_{1})}{f(y_{1})} \frac{d y_1}{d \rho}.  
  \end{equation}
This implies 
  \begin{equation} \label{eq-y00} 
  f(y_{1}) = f^{\prime}(y_{1}) 
  \frac{d y_1}{d \rho} 
  \hspace{4mm} 
  \left(\frac{d y_1}{d \rho} = \frac{f(y_{1})}{f^{\prime}(y_{1})}\right). 
  \end{equation}
  Differentiating the above equality, we obtain 
  \begin{equation}\label{eq-y0}
  f^{\prime} (y_{1}) \frac{d y_1}{d \rho} 
  = f^{\prime \prime} (y_{1}) \left(\frac{d y_1}{d \rho}
  \right)^{2} + f^{\prime}(y_{1}) \frac{d^2 y_1}{d \rho^2}.
  \end{equation}
It follows from \eqref{eq-y00} 
and \eqref{eq-y0} that 
   \begin{equation} \label{eq-y3}
  \begin{split}
   \frac{d^2 y_1}{d \rho^2} 
  = \frac{d y_1}{d \rho} 
  - \frac{f^{\prime \prime} (y_{1})}{f^{\prime}(y_{1})} 
  \left(\frac{d y_1}{d \rho}
  \right)^{2} 
  = \frac{f(y_{1})}{f^{\prime}(y_{1})}
  - \frac{f^{\prime \prime} (y_{1})}{f^{\prime}(y_{1})} 
  \left(\frac{f(y_{1})}{f^{\prime}(y_{1})}
  \right)^{2} 
  & = \frac{f(y_{1})}{f^{\prime}(y_{1})}
  - \frac{f^{\prime \prime} (y_{1}) f^{2}(y_{1})}
  {(f^{\prime}(y_{1}))^{3}} \\
  & = - \frac{f(y_{1})}{f^{\prime}(y_{1})}
  \left(\frac{f^{\prime \prime} (y_{1}) f(y_{1})}
  {(f^{\prime}(y_{1}))^{2}} - 1\right).  
  \end{split}
    \end{equation}
We claim that 
$\frac{d^2 y_1}{d\rho^2}< 0$ for all $\rho$ sufficiently large. Indeed, it follows from $g(s) = \log f(s)$ that 
\begin{equation} \label{eq-y11}
    g'(s) = \frac{f'(s)}{f(s)}, 
    \hspace{4mm} 
    g''(s) = \frac{f(s) f''(s) - (f'(s))^{2}}
    {(f(s))^{2}}  
    = (g'(s))^{2} \left(\frac{f(s) f''(s)}{(f'(s))^{2}}-1\right). 
    \end{equation}
Therefore, by Lemma \ref{lem2-2} 
and \eqref{eq-y3}, we obtain the result. 

Now, we determine the second-order asymptotic 
$y_{2}$
of a singular solution. Thanks to the above assertion, we can choose $y_{2}$ so that   
  \begin{equation} \label{eq-y2}
  \frac{e^{g^{\prime}(y_{1}) y_{2}}}{4} := -\frac{d^2 y_1}{d\rho^2}=
  - \frac{f(y_{1})}{f^{\prime}(y_{1})}
  \left(1 - \frac{f^{\prime \prime} (y_{1}) f(y_{1})}
  {(f^{\prime}(y_{1}))^{2}}\right)=\frac{g''(y_1)}{g'(y_1)^3}. 
  \end{equation}
Thus, it follows from \eqref{eq-y4}, \eqref{eq-y11} and  \eqref{eq-y2} that
  \begin{equation} \label{shape-y2} 
  \begin{split}
  y_{2} 
  & = \frac{1}{g^{\prime} (y_{1})}
  \left( \log \left[
  \frac{g''(y_{1})}{(g'(y_{1}))^{3}}
  \right] + \log 4
  \right) \\ 
   & = \frac{1}{g^{\prime} (y_{1})}
  \left(-2 \log g(y_{1}) + 2 \log g(y_{1}) + 
  \log \left[\frac{1}{g^{\prime}(y_{1})}\right] 
  + \log \left[
  \frac{g''(y_{1})}{(g'(y_{1}))^{2}}
  \right] + \log 4
  \right) \\ 
    & = \frac{1}{g'(y_1)}(-2\log\rho 
    + \log\frac{g(y_{1})}{g'(y_{1})} + 
    \log  \frac{g(y_{1})g''(y_{1})
    }{(g'(y_{1}))^{2}} +\log 4) \\
    & = \frac{1}{g'(y_1)} 
    (-2\log\rho - \log\frac{g'(y_{1})}{g(y_{1})} 
    +\log G(\rho) + \log 4),
  \end{split}
  \end{equation}
where 
\begin{equation} \label{eq-y59}
    G(\rho) := H(y_1(\rho))= 
    \frac{g(y_1)g^{\prime \prime}(y_{1})}{(g^{\prime}(y_{1})^{2}}. 
    \end{equation}
\subsubsection{Equation of the remainder term}
Let $y$ be a solution to \eqref{ODE}. We define $\eta:=y-y_1-y_2$. 
Put 
    \begin{equation} \label{eq-y96}
    h(\eta) := g(y_1+y_2 + \eta)-g(y_1) 
   - g'(y_{1}) (\eta + y_{2}) - \frac{1}{2}g''(y_{1})(\eta 
    + y_{2})^{2}.
    \end{equation}
Then, we have
\begin{align*}
    -\frac{d^{2} \eta}{d \rho^2}&=\frac{d^{2} y_1}{d \rho^2}+\frac{d^{2} y_2}{d \rho^2}+\frac{1}{4}e^{g(\eta+y_1+y_2)-g(y_1)}\\
    & = \frac{d^{2} y_1}{d \rho^2}+\frac{d^{2} y_2}{d \rho^2} +\frac{1}{4}
   e^{g'(y_{1}) (\eta + y_{2}) + \frac{1}{2}g''(y_{1})(\eta 
    + y_{2})^{2}
    + h(\eta)}\\
    & =  \frac{d^{2} y_1}{d \rho^2}+\frac{d^{2} y_2}{d \rho^2} + 
    \frac{1}{4}
    e^{g'(y_{1}) (\eta + y_{2}) + \frac{1}{2}g''(y_{1})(\eta 
    + y_{2})^{2}} + 
    \frac{1}{4}
    e^{g'(y_{1}) (\eta + y_{2}) + \frac{1}{2}
    g''(y_{1})(\eta 
    + y_{2})^{2}}
    (e^{h(\eta)} - 1) \\
    & =  \frac{d^{2} y_1}{d \rho^2}+\frac{d^{2} y_2}{d \rho^2} + 
    \frac{1}{4}
    e^{g'(y_{1}) (\eta + y_{2})}
    + 
    \frac{1}{4}
    e^{g'(y_{1}) (\eta + y_{2})}
    \left(e^{\frac{1}{2}g''(y_{1})(\eta 
    + y_{2})^{2}} -1 \right)\\
    & \quad + 
    \frac{1}{4}
    e^{g'(y_{1}) (\eta + y_{2}) + \frac{1}{2}
    g''(y_{1})(\eta 
    + y_{2})^{2}}
    (e^{h(\eta)} - 1) \\
    &=\frac{d^{2} y_1}{d \rho^2}+\frac{d^{2} y_2}{d \rho^2}+\frac{1}{4}e^{g'(y_1)y_2} 
    \left\{1+g'(y_1)\eta+
    \left(e^{g'(y_1)\eta}-1-g'(y_1)\eta \right) 
    \right\} \\[6pt]
    & \quad + 
    \frac{1}{4}
    e^{g'(y_{1}) (\eta + y_{2})}
    \left(e^{\frac{1}{2}g''(y_{1})(\eta 
    + y_{2})^{2}} -1 \right) 
    + \frac{1}{4}
    e^{g'(y_{1}) (\eta + y_{2}) + \frac{1}{2}
    g''(y_{1})(\eta 
    + y_{2})^{2}}
    (e^{h(\eta)} - 1).
\end{align*}
Then, by \eqref{eq-y2}, we obtain 
    \begin{equation} \label{eq-eta}
    \begin{split}
    -\frac{d^{2} \eta}{d \rho^2} 
    & = \frac{1}{4}e^{g'(y_1)y_2} 
    g'(y_1)\eta
    + \frac{1}{4}e^{g'(y_1)y_2}
    \left(e^{g'(y_1)\eta}-1-g'(y_1)\eta \right) 
     \\
    & \quad + 
    \frac{1}{4}
    e^{g'(y_{1}) (\eta + y_{2})}
    \left(e^{\frac{1}{2}g''(y_{1})(\eta 
    + y_{2})^{2}} -1 \right) 
    + \frac{1}{4}
    e^{g'(y_{1}) (\eta + y_{2}) + \frac{1}{2}
    g''(y_{1})(\eta 
    + y_{2})^{2}}
    (e^{h(\eta)} - 1) + \frac{d^{2} y_2}{d \rho^2} \\
    &=I(\rho)\eta+ N_{1}(\eta,\rho)+ 
    N_2(\eta,\rho) + 
    N_3(\eta,\rho) + \frac{d^{2} y_2}{d \rho^2}, 
\end{split}
\end{equation}
where $I(\rho)$, $N_{1}(\eta, \rho)$ 
and $N_{2}(\eta, \rho)$ are defined as
    \begin{align}
    & I(\rho) := 
    \frac{1}{4}e^{g'(y_1)y_2}
    g'(y_1),\notag
    \\
    & N_{1}(\eta, \rho) := 
    \frac{1}{4}e^{g'(y_1)y_2} 
    \left(
    e^{g'(y_1)\eta}-1-g'(y_1)\eta \right), 
    \notag
    \\
    & N_{2}(\eta, \rho) := 
    \frac{1}{4}
    e^{g'(y_{1}) (\eta + y_{2})}
    \left(e^{\frac{1}{2}g''(y_{1})(\eta 
    + y_{2})^{2}} -1 \right),\notag \\
    & N_{3}(\eta, \rho) := 
    \frac{1}{4}
    e^{g'(y_{1}) (\eta + y_{2}) + \frac{1}{2}
    g''(y_{1})(\eta 
    + y_{2})^{2}}
    (e^{h(\eta)} - 1). \label{eq-y58}
    \end{align}
It follows from 
\eqref{eq-y4}, \eqref{eq-y2} and 
\eqref{eq-y59} that
$I$ can be written by 
\begin{equation} \label{eq-y7}
\begin{split}
    \quad I(\rho) 
    =g'(y_1)\frac{g''(y_1)}{g'^3(y_1)}=\frac{1}{g(y_1)}G(\rho)
  =\frac{1}{\rho}G(\rho) \hspace{4mm}(\text{i.e.,}\quad  e^{g'(y_1)y_2} = 4 \frac{G(\rho)}{\rho 
    g'(y_{1})}).
\end{split}
\end{equation}
Furthermore, by \eqref{eq-y7}, 
we obtain 
    \[
    \begin{split}
    N_{2}(\eta, \rho) 
    & = 
    \frac{1}{8}
    e^{g'(y_{1}) (\eta + y_{2})}g''(y_{1})
    (\eta + y_2)^2
   + \frac{1}{4}e^{g'(y_{1}) (\eta + y_{2})}
    \left(
    e^{\frac{1}{2}g''(y_{1})(\eta + y_{2})^{2}} -1 
    - \frac{1}{2}g''(y_{1})(\eta + y_2)^2 
    \right) 
 \\[6pt]
    & = \frac{G(\rho)}{2\rho 
    g'(y_{1})}g''(y_{1}) 
    \left(1 + 
\left(e^{g'(y_{1}) \eta} -1\right) \right) 
(\eta + y_2)^2 
+ \frac{G(\rho)}{\rho 
    g'(y_{1})}  
e^{g'(y_{1})\eta}
    \left(
    e^{\frac{1}{2}g''(y_{1})(\eta + y_{2})^{2}} -1 
    - \frac{1}{2}g''(y_{1})(\eta + y_2)^2 
    \right) \\[6pt]
    & = \frac{G(\rho)}{2\rho 
    g'(y_{1})}g''(y_{1}) (\eta + y_2)^2 
    + \frac{G(\rho)}{2\rho 
    g'(y_{1})}
    g''(y_{1})
    \left(e^{g'(y_{1}) \eta} -1\right)
    (\eta + y_2)^2 \\[6pt]
    & \quad 
    + \frac{G(\rho)}{\rho 
    g'(y_{1})}  
    \left(
    e^{\frac{1}{2}g''(y_{1})(\eta + y_{2})^{2}} -1 
    - \frac{1}{2}g''(y_{1})(\eta + y_2)^2 
    \right) \\[6pt]
    & \quad 
    + \frac{G(\rho)}{\rho 
    g'(y_{1})} \left(e^{g'(y_{1})\eta} 
    - 1\right)
    \left(
    e^{\frac{1}{2}g''(y_{1})(\eta + y_{2})^{2}} -1 
    - \frac{1}{2}g''(y_{1})(\eta + y_2)^2 
    \right) \\[6pt]
     & = \frac{G(\rho)}{2\rho 
    g'(y_{1})} g''(y_{1}) y_{2}^{2} 
    + \frac{G(\rho)}{2\rho 
    g'(y_{1})}g''(y_{1}) (2 y_{2}\eta + \eta^2) 
    + \frac{G(\rho)}{2\rho 
    g'(y_{1})} g''(y_{1}) 
    \left(e^{g'(y_{1}) \eta} -1\right)
    (\eta + y_2)^2 \\[6pt]
    & \quad 
    + \frac{G(\rho)}{\rho 
    g'(y_{1})}  
    \left(
    e^{\frac{1}{2}g''(y_{1})(\eta + y_{2})^{2}} -1 
    - \frac{1}{2}g''(y_{1})(\eta + y_2)^2 
    \right) \\[6pt]
    & \quad 
    + \frac{G(\rho)}{\rho 
    g'(y_{1})} \left(e^{g'(y_{1})\eta} 
    - 1\right)
    \left(
    e^{\frac{1}{2}g''(y_{1})(\eta + y_{2})^{2}} -1 
    - \frac{1}{2}g''(y_{1})(\eta + y_2)^2 
    \right) \\
    & = J(\rho)
+  \sum_{i = 1}^{4} N_{2, i} (\eta, \rho), 
    \end{split}
    \]
where 
    \[
    \begin{split}
    &J(\rho) :=\frac{G(\rho)}{2\rho 
    g'(y_{1})} g''(y_{1}) y_{2}^{2},\\ 
    & N_{2, 1}(\eta, \rho) 
    := \frac{G(\rho)}{2\rho 
    g'(y_{1})} g''(y_{1})(\eta^{2}+2 \eta y_2), \\
& N_{2, 2}(\eta, \rho) 
    := \frac{G(\rho)}{2\rho 
    g'(y_{1})} g''(y_{1}) 
    \left(e^{g'(y_{1}) \eta} -1\right)
    (\eta + y_2)^2, \\
& N_{2, 3}(\eta, \rho) 
    := \frac{G(\rho)}{\rho 
    g'(y_{1})}  
    \left(
    e^{\frac{1}{2}g''(y_{1})(\eta + y_{2})^{2}} -1 
    - \frac{1}{2}g''(y_{1})(\eta + y_2)^2 
    \right), \\
& N_{2, 4}(\eta, \rho) 
    := \frac{G(\rho)}{\rho 
    g'(y_{1})} \left(e^{g'(y_{1})\eta} 
    - 1\right)
    \left(
    e^{\frac{1}{2}g''(y_{1})(\eta + y_{2})^{2}} -1 
    - \frac{1}{2}g''(y_{1})(\eta + y_2)^2 
    \right). 
    \end{split}
    \]
    
We now transform the equation \eqref{eq-eta} 
to an integral equation. 
However, since the solution to the linear part of the 
equation \eqref{eq-eta} cannot be 
express by some elementary functions, 
we need some ingredients. 
To this end, we first put 
\begin{equation} \label{eq-y22}
    a^4(\rho) :=\frac{\rho}{G(\rho)},\hspace{4mm
    }
    b(\rho): =\int_{\rho_0}^{\rho}(\frac{G(\tau)}{\tau})^{\frac{1}{2}}\,d\tau
\end{equation}
for sufficiently large 
fixed $\rho_0 >0$. 
Then, we have 
$a^2 b'(\rho) = 1$ for any $\rho > \rho_0$, which implies that 
    \begin{equation} \label{eq-ab}
       2 \frac{d a}{d \rho} \frac{d b}{d \rho} 
       + a \frac{d^2 b}{d \rho^2} = 0. 
    \end{equation}
Then, we can verify that 
if $\eta$ satisfies 
\eqref{eq-eta}, 
we see that $\eta$ is a solution of
    \begin{equation} \label{eq-z2}
     \frac{d^{2} \eta}{d \rho^2}+I \eta -
     \frac{1}{a} \frac{d^2 a}{d \rho^2}\eta 
     = - \frac{d^{2} y_2}{d \rho^2}
    -\frac{1}{a} \frac{d^2 a}{d \rho^2}\eta 
     - J(\rho)
    - N_{1}(\eta,\rho)
    - \sum_{i = 1}^{4} N_{2, i} (\eta, \rho) 
    - N_{3}(\eta, \rho).  
    \end{equation}
We consider the following type of functions
\begin{equation*}
    \Phi_1(\rho)=a(\rho)\sin b(\rho), 
    \hspace{4mm} 
    \Phi_2(\rho)=a(\rho)\cos b(\rho).
\end{equation*}   
Using \eqref{eq-y7} and \eqref{eq-ab}, 
we can compute
\begin{equation*}
\begin{split}
    \frac{d^{2} \Phi_{i}}{d \rho^2} 
    +I\Phi_i - \frac{1}{a} \frac{d^2 a}{d \rho^2} \Phi_i
    & = (\frac{1}{a} \frac{d^2 a}{d \rho^2}-\left(\frac{d b}{d \rho} 
    \right)^{2})\Phi_i
    +\frac{G}{\rho} \Phi_{i}
     - \frac{1}{a} \frac{d^2 a}{d \rho^2} \Phi_i
    =(\frac{1}{a} \frac{d^2 a}{d \rho^2}-\frac{1}{a^4}+\frac{G}{\rho} 
     - \frac{1}{a} \frac{d^2 a}{d \rho^2})\Phi_i 
    = 0.
\end{split}
\end{equation*}
Then, we deduce that 
($\Phi_1, \Phi_2$) is a 
pair of fundamental solutions of the following:
\begin{align} \label{eq-lz}
\frac{d^{2} \eta}{d \rho^2}+I \eta-\frac{1}{a} \frac{d^2 a}{d \rho^2}\eta=0.
\end{align}
Note that the RHS of 
the equation \eqref{eq-lz} is the linear part 
of \eqref{eq-z2}. 
Then, we deduce that 
($\Phi_1, \Phi_2$) is a 
pair of fundamental solutions of 
\begin{align*}
   \frac{d^{2} \eta}{d \rho^2} +I 
  \eta-\frac{1}{a} \frac{d^2 a}{d \rho^2}\eta=0.
\end{align*}
Therefore, the equation \eqref{eq-z2} 
(equivalently, \eqref{eq-eta}) 
with $\lim_{\rho \to \infty} \eta(\rho) = 0$
is equivalent to 
\begin{equation} \label{eq-y23}
\eta(\rho) = 
\int^{\infty}_{\rho}a(\rho)a(\tau)
\sin\left(
(b(\rho)-b(\tau))\right)[\frac{d^{2} y_2}{d \rho^2}(\tau)
+ J(\tau)+\frac{1}{a} 
\frac{d^2 a}{d \rho^2} \eta + N_{1}(\eta,\tau) 
+ \sum_{i = 1}^{4} N_{2, i} (\eta, \rho)
+ N_{3}(\eta,\tau)
]\,d\tau. 
\end{equation}

\subsubsection{Contraction argument and the 
proof of Proposition \ref{prop-es}}
We now give the proof 
of Proposition \ref{prop-es}. 
It suffices to seek a solution to 
the integral equation \eqref{eq-y23} by 
the fixed point theorem. At first, we introduce the following 
\begin{lemma}
\label{lem3-0}
We have
\begin{align}
    & 0<y_{1} \leq 
    C \rho^{1- \frac{1}{q} + o(1)}, \label{eq-y50} \\
    & 0<-y_{2} \leq 
    C \rho^{-\frac{1}{q} + o(1)} (\log \rho) 
    \label{eq-y51-2}
\end{align}
and
\begin{equation}
\label{derivative-y2}
|\frac{d y_2}{d \rho}|\le \frac{C(\log \rho)}{\rho^{1+\frac{1}{q}+o(1)}}, \hspace{4mm} |\frac{d^2 y_2}{d \rho^2}|\le \frac{C(\log \rho)}{\rho^{2+\frac{1}{q}+o(1)}}, \hspace{4mm} |\frac{d^3 y_2}{d \rho^3}|\le \frac{C(\log \rho)}{\rho^{3+\frac{1}{q}+o(1)}}.
\end{equation}
Moreover, we obtain
\begin{equation}\label{derivative-a}
|\frac{1}{a} \frac{d a}{d \rho}|\le \frac{C}{\rho},\hspace{4mm}|\frac{1}{a} 
\frac{d^2 a}{d \rho^2}|\le \frac{C}{\rho^2}.
\end{equation}
\end{lemma}
\begin{proof}
We first remark that we have \eqref{eq-y50} by \eqref{eq-y52} and \eqref{eq-y4}. In addition,
we remind that \eqref{eq-y4} and \eqref{shape-y2} implies
\begin{equation} \label{eq-y77}
y_2= \frac{1}{g'(y_1)} 
    (-2\log\rho - \log\frac{g'(y_{1})}{g(y_{1})} 
    +\log G(\rho) + \log 4)
    = \frac{1}{g'(y_1)} 
    (- \log\rho - \log g'(y_{1}) 
    +\log G(\rho) + \log 4).
\end{equation}
Thus, by Lemma \ref{lem3-2} and \eqref{eq-y4}, we have \eqref{eq-y51-2}.
Since the computation of 
\eqref{derivative-y2} is rather 
technical and long, 
we give it in Appendix \ref{sec-y2} below. 

Finally, we prove \eqref{derivative-a}. 
By direct computation, we obtain
\[4a^3 \frac{d a}{d \rho} = 
\frac{1}{G}-\frac{\rho H'(y_1)}{G^2 g'(y_1)},\]
\[4a^3\frac{d^2 a}{d\rho^2}+12a^2 (\frac{da}{d\rho})^2=\frac{-H'(y_1)}{G^2 g'(y_1)}-\frac{H'(y_1)G^2g'(y_1)+\rho H''(y_1)G^2-2G\rho H'^2(y_1)- G^3 H'(y_1)g'(y_1)}{G^4g'(y_1)^2}.
\]
Thus, the estimate \eqref{derivative-a} follows from Lemma \ref{lem3-2} and (G2).
\end{proof}

\begin{lemma} \label{lem3-1}
Assume that (G1) and (G2)
\begin{equation}\label{eq-y19-1}
 |\mathcal{P}(\frac{d^2 y_2}{d \rho^2})|
\leq \frac{C}{\rho^{1 + \frac{1}{q}+o(1)}}
\log \rho
\hspace{4mm}\text{and}\hspace{4mm}|\mathcal{P}(J)|
\leq \frac{C}{\rho^{1 + \frac{1}{q}+o(1)}} 
\log \rho,
\end{equation}
where
\begin{equation*}
\mathcal{P}(\kappa)=\int^{\infty}_{\rho}a(\rho)a(\tau)\sin\left(2(b(\rho)-b(\tau))\right)\kappa(\tau)\,d\tau.
\end{equation*}
\end{lemma}

\begin{proof}
Observe from \eqref{eq-y4} and 
\eqref{eq-y76} that 
    \[
    \begin{split}
    2J(\rho)=\frac{G(\rho)}{\rho g'(y_{1})} g''(y_{1}) y_{2}^{2}
    = \frac{G(\rho)}{\rho g'(y_{1})} 
    \frac{g(y_{1})}{g'(y_{1})}
    \frac{g'(y_{1})}{g(y_{1})}g''(y_{1}) y_{2}^{2}
    & = \frac{g'(y_{1})}{\rho^2} 
   G(\rho) \frac{g(y_{1}) g''(y_{1})}{
   (g'(y_{1}))^{2}} y_{2}^{2} \\
    & = \frac{g'(y_{1})}{g(y_{1})\rho} 
    G^{2}(\rho) y_{2}^{2}, 
    \end{split}
    \]
\[2J'(\rho)=-4\mathcal{J(\rho)}\rho^{-1}+\frac{g''(y_1)}{\rho^2 g'(y_1)}G^2(\rho)y_2^2+\frac{2H'(y_1)G(\rho)y_{2}^2}{\rho^2}+\frac{2g'(y_1)G^2(\rho)y_2 y_{2}'}{\rho^2}.\]

By combining the above two equalities
with Lemma \ref{lem3-2} and Lemma \ref{lem3-0}, we have   
\begin{equation}
\label{derivative-J}
    |J(\rho)|\le \frac{C}{\rho^{2 +\frac{1}{q}+o(1)}}
    (\log \rho)^2 \hspace{4mm}\text{and}\hspace{4mm}|J'(\rho)|\le \frac{C}{\rho^{3+\frac{1}{q}+o(1)}}
    (\log \rho)^2. 
\end{equation}
Now, for every $\kappa\in C^2(\rho_0,\infty)$ with $\lim_{\rho \to \infty} 
\kappa(\rho) \rho^{\frac{3}{4}} 
= 0$, 
using the fact that 
$a^{2} \frac{d b}{d \rho} = 1$, 
we can verify that
\begin{equation*}
\begin{split}
  \mathcal{P}(\kappa(\rho)):=&\quad 
\int^{\infty}_{\rho}a(\rho)a(\tau)\sin
\left((b(\rho)-
b(\tau))\right)\kappa(\tau)\,d\tau \\ 
   & = 
   \int^{\infty}_{\rho}a(\rho)a(\tau)
   \frac{1}{b'(\tau)}
   \frac{d}{d \tau}
   \left(\cos \left((b(\rho)-
   b(\tau))\right) 
   \right)\kappa(\tau) \,d\tau \\
   & = \int^{\infty}_{\rho}a(\rho)a^3(\tau)
   \frac{d}{d \tau} 
   \left(\cos \left((b(\rho)-
   b(\tau))\right) 
   \right)\kappa(\tau)\,d\tau \\
   &=[a(\rho)a^3(\tau)
   \cos\left(2(b(\rho)-b(\tau))\right)
   \kappa(\tau)
   ]^{\infty}_{\tau=\rho}
   \\
   & \quad -a(\rho)\int^{\infty}_{\rho} 
   (a(\tau)^{3}\kappa'(\tau) 
   +3a^2(\tau)a'(\tau) 
   \kappa(\tau))\cos\left((b(\rho)-b(\tau))\right)\,d\tau\\
   & = -a^{4}(\rho) 
   \kappa(\rho) -a(\rho)\int^{\infty}_{\rho} (a(\tau)^{3}\kappa'(\tau) + 
   3a^2(\tau)a'(\tau)
   \kappa(\tau))
   \cos\left((b(\rho)-b(\tau))\right)\,d\tau.
\end{split}
\end{equation*}
By the above equality, 
\eqref{derivative-a} and 
\eqref{eq-y22}, we have
\[|\mathcal{P}(\kappa)|\le C\rho\kappa+C\rho^{1/4}
\int_{\rho}^{\infty}(\tau^{3/4}|\kappa'(\tau)| 
+\tau^{-1/4}|\kappa(\tau)|) d\tau.\]
Hence, the estimate \eqref{eq-y19-1} follows from \eqref{derivative-y2} and \eqref{derivative-J}.
\end{proof}

We now estimate the nonlinear terms 
$N_{1}, N_{2, i}\; (1 \leq i \leq 4)$ 
and $N_{3}$. 
To this end, we put 
    \[
    h_{1}(\eta, \rho) := (\eta + y_{2})^{2}, 
    \hspace{4mm} 
    h_{2}(\eta, \rho) := e^{g'(y_{1}) \eta} -1, 
    \hspace{4mm} 
    h_{3}(\eta, \rho) := e^{\frac{1}{2}
    g''(y_{1})(\eta+y_2)^2} -1.  
    \]
Then, we first show the following:
\begin{lemma}  
\label{lem3-4}
Assume $\limsup_{\rho\to\infty}g'(y_1)\rho^{1-\e}|\eta(\rho)|\le \delta$ with some $0<\e\le \frac{1}{2}$ and $0<\delta<\frac{1}{100}$. 
Then, one has 
    \begin{align}
    & |h_{1}(\eta, \rho)| \leq 
    C \rho^{-\frac{2}{q} + o(1)}, \label{eq-y50-1} \\
    & |h_{2}(\eta, \rho)|\le 8\delta \rho^{-1 + \e}, 
    \label{eq-y51} \\
   & |h_{3}(\eta, \rho)| 
   \leq C \rho^{-1 +o(1)},  
   \label{eq-y49} \\
   & |h(0, \rho)| \leq  C\rho^{-2+o(1)} 
     \label{eq-y66}
    \end{align}
as $\rho\to\infty$, where $h$ is defined by \eqref{eq-y96}. 
\end{lemma}
\begin{proof}

We first remark that by using \eqref{shape-y2} and \eqref{eq-y52-2}, one has 
\begin{equation}
\label{eq-daiji-1}
|\eta|\le O(\rho^{-1-\frac{1}{q}+\e+o(1)})=o(|y_2|).
\end{equation}
Thus \eqref{eq-y50-1} follows from \eqref{eq-y51-2} and \eqref{shape-y2}. Then, we have 
\[|g'(y_1)\eta|\le 2\delta \rho^{-1+\e}
\qquad\text{if $\rho$ is sufficiently large.}  \]
Thus, we obtain
\[|h_2(\eta, \rho)| \le 4 g'(y_1)|\eta|= 8\delta\rho^{-1+\e}\qquad\text{if $\rho$ is sufficiently large.}
\]
In addition, it follows from the assumption (G1), 
\eqref{eq-y4} and Lemma \ref{lem3-2} that 
    \begin{equation} \label{eq-y54}
    g''(y_{1}) 
    = G(\rho) \frac{(g'(y_{1}))^{2}}
    {g(y_{1})} 
    \leq C \frac{(g'(y_{1}))^{2}}
    {g(y_{1})} 
    \leq Cg(y_{1})^{\frac{2}{q} -1 +o(1)}
    \leq C\rho^{\frac{2}{q}-1+o(1)}. 
    \end{equation}
Thus, we obtain that 
   \begin{equation} \label{eq-y53}
    g''(y_{1})\eta^{2} 
    \leq C\delta^2 \rho^{-1}g'(y_1)^2 
    \times g'(y_1)^{-2}\rho^{-2 + 2\e} = 
    C\delta^2\rho^{-3 + 2\e} \qquad\text{if $\rho$ is sufficiently large.}  
    \end{equation}
In addition, it follows from \eqref{eq-y51-2} and 
\eqref{eq-y54} that 
\begin{equation} \label{eq-y16}
\begin{split}
   g''(y_{1})y^2_2\le C \rho^{-1+o(1)}.
\end{split}
\end{equation}
By \eqref{eq-daiji-1}, \eqref{eq-y53} and \eqref{eq-y16}, one has  
    \[
    |h_{3}(\eta, \rho)| = 
    |e^{\frac{1}{2}
    g''(y_{1})(\eta+y_2)^2} -1| 
    \leq C g''(y_{1})(\eta + y_{2})^{2}
    \leq C \rho^{-1+o(1)}.  
    \]
Thus, we obtain \eqref{eq-y49}.  

Finally, we shall show \eqref{eq-y66}. 
It follows from the convexity of $g$ (see Lemma 
\ref{lem2-2}) and \eqref{eq-y51-2} that   
\begin{equation}
\label{eq-y-special-1}
\frac{g(y_1+y_2)}{g(y_1)}=1+\frac{g(y_1+y_2)-g(y_1)}{g(y_1)}\ge 1+ \frac{g'(y_1)y_2}{g(y_1)}\ge 1-O(\rho^{-1}\log\rho). 
\end{equation}
This together with 
the Taylor expansion, Lemma \ref{lem2-2} and 
\ref{lem2.7}, we have
   \[
    |h(0)| 
    \leq C |g'''(s)||y_{2}|^{3} 
    \leq C\frac{g'(s)^3}{g^2(s)}|y_{2}|^{3} 
    \leq C\frac{g'(y_1)^3}{g(y_1)^2}\frac{g^2(y_1)}{g^2(y_1+y_2)}|y_{2}|^{3}  
    \le C(\log \rho)^3 \rho^{-2 + o(1)},
    \] 
where $s=y_{1} + \theta y_2 \in (y_{1} + y_{2}, y_{1})$. 
This competes with the proof. 
\end{proof}

\begin{lemma}
\label{lem3-5}
Assume that $\limsup_{\rho\to\infty}g'(y_1)\rho^{1-\e}|\eta_i(\rho)|\le \delta$ for $i=1,2$
with some $0<\e\le \frac{1}{2}$ and $0<\delta<\frac{1}{100}$.   
Then, one has 
    \begin{align}
    & |h_{1}(\eta_{1}, \rho) - h_{1}(\eta_{2}, \rho)| 
    \leq C \rho^{- \frac{1}{q} + o(1)} (\log \rho) 
    |\eta_{1} - \eta_{2}|, 
    \label{eq-y61}\\
    & |h_{2}(\eta_{1}, \rho) - h_{2}(\eta_{2}, \rho)| 
    \leq C \rho^{\frac{1}{q} + o(1)} 
    |\eta_{1} - \eta_{2}|, 
    \label{eq-y48} \\
    & |h_{3}(\eta_{1}, \rho) - h_{3}(\eta_{2}, \rho)|
    \leq C \rho^{\frac{1}{q}-1 + o(1)} |\eta_{1} - \eta_{2}|, 
    \notag
    \\
    & |h(\eta_{1}, \rho) - h(\eta_{2}, \rho)|
    \leq C \rho^{\frac{1}{q}-1+ o(1)} |\eta_{1} - \eta_{2}|. \label{eq-y65} 
    \end{align}
\end{lemma}
\begin{proof}
By \eqref{eq-y51} and \eqref{eq-daiji-1}, 
we obtain 
    \[
    |h_{1}(\eta_{1}, \rho) - h_{1}(\eta_{2}, \rho)| 
    =
    |(\eta_{1} + \eta_{2} + 2 y_{2})||\eta_{1} - \eta_{2}| 
    \leq C |y_{2}||\eta_{1} - \eta_{2}| 
    \leq C \rho^{- \frac{1}{q} + o(1)} (\log \rho)
    |\eta_{1} - \eta_{2}|. 
    \]
Thus, \eqref{eq-y61} holds. Moreover, it follows that 
    \begin{equation} \label{eq-y64}
    g'(y_{1})|\eta_{2}| 
    \leq 2\delta\rho^{-1+ \e}\qquad\text{if $\rho$ is sufficiently large.}  
    \end{equation}  
This together with \eqref{eq-y50} yields 
that 
    \[
    \begin{split}
    |h_{2}(\eta_{1}, \rho) - h_{2}(\eta_{2}, \rho)| 
    = |e^{g'(y_{1})\eta_{1}} - 
    e^{g'(y_{1}) \eta_{2}}| 
    \leq e^{g'(y_{1}) \eta_{2}} 
    |e^{g'(y_{1})(\eta_{1} - \eta_{2})} - 1| 
    & \leq e^{ 
    \rho^{-1+\e}}g'(y_{1})|\eta_{1} -\eta_{2}| 
    \\
    & \leq C \rho^{\frac{1}{q} + o(1)} 
    |\eta_{1} - \eta_{2}|. 
    \end{split}
    \]
We observe from \eqref{eq-daiji-1}--
\eqref{eq-y16} that 
    \[
    \begin{split}
    |h_{3}(\eta_{1}, \rho) - h_{3}(\eta_{2}, \rho)| 
    & = \biggl|e^{\frac{1}{2}g''(y_{1})
    (\eta_{1}+y_2)^2} - 
    e^{\frac{1}{2}g''(y_{1})
    (\eta_{2}+y_2)^2} \biggl| \\
    & = e^{\frac{1}{2}g''(y_{1})
    (\eta_{2}+y_2)^2} 
    |e^{\frac{1}{2}g''(y_{1}) 
    \left((\eta_{1} + y_{2})^{2} 
    - (\eta_{2} + y_{2})^{2}\right)} - 1| \\
    & \leq 
    C |g''(y_{1}) 
    \left((\eta_{1} + y_{2})^{2} 
    - (\eta_{2} + y_{2})^{2}\right)| \\
    & \leq C \rho^{\frac{2}{q}-1+o(1)} 
    y_{2} |\eta_{1} - \eta_{2}| \\
    & \leq C \rho^{\frac{1}{q}-1 + o(1)}  
    |\eta_{1} - \eta_{2}|. 
    \end{split}
    \]

Finally, by using \eqref{eq-y-special-1}, we obtain for any $y_1+y_2\le s\le y_1$ that 
\begin{equation*}
\frac{|g''(y_1)-g''(s)|}{g''(y_1)}\le \frac{|g'''(t)|}{g''(y_1)}|y_2|\le\frac{g'(t)^3}{g^2(t)g''(y_1)}|y_2|\le \frac{g'(y_1)^3}{g^2(y_1)g''(y_1)}\frac{g^2(y_1)}{g^2(y_1+y_2)}|y_2|\le  C\rho^{-1}\log\rho,
\end{equation*}
where $s\le t\le  y_1$.
Therefore, by \eqref{eq-y51-2}, \eqref{eq-daiji-1} and \eqref{eq-y54},  
we have
\begin{align*}
    |h(\eta_{1})-h(\eta_{2})|&\le |g(y_1+y_2+\eta_{1})-g(y_1+y_2+\eta_{2})-g'(y_1)(\eta_{1}-\eta_{2})|+\frac{1}{2}|g''(y_1)(\eta_{1}-\eta_{2})(\eta_{1}+\eta_{2}+2y_2)|\\
    &\le \left(|g'(y_1+y_2+\eta_{1}+\theta(\eta_{2}-\eta_{1}))-g'(y_1)|+C|g''(y_1)||y_2|\right)|\eta_{1}-\eta_{2}|\\
    &\le C (g''(s) +g''(y_1))|y_2||\eta_{1}-\eta_{2}|\\
    &\le C\rho^{\frac{1}{q}-1+o(1)}|\eta_{1}-\eta_{2}|, 
\end{align*}
where $0\le \theta\le 1$ and $y_1+y_2\le s\le y_1$.
This completes the proof. 
\end{proof}
    
\begin{lemma}\label{lem-nest}
Assume that $\limsup_{\rho\to\infty}g'(y_1)\rho^{1-\e}|\eta_i(\rho)|\le \delta$ for $i=1,2$ with some $0<\e\le 1/2$ and $\delta>0$. Then, we have
\begin{align}
& |N_{1}(\eta_{1}, \rho) - 
N_{1}(\eta_{2}, \rho)| \leq 
C\delta\rho^{-2+\e}|\eta_{1} - \eta_{2}|,
\label{eq-y63}\\
& |N_{2, 1}(\eta_{1}, \rho) - 
N_{2, 1}(\eta_{2}, \rho)| \leq 
C \rho^{-2 + o(1)} 
|\eta_{1} - \eta_{2}|, 
\label{eq-y43}
\\
& |N_{2, 2}(\eta_{1}, \rho) - 
N_{2, 2}(\eta_{2}, \rho)| \leq 
C \rho^{-2 + o(1)}
|\eta_{1} - \eta_{2}|, 
\label{eq-y44} \\
& |N_{2, 3}(\eta_{1}, \rho) - 
N_{2, 3}(\eta_{2},\rho)| \leq 
C  \rho^{-3 + o(1)}|\eta_{1} - \eta_{2}|, 
\label{eq-y45} \\
& |N_{2, 4}(\eta_{1}, \rho) - 
N_{2, 4}(\eta_{2}, \rho)| \leq 
C \rho^{-3 + o(1)}|\eta_{1} - \eta_{2}|, 
\label{eq-y98}\\
&|N_{3}(\eta_{1},\rho)-N_{3}(\eta_{2},\rho)|\le C\rho^{-2+o(1)}|\eta_{1}-\eta_{2}|.
\label{eq-y46}  
\end{align}
Moreover, we have
\begin{align} \label{eq-y97}
|N_{2,3}(0, \rho)|\le \rho^{-\frac{1}{q}-3+o(1)}, \hspace{4mm}|N_3(0, \rho)| \le \rho^{-\frac{1}{q}-3+o(1)}.
\end{align}
\end{lemma}
\begin{proof}

We first show \eqref{eq-y43}. 
Note that by \eqref{eq-y59},  
assumption (G1) and \eqref{eq-y52-2}, when $q=1$, one has 
\begin{equation} \label{eq-y60}
    \frac{G(\rho)}{\rho 
    g'(y_{1})} g''(y_{1}) 
    \leq \frac{G(\rho)g'(y_{1})}{\rho^{2}} 
    \frac{g(y_{1})g''(y_{1})}{(g'(y_{1}))^{2}}
    \leq C \frac{g'(y_{1})}{\rho^{2}}
    \leq \frac{C}{\rho^{2-\frac{1}{q}+o(1)}}.
    \end{equation}
This together with \eqref{eq-y51-2} and \eqref{eq-daiji-1}
yields that  
    \[
    \begin{split}
    |N_{2, 1}(\eta_{1}, \rho) - 
    N_{2, 1}(\eta_{2}, \rho)|
    & \leq \frac{G(\rho)}{\rho g'(y_{1})} 
    g''(y_{1})
    \left(|\eta_{1} + \eta_{2}||\eta_{1} - \eta_{2}| 
    + 2 y_{2} |\eta_{1} - \eta_{2}|
    \right) \\ 
    & = \frac{C}{\rho^{2-\frac{1}{q}+o(1)}}
    \left(|\eta_{1} + \eta_{2}| + 2 y_{2} 
    \right) |\eta_{1} - \eta_{2}| \\
    & \leq \frac{C}{\rho^{2-\frac{1}{q}+o(1)}} y_{2}|\eta_{1} - \eta_{2}| \\
    & \leq C \rho^{-2+ o(1)} 
    |\eta_{1} - \eta_{2}|. 
    \end{split}
    \]
Next, we give the proof of \eqref{eq-y44}.
Note that $N_{2, 2}(\eta, \rho) = \frac{G(\rho)}{2\rho 
    g'(y_{1})} g''(y_{1}) 
    h_{1}(\eta, \rho) h_{2}(\eta, \rho)$. 
This together with \eqref{eq-y60}, 
\eqref{eq-y54}, 
\eqref{eq-y61} and \eqref{eq-y48}
yields that 
     \[
    \begin{split}
     |N_{2, 2}(\eta_{1}, \rho) - 
    N_{2, 2}(\eta_{2}, \rho)|
    & \leq 
    \frac{G(\rho)}{2 \rho 
    g'(y_{1})} g''(y_{1}) 
   |h_{1}(\eta_{1}, \rho)| 
   |h_{2}(\eta_{1}, \rho) - h_{2}(\eta_{2}, \rho)| 
    \\
    & \quad + \frac{G(\rho)}{
    2\rho g'(y_{1})} g''(y_{1}) 
    |h_{2}(\eta_{2}, \rho)| 
   |h_{1}(\eta_{1}, \rho) - h_{1}(\eta_{2}, \rho)| \\
   & \leq C \rho^{\frac{1}{q}-2+o(1)}
   \rho^{-\frac{2}{q} + o(1)} (\log \rho) 
   \rho^{\frac{1}{q} + o(1)} |\eta_{1} - \eta_{2}| \\
   & \quad+ C  \rho^{\frac{1}{q}-2+o(1)}
   \rho^{-1+ \e}  
   \rho^{-\frac{1}{q} + o(1)} (\log \rho) |\eta_{1} - \eta_{2}| \\
   & \leq C \rho^{-2+ o(1)}
   (\log \rho)|\eta_{1} - \eta_{2}|. 
    \end{split}
    \]

Next, we show \eqref{eq-y45}. 
By \eqref{eq-daiji-1}, \eqref{eq-y53} and \eqref{eq-y16}, we obtain 
    \[
    \begin{split}
    |N_{2, 3}(\eta_{1}, \rho) - 
    N_{2, 3}(\eta_{2}, \rho)|
    & = \frac{G(\rho)}{\rho g'(y_{1})} \biggl|e^{\frac{1}{2}g''(y_{1}) (\eta_{2}+ y_{2})^{2}} 
    (e^{\frac{1}{2}g''(y_{1}) \left((\eta_{1}+ y_{2})^{2} 
    - (\eta_{2}+ y_{2})^{2}\right)} - 1) 
    - \frac{1}{2}g''(y_{1}) \left((\eta_{1}+ y_{2})^{2} 
    - (\eta_{2}+ y_{2})^{2}\right) \biggl| \\
    & \leq \frac{G(\rho)}{\rho g'(y_{1})} \biggl|e^{\frac{1}{2}g''(y_{1})(\eta_{2}+ y_{2})^{2}} 
    - 1 \biggl| \biggl|
    e^{\frac{1}{2}g''(y_{1}) \left((\eta_{1}+ y_{2})^{2} 
    - (\eta_{2}+ y_{2})^{2}\right)} - 1\biggl| \\
    & \quad 
    + \frac{G(\rho)}{\rho g'(y_{1})} \biggl|e^{\frac{1}{2}g''(y_{1}) \left((\eta_{1}+ y_{2})^{2} 
    - (\eta_{2}+ y_{2})^{2}\right)} - 1
    - \frac{1}{2}g''(y_{1}) \left((\eta_{1}+ y_{2})^{2} 
    - (\eta_{2}+ y_{2})^{2}\right) \biggl| \\
    & \leq C \frac{G(\rho)}{\rho g'(y_{1})} |g''(y_{1})(\eta_{2}+ y_{2})^{2}|
    |g''(y_{1}) \left((\eta_{1}+ y_{2})^{2} 
    - (\eta_{2}+ y_{2})^{2}\right)| \\
    & \quad + C \frac{G(\rho)}{\rho g'(y_{1})} |g''(y_{1}) \left((\eta_{1}+ y_{2})^{2} 
    - (\eta_{2}+ y_{2})^{2} \right)|^{2} \\
    & \leq C \frac{G(\rho)}{\rho g'(y_{1})} g''(y_{1})^2 y_{2}^{2} 
    |2y_{2} + \eta_{1} + \eta_{2}||\eta_{1} - \eta_{2}| 
    + C \frac{G(\rho)}{\rho g'(y_{1})} (g''(y_{1}) (2y_{2} + \eta_{1} + \eta_{2}) 
    (\eta_{1} - \eta_{2}))^2.
    \end{split}
    \]
It follows from \eqref{eq-y60}, \eqref{eq-daiji-1}, \eqref{eq-y54},
\eqref{eq-y16} and \eqref{eq-y51} that 
    \[
    \begin{split}
    |N_{2, 3}(\eta_{1}, \rho) - 
    N_{2, 3}(\eta_{2}, \rho)|
    & \leq C \rho^{-2 + \frac{1}{q}+o(1)} 
    \rho^{\frac{2}{q} -1+o(1)}
    \rho^{- \frac{3}{q} + o(1)} |\eta_{1} - \eta_{2}| 
    \\
    & \quad + C \rho^{-1-\frac{1}{q}+o(1)} 
    \rho^{\frac{4}{q}-2+o(1)}
    \rho^{-\frac{2}{q}+o(1)}
    \rho^{-\frac{1}{q}-1+\e+o(1)}
    |\eta_{1} - \eta_{2}| \\
    & \leq C  \rho^{-3+o(1)} |\eta_{1} - \eta_{2}|. 
    \end{split}
    \]
    
Next, we will prove \eqref{eq-y97}. 
It follows from Lemmata \ref{lem3-2} and 
\eqref{eq-y51-2} that 
\begin{align*}
    N_{2,3}(0,\rho)&=\frac{G(\rho)}{\rho g'(y_1)}(e^{\frac{1}{2}g''(y_1)y_{2}^2}-1-\frac{1}{2}g''(y_1)y_{2}^2)\\
&\le C\frac{G(\rho)}{\rho g'(y_1)} g''(y_1)^2 y_{2}^4\\
&\le C \rho^{-\frac{1}{q}-1} 
\rho^{\frac{4}{q}-2} \rho^{- \frac{4}{q}+o(1)} 
\le C\rho^{-\frac{1}{q}-3+o(1)}.
\end{align*}

Observe that 
    \[
    N_{2, 4} (\eta, \rho) = 
    h_{2}(\eta, \rho) N_{2, 3}(\eta, \rho). 
    \]
We have by \eqref{eq-y51}, 
\eqref{eq-y48} and 
\eqref{eq-y97} that 
    \[
    \begin{split}
    |N_{2, 4} (\eta_{1}, \rho) - 
    N_{2, 4} (\eta_{2}, \rho)| 
    & \leq |h_{2}(\eta_{1}, \rho)| 
    |N_{2, 3} (\eta_{1}, \rho) - 
    N_{2, 3} (\eta_{2}, \rho)| 
    + |h_{2}(\eta_{1}, \rho) - 
    h_{2}(\eta_{2}, \rho)| |N_{2, 3}(\eta_{2}, \rho)| 
    \\
    & \leq C \rho^{-1+\e}\rho^{-3+o(1)} 
    |\eta_{1} - \eta_{2}| + 
    C \rho^{\frac{1}{q}+ o(1)}
    \rho^{-\frac{1}{q}-3 + o(1)}|\eta_{1} - \eta_{2}| \\
    & \leq C  \rho^{-3+o(1)}|\eta_{1}-\eta_{2}|. 
    \end{split}
    \]
Thus, we obtain \eqref{eq-y98}. 
    
We show \eqref{eq-y63}. 
It follows from \eqref{eq-y7} that 
    \[
    N_{1}(\eta, \rho) 
    = \frac{G(\rho)}{\rho g'(y_{1})} 
    (e^{g'(y_{1}) \eta} - 1 - g'(y_{1}) \eta). 
    \]
Then, by \eqref{eq-y64}, we obtain 
    \[
    \begin{split}
    |N_{1} (\eta_{1}, \rho) - 
    N_{1} (\eta_{2}, \rho)| 
    & \leq \frac{G(\rho)}{\rho g'(y_{1})} 
    |e^{g'(y_{1}) \eta_{1}} - e^{g'(y_{1}) \eta_{2}} 
    - g'(y_{1}) (\eta_{1} - \eta_{2})| \\
    & = \frac{G(\rho)}{\rho g'(y_{1})} 
    |e^{g'(y_{1}) \eta_{2}} (e^{g'(y_{1}) 
    (\eta_{1} - \eta_{2})} - 1) 
    - g'(y_{1}) (\eta_{1} - \eta_{2})| \\
    & \leq 
    \frac{G(\rho)}{\rho g'(y_{1})} 
    |e^{g'(y_{1}) \eta_{2}} - 1|
    |e^{g'(y_{1}) 
    (\eta_{1} - \eta_{2})} - 1| \\
    & \quad + \frac{G(\rho)}{\rho g'(y_{1})}
     |e^{g'(y_{1}) 
    (\eta_{1} - \eta_{2})} - 1 
    - g'(y_{1}) (\eta_{1} - \eta_{2})| \\
    & \leq C \frac{G(\rho)}{\rho g'(y_{1})} 
    |g'(y_{1})\eta_{2}||g'(y_{1}) 
    (\eta_{1} - \eta_{2})| 
    + C \frac{G(\rho)}{\rho g'(y_{1})} 
    |g'(y_{1})(\eta_{1} - \eta_{2})|^{2} \\
    & \leq C \frac{G(\rho)}{\rho} |g'(y_{1})|
    |\eta_{2}||\eta_{1} - \eta_{2}| 
    + C \frac{G(\rho)}{\rho} |g'(y_{1})| 
    |\eta_{1} - \eta_{2}|^{2} \\
    & \leq C\delta\rho^{- 2+ \e}
    |\eta_{1} - \eta_{2}|.
    \end{split}
    \]
Thus, we see that \eqref{eq-y63} holds. 
    
In addition, we recall from 
\eqref{eq-y58}, \eqref{eq-y7},
\eqref{eq-y66}, \eqref{eq-daiji-1} and \eqref{eq-y65} that
\begin{align*}
    N_{3}(\eta,\rho)=\frac{1}{4}e^{g'(y_1)(\eta+y_2)+\frac{1}{2}g''(y_1)(\eta+y_2)^2}(e^{h(\eta)}-1)=\frac{G(\rho)}{\rho g'(y_1)}(1+h_2(\eta,\rho))(1+h_3(\eta,\rho))
    (e^{h(\eta)}-1) 
\end{align*}
and
\begin{equation*}
|h(\eta)|\le |h(0)|+|h(\eta)-h(0)|\le C\rho^{-2+o(1)}+C\rho^{\frac{1}{q}-1+o(1)}|\eta|\le C\rho^{-2+\e+o(1)}. 
\end{equation*}
Therefore, we have
 \[\begin{split}
    |N_{3}(\eta_{1}, \rho) - N_{3}(\eta_{2}, \rho)| 
    & \leq \frac{G(\rho)}{\rho g'(y_1)}(1+h_3(\eta_{1},\rho))|(e^{h(\eta_{1},\rho)}-1)||h_2(\eta_{1},\rho)-h_2(\eta_{2},\rho)|  \\
    & \quad + \frac{G(\rho)}{\rho g'(y_1)}(1+h_2(\eta_{2},\rho))|(e^{h(\eta_{1},\rho)}-1)||h_3(\eta_{1},\rho)-h_3(\eta_{2},\rho)|\\
    & \quad + \frac{G(\rho)}{\rho g'(y_1)}(1+h_2(\eta_{2},\rho))(1+h_3(\eta_{2},\rho))e^{h(\eta_{2},\rho)}|e^{h(\eta_{1},\rho)-h(\eta_{2},\rho)}-1| \\
    &\le \frac{C|\eta_{1}-\eta_{2}|}{\rho g'(y_1)}(\rho^{-2+\e+o(1)}\rho^{\frac{1}{q}+o(1)}+\rho^{-2+\e+o(1)}\rho^{\frac{1}{q}-1+o(1)}+\rho^{\frac{1}{q}-1+o(1)})\\
    &\le\rho^{-2+o(1)}|\eta_{1}-\eta_{2}|. 
    \end{split}
    \]
This yields \eqref{eq-y46}. 
    
Finally,  by \eqref{eq-y66}, we obtain that
\begin{align*}
    |N_3(0,\rho)|&=\frac{1}{4}e^{g'(y_1)y_2 +\frac{1}{2}g''(y_1)y_{2}^2}|
    e^{h(0. \rho)}-1|\\
    &\le C\frac{G(\rho)}{\rho g'(y_1)}e^{\frac{1}{2}g''(y_1)y_{2}^2}|h(0, \rho)|\\
    &\le \frac{C}{\rho g'(y_1)}(1+\frac{1}{2}g''(y_1)y_{2}^2)|h(0, \rho)|\\
    &\le \frac{C}{\rho^{3+\frac{1}{q}+o(1)}}.
\end{align*}
This completes the proof. 
\end{proof}
\begin{remark}
\label{rem3-1}
\rm{We point out that by the proofs of Lemmata \ref{lem3-4}, \ref{lem3-5} and \ref{lem-nest}, we can also obtain the following. For any $\e_0>0$, there exists $\Lambda$ sufficiently large such that if
$g'(y_1)\rho^{1-\e}|\eta_{i}|\le \delta$ for $i=1,2$ and all $\rho\ge \Lambda$ with some $0<\e\le\frac{1}{2}$ and $\delta>0$, then, all estimates in Lemma \ref{lem-nest} are satisfied with $o(1)=\e_0$ for any $\rho\ge \Lambda$.}
\end{remark}

\begin{proof}[Proof of Proposition 
\ref{prop-es}]
Observe that \eqref{eq-y23} can be written 
by the following:
\begin{equation*}
\eta(\rho) = \mathcal{T}[\eta](\rho),
\end{equation*}
in which
\begin{equation*}
\mathcal{T}[\eta](\rho)
= \int_{\rho}^{\infty} 
a(\rho) a(\tau)
\sin ((b(\rho) - b(\tau)))
F(\tau, \eta) d \tau,
\end{equation*}
where
\begin{equation*}
F[\rho, \eta] = 
\frac{d^{2} y_2}{d \rho^2}(\rho)
+ \frac{G(\rho)}{2\rho g'(y_{1})} g''(y_{1}) y_{2}^{2}+\frac{1}{a(\rho)} \frac{d^2 a}{d \rho^2}(\rho)\eta 
+ N_{1}(\eta, \rho) 
+ \sum_{i = 1}^{4} N_{2, i} (\eta, \rho)
+ N_{3}(\eta,\rho).
\end{equation*}
Fix $0<\e\le \frac{1}{2}$. Let $\Lambda>0$ and $\delta>0$ be constants defined later. We define $\Sigma$ as a space of continuous functions on $[\Lambda, \infty)$ equipped with the following norm:
\begin{equation*}
\|\xi\| := \sup\left\{g'(y_1)|\rho|^{1- 
\e} |\eta(\rho)|
\colon \rho \geq \Lambda \right\}.
\end{equation*}
Then, we define the subspace
\begin{equation*}
\mathcal{X} = \left\{\xi \in \Sigma \colon \|\xi\| \leq \delta \right\}.
\end{equation*}
Then, it follows from \eqref{eq-y52}, \eqref{eq-y4} and Lemma \ref{lem3-1} that we obtain that
\begin{equation} \label{eq-contra-1}
\biggl| \int_{\rho}^{\infty}  
a(\rho) a(\tau)
\sin ((b(\rho) - b(\tau)))
\left(\frac{d^{2} y_2}{d \rho^2}(\tau)
+ \frac{G(\tau)}{\rho g'(y_{1})} g''(y_{1}) y_{2}^{2} \right) d\tau
\biggl|
\leq \frac{\delta}{3} g'(y_1(\rho))^{-1}\rho^{-1+\e}
\end{equation}
for any $\rho\ge \Lambda$ if $\Lambda$ is sufficiently large. Moreover, 
by Lemmata \ref{lem3-0}, \ref{lem-nest} and Remark \ref{rem3-1} we obtain
\begin{equation}
\label{nagai-1}
\biggl|\frac{1}{a(\rho)} \frac{d^2 a}{d \rho^2}(\rho)\eta 
+ N_{1}(\eta, \rho) 
+ \sum_{i = 1}^{4} N_{2, i} (\eta, \rho)
+ N_{3}(\eta,\rho)\biggl|
\leq C\delta \rho^{-2 + \e} |\eta| 
\leq C\delta^2 g'(y_1)^{-1}\rho^{-3+2\e}. 
\end{equation}
and
\begin{align}
\label{nagai-2}
\begin{split}
\biggl|\frac{1}{a(\rho)} \frac{d^2 a}{d \rho^2}(\rho)(\eta_1-\eta_2)\biggl|
+ \biggl|N_{1}(\eta_1, \rho)-N_{1}(\eta_1,\rho
)\biggl|+ \sum_{i = 1}^{4} \biggl|N_{2, i} (\eta_1, \rho)-N_{2, i} (\eta_2, \rho)\biggl|&+ \biggl|N_{3}(\eta_1,\rho)-N_{3}(\eta_2, \rho)\biggl|\\
&\leq C\delta\rho^{-2 + \e} |\eta_1-\eta_2|  
\end{split}
\end{align}
for any $\eta,\eta_1,\eta_2\in \mathcal{X}$ for any  $\rho\ge \Lambda$ if $\Lambda$ is sufficiently large.
We define $\Lambda$ so that the above three estimates are satisfied. 

We shall show that $\mathcal{T}$ maps $\mathcal{X}$
to itself. By \eqref{nagai-1}, we obtain
\begin{equation*}
\begin{split}
& \quad \biggl| \int_{\rho}^{\infty}
a(\rho) a(\tau)
\sin ((b(\rho) - b(\tau)))
\left(
\frac{1}{a(\rho)} \frac{d^2 a}{d \rho^2}(\rho)\eta + 
N_{1}(\eta, \tau) 
+ \sum_{i = 1}^{4} N_{2, i} (\eta, \tau)
+ N_{3}(\eta, \tau) \right)
d\tau \biggl|\\
& 
\leq C\delta^2 \int_{\rho}^{\infty} \rho^{\frac{1}{4}} \tau^{\frac{1}{4}}g'(y_1(\tau))^{-1}
\tau^{-3+ 2\e}  d\tau
\leq C\delta^2\rho^{\frac{1}{4}}g'(y_1(\rho))^{-1} \int_{\rho}^{\infty} \tau^{-\frac{11}{4}+ 2\e} d\tau
\leq C\delta^2 g'(y_1)^{-1}\rho^{-\frac{3}{2}+ 2 \e}.
\end{split}
\end{equation*}
for $\rho \geq \Lambda$. This together with \eqref{eq-contra-1} yields 
\begin{equation*}
|\mathcal{T}[\eta](\rho)| \leq \delta g'(y_1)^{-1} 
\rho^{-1+ \e}(\frac{1}{3}+C\delta \rho^{-1/2+\e})
\end{equation*}
for $\eta \in \mathcal{X}$. Thus, by choosing $\delta$ so that $C\delta<\frac{2}{3}$, we have $\mathcal{T}[\eta] \in \mathcal{X}$. 

Next, we will show that
$\mathcal{T}$ is a contraction mapping.
For $\eta_{1}, \eta_{2} \in \mathcal{X}$, we have
\begin{equation*}
\begin{split}
|\mathcal{T}[\eta_{1}](\rho) - 
\mathcal{T}[\eta_{2}](\rho)|
& \leq 
\int_{\rho}^{\infty}
|a(\rho) a(\tau) \sin (
(b(\rho) - b(\tau)))|
|
\frac{1}{a(\rho)} \frac{d^2 a}{d \rho^2}(\rho)
(\eta_{1}(\tau) - \eta_{2}(\tau))| d \tau \\
& \quad 
+ \int_{\rho}^{\infty}
|a(\rho) a(\tau) \sin (
(b(\rho) - b(\tau)))|
|N_{1}(\eta_{1}, \tau) - 
N_{2}(\eta_{2}, \tau)| d\tau \\
& \quad 
+ \sum_{i=1}^{4} \int_{\rho}^{\infty}
|a(\rho) a(\tau) \sin (
(b(\rho) - b(\tau)))|
|
N_{2, i} (\eta_{1}, \rho)
- N_{2, i} (\eta_{2}, \rho)| d\tau \\ 
& \quad + \int_{\rho}^{\infty}
|a(\rho) a(\tau) \sin (
(b(\rho) - b(\tau)))|
|N_{3}(\eta_{1}, \tau) - 
N_{3}(\eta_{2}, \tau)| d\tau.\\
\end{split}
\end{equation*}
It follows from \eqref{nagai-2} that
\begin{equation*}
\begin{split}
|\mathcal{T}[\eta_{1}](\rho) - \mathcal{T}[\eta_{2}](\rho)|
& \leq C\delta
\int_{\rho}^{\infty}
|a(\rho) a(\tau) \sin (
(b(\rho) - b(\tau)))|
\tau^{-2 + \e} |\eta_{1} - \eta_{2}| 
d\tau\\
& \leq
C \delta^2 \int_{\rho}^{\infty} \tau^{\frac{1}{4}} \rho^{\frac{1}{4}} g'(y_1(\tau))^{-1}
\tau^{-3+2\e}
\|\eta_{1} - \eta_{2}\| d \tau\\
& \leq
C\delta^2 g'(y_1(\rho))^{-1}\rho^{-\frac{3}{2}+2\e}
\|\eta_{1} - \eta_{2}\|
\end{split}
\end{equation*}
for any $\rho\ge \Lambda$.
This yields that
\begin{equation*}
\|\mathcal{T}[\eta_{1}](\rho) - \mathcal{T}[\eta_{2}](\rho)\|
\leq C\delta \rho^{-\frac{1}{2}+ \e} 
\|\eta_{1} - \eta_{2}\| 
< \frac{1}{2} \|\eta_{1} - \eta_{2}\|
\end{equation*}
for ant $\eta_1,\eta_2\in \mathcal{X}$
by choosing $\delta$ so that $C\delta<\frac{1}{2}$.
Thus, we find that $\mathcal{T}$ is a contraction mapping.
This completes the proof.    
\end{proof}
\begin{proof}[Proof of Corollary \ref{cor}]
By Proposition \ref{prop-es}, we have $y_{\infty} = 
y_1+y_2+o(g'(y_1)^{-1}\rho^{-1+\varepsilon})$. It follows from 
\eqref{eq-y4},  
\eqref{shape-y2} and 
\eqref{eq-y52-2} that 
\begin{align*}
    g(y_{\infty}) 
    & = g(y_1)+g'(y_1)(y_2+o(g'(y_1)^{-1}\rho^{-1+\varepsilon}))+\frac{1}{2}g''(\hat{y})(y_2 
    + o(g'(y_1)^{-1}\rho^{-1+\varepsilon}))^2\\
    & = \rho+(-2\log \rho + \log ((g/g')(y_{1}))+\log 4G)+J_0+o(\rho^{-1+\varepsilon})
\end{align*}
Here, 
$\hat{y} = y_{1} + 
\theta (y_{2} + o(g'(y_1)^{-1}\rho^{-1+\varepsilon}))$ for some 
$\theta \in (0, 1)$ and 
$2J_0=g''(\hat{y})(y_2+o(g'(y_1)^{-1}\rho^{-1+\varepsilon}))^2$. We remark that $o(g'(y_1)^{-1}\rho^{-1+\e})=o(|y_2|)$. Thus,
it follows from the mean value theorem, the assumption (G1),
Lemma \ref{lem2.7} and 
Lemma \ref{lem3-2} that 
    \begin{align*}
    |g''(\hat{y}) - g''(y_{1})| 
    \leq |g'''(\tau y_{1} + 
    (1 - \tau) 
    \hat{y})| |\hat{y} - y_{1}|
    & \leq C 
    \frac{g'^3}{g^2} 
    (\tau y_{1} + (1- \tau) 
    \hat{y}) |y_{2}| \\
    & \leq C g^{-2}(\tau y_{1} + (1- \tau) 
    \hat{y})g'^3(y_1)|y_2|\\
    &\le C (\log \rho)g(y_1)g''(y_1)\frac{g'(y_1)^2}{g''(y_1)g(y_1)} g^{-2}(\tau y_{1} + (1- \tau)) \\
    & \le C\frac{\log \rho}{\rho}g''(y_1) 
    \end{align*}
with some $0<\tau<1$. Here, the last inequality follows from the fact that
\begin{equation*}
   |g(y_1)-g(\tau y_1+(1-\tau)\hat{y})|\le Cg'(y_1)|y_2|\le o(g(y_1)).
\end{equation*}
    Then, using this,  
    \eqref{shape-y2},  
    the assumption (G1) and 
    \eqref{eq-y52-2}, we obtain 
\begin{equation*}
    J=g''(\hat{y})(y_2+o(g'(y_1)^{-1}\rho^{-1- \varepsilon}))^2\le 
    2g''(y_{1})(1+O(\log\rho/\rho))y^2_2\le \frac{C}{\rho}
    \frac{g''(y_{1})g(y_{1})}
    {g'^2(y_{1})}(\log \rho)^2 
    \le 
    \frac{C}{\rho^{1-\frac{\e}{2}}}.
\end{equation*}
Therefore, using $g(y_{1}) = \rho$ 
(i.e. $g^{-1}(\rho) = y_{1}$), we have
\begin{align*} 
    g(y_{\infty}) 
    = \rho-2\log\rho+\log [\frac{g}{g'}(g^{-1}(\rho))]+\log 4H(g^{-1}(\rho))+o(\rho^{-1+\varepsilon}).
\end{align*}   
\end{proof}

We are now in a position to prove 
Theorem \ref{thm-sing}.
\begin{proof}[Proof of Theorem \ref{thm-sing}]
Put $U_{\infty}(r) = y_{\infty}(\rho)$, where $y_{\infty}(\rho)$
is the solution to \eqref{ODE} obtained in 
Proposition \ref{prop-es}.
Then, $U_{\infty}$ satisfies the following:
\begin{equation}\label{eq-y67} 
- \frac{d^{2} U_{\infty}}{d r^{2}} - \frac{1}{r} \frac{d U_{\infty}}{d r}
= f(U_{\infty}) \hspace{4mm} \mbox{for
$r \in (0, R_{\infty})$},
\end{equation}
where $R_{\infty} = e^{- \Lambda_{\infty}/2}$.
Note that $U_{\infty}$ is a solution to the ordinary differential
equation \eqref{eq-y65}. 
Thus, as long as $U_{\infty}$ remains bounded, 
we can extend $U_{\infty}$ in the direction 
of $r$. We shall show $U_{\infty}(r)$ exists for all 
$r > 0$. To this end, we define 
    \[
    E(r):= \frac{1}{2} \left(\frac{d U_{\infty}}{d r}
    (r)\right)^2 + \int_{r_0}^{U_{\infty}(r)} f(s) ds, 
    \]
where $r_{0} \in (0, R_{\infty})$. 
Then, we have 
    \[
    \frac{d E}{d r}(r) 
    = \frac{d U_{\infty}}{d r}
    (r)\frac{d^2 U_{\infty}}{d r^2}
    (r) + f(U_{\infty}(r)) \frac{d U_{\infty}}{d r}
    (r)
    = - \frac{1}{r} 
    \left(\frac{d U_{\infty}}{d r}(r) \right)^{2} 
    \leq 0. 
    \]
Thus, $E(r)$ is decreasing in $r>0$. 
This together with the positivity of $f$ implies 
that 
    \begin{equation} \label{eq-y109}
    E(r_{0}) \geq E(r) 
    \geq \frac{1}{2} \left(\frac{d U_{\infty}}{d r}
    (r)\right)^2. 
    \end{equation}
Therefore, there exists a constant $C_{0}>0$ 
such that $|\frac{d U_{\infty}}{d r}
    (r)| \leq C_{0}$. 
This yields that 
    \[
    |U_{\infty}(r)| \leq |U_{\infty}(r_0)| 
    + \int_{r_{0}}^{r} 
    \biggl|\frac{d U_{\infty}}{d s}(s)\biggl| ds
    \leq |U_{\infty}(r_0)| + C_{0} r. 
    \]
From this, we see that $U_{\infty}(r)$ exists 
for all $r > 0$. 

Then, all we have to show for Theorem \ref{thm-sing}
is that $U_{\infty}$ 
must have a zero at some point by contradiction.
Suppose to the contrary that 
$U_{\infty}(r) >0$ for all $0 < r < \infty$.
This yields that $U_{\infty}$ is monotone decreasing.
Indeed, otherwise, $U_{\infty}$ has a local 
minimum point at $r = r_{*} \in (0,
\infty)$. This implies that
$\partial^{2}_{r} U_{\infty}(r_{*}) \geq 0$ and
$\partial_{r} U_{\infty}(r_{*}) = 0$.
Then, from the equation \eqref{eq-y67}, we obtain
\begin{equation*}
0 \leq \frac{d^{2} U_{\infty}}{d r^{2}}(r_{*})
= - f(U_{\infty} (r_{*})) 
= e^{g(U_{\infty}(r_{*}))}< 0,
\end{equation*}
which is absurd.
\par
Note that $U_{\infty}$ is positive and monotone decreasing for $r > 0$. 
Thus, there exists a constant $C_{\infty} \geq 0$
such that $U_{\infty}(r) \to C_{\infty}$
as $r \to \infty$. 
This together with \eqref{eq-y67} and 
\eqref{eq-y109} yields that 
there exists a sufficiently large $r_{1} > 0$ 
such that 
    \[
    \frac{d^{2} U_{\infty}}{d r^{2}}(r) 
    \leq - \frac{1}{2} f(C_{\infty})
    \]
for all $r \geq r_{1}$. 
It follows that 
    \[
    \frac{d U_{\infty}}{d r}
(r) \leq \frac{d U_{\infty}}{d r}
(r_0) - \frac{1}{2} f(C_{\infty}) (r - r_{0}), 
    \]
which contradicts \eqref{eq-y109}. 
Therefore, there exists $R_{\infty} > 0$ such that
$U_{\infty}(r)$ has a zero at $r = R_{\infty}$.
This completes the proof.
\end{proof}

\section{Properties of singular solutions to 
\eqref{ODE} and generalized Emden-type transformation} 
\label{sec-emden}
\begin{theorem} \label{thm4-1}
Assume that $1 \leq q < 2$. 
Let $\varepsilon>0$. Let $z_{\infty}$ be a singular solution 
to \eqref{ODE} and $y_{\infty}$ be the one obtained by Theorem \ref{thm-sing}. 
We put 
\begin{equation} \label{func-h}
    \mathcal{F}(s) 
    :=g(s) +2\log g(s) 
    -\log \left( \frac{g}{g'}(s) \right) 
    -\log (4 H(s)) 
    -\frac{C}{g^{1-\varepsilon}}(s).
    \end{equation}
for $C > 0$.
There exists $C_1 > 0$ depending 
only on $f$ such that 
one of the following holds.
\begin{itemize}
    \item[{(i)}] there exists a sequence $\rho_n\to\infty$ such that 
    $g(y_{\infty}(\rho_n)) \leq g(z_{\infty}(\rho_n))$,
    \item[{(ii)}] 
    $\mathcal{F}(z_{\infty}(\rho)) - 
    \rho > -C_{1}$ for all 
    $\rho > 0$ sufficiently large,
    \item[{(iii)}] there exists a sequence $\rho_n\to\infty$ such that $\mathcal{F}(z_{\infty}(\rho_n))>\rho_n\ge
    \mathcal{F}(y_{\infty}(\rho_n))$ . 
\end{itemize}
\end{theorem}

\begin{remark}\label{rem4-1}
\rm{
Since $\mathcal{F}$ is non-decreasing function 
(see Lemma \ref{lem4-3} below), 
This together with the 
monotonicity of $g$ (see Lemma 
\ref{lem2-2}) implies
$(iii)$ means $(i)$. 
Thus,  
$(i)$ or $(iii)$ means that
$y_{\infty}(\rho_n)\le z_{\infty}(\rho_n)$ for all $\rho_n\to\infty$. 
Moreover, by using 
Lemma \ref{lem4-3}, 
$\mathcal{F}(y_{\infty}(\rho_n))
\le \mathcal{F}(z_{\infty}(\rho_n))$ is satisfied in this case.}
\end{remark}

\subsection{Properties of 
singular solutions to \eqref{ODE}}

\begin{lemma}
\label{basiclem}
Assume that $g$ satisfies $(G1)$. 
Let $z_{\infty}$ be a singular solution to 
\eqref{ODE}. We have 
 $\lim_{\rho \to \infty} z_{\infty}'(\rho) = 0$, 
$z_{\infty}'\ge 0$ and $g(z_{\infty})\le \rho+\log 4$ and
\begin{equation*}
    \frac{d z_{\infty}}{d \rho}(\rho)=\int_{\rho}^{\infty}\frac{1}{4}e^{g(z_{\infty}(s))-s}\,ds.
\end{equation*}
\end{lemma}

\begin{proof}
We first show that 
    \begin{equation} \label{eq-y55}
    \lim_{\rho \to \infty} z_{\infty}'(\rho) = 0.
    \end{equation}
Observe from \eqref{ODE} 
that $z_{\infty}'(\rho)$ is a decreasing function 
of $\rho$, 
so that the limit of $z_{\infty}'(\rho)$ exists. 
Since $\lim_{\rho \to \infty} z_{\infty}(\rho) = \infty$, 
the limit $\lim_{\rho \to \infty} z_{\infty}'(\rho)$ must be non-negative. 
Suppose to the contrary that  
\eqref{eq-y55} does not hold and  
$z_{\infty}'(\rho) \to c$ as 
$\rho\to\infty$ for some $c > 0$.
Then, we have $z_{\infty}(\rho) 
\ge c_1\rho-C$ for some $c_1<c$. By Lemma \ref{lem3-2}, we have
$g(z_{\infty})>z_{\infty}^{1+\e_0}$ with some $\e_0>0$. 
This together with \eqref{ODE} implies that 
    \[
    - z_{\infty}''(\rho) \geq \frac{1}{4} e^{c_{1}^{1+\e_0} \rho^{1+\e_0} 
    - \rho - C} \to \infty 
    \hspace{4mm} \mbox{as $\rho \to \infty$}. 
    \]
This is absurd because $\lim_{\rho \to \infty} 
z_{\infty}'(\rho) = c$. Thus, \eqref{eq-y55} holds.  

It follows from \eqref{eq-y55} and 
the equation \eqref{ODE} that 
    \[
    z_{\infty}^{\prime}(\rho) = \int_{\rho}^{\infty} 
    \frac{1}{4} f(z_{\infty}(s)) e^{-s}\, ds > 0.
    \]
Namely, we find that $z_{\infty}$ is non-decreasing. 
By the above equation, 
$f(u) = e^{g(u)}$, 
$\lim_{\rho \to \infty} z_{\infty}^{\prime}
(\rho) = 0$, 
the monotonicity of $g$ and $z_{\infty}$, 
we have
\begin{equation} \label{eq-y68}
    z_{\infty}'(\rho)= \int_{\rho}^{\infty}\frac{1}{4}e^{g(z_{\infty}(s))-s}\,ds
    \ge \frac{e^{g(z_{\infty}(\rho))}}{4} 
    \int_{\rho}^{\infty} e^{-s} \,ds
    =  \frac{1}{4}e^{g(z_{\infty}(\rho))-\rho}.
\end{equation}

Finally, we show $g(z_{\infty})\le \rho+\log 4$.  
It follows from \eqref{eq-y68} that 
\begin{equation} \label{eq-y57}
\int_{z_{\infty}(\rho)}^{\infty} e^{-g(\tau)}\,d\tau = 
\int_{z_{\infty}(\rho)}^{z_{\infty}(\infty)}e^{-g(\tau)}\,d\tau
= \int_{\rho}^{\infty} 
e^{- g(z_{\infty}(s))} z_{\infty}'(s)\, ds 
\ge \frac{1}{4} \int_{\rho}^{\infty} e^{-s} \,ds
= \frac{e^{-\rho}}{4}.
\end{equation}
In addition, we put 
    \[
    \varphi(s) := e^{- g(s)} - 
    \int_{s}^{\infty} 
    e^{- g(\tau)}\, d \tau.  
    \]
Then, one has 
    \[
    \varphi'(s) = - e^{- g(s)}g'(s) 
    - e^{- g(s)} < 0. 
    \]    
This together with $\lim_{s \to \infty} \varphi(s) 
= 0$ yields that $\varphi(s) \geq 0$, so that 
    \begin{equation}\label{eq-y56}
    e^{- g(z_{\infty})} \geq 
    \int_{z_{\infty}(\rho)}^{\infty} e^{-g(\tau)}\,d\tau.
    \end{equation}
Thus, by \eqref{eq-y57} and \eqref{eq-y56}, 
we get 
    \[
    e^{-g(z_{\infty})} \geq \frac{e^{- \rho}}{4}, 
    \]
which implies that 
$g(z_{\infty})\le \rho+\log 4$. 
This completes the proof. 
\end{proof}
\begin{lemma}\label{lower}
Assume that $1 \leq q$. 
Let $z_{\infty}$ be a singular solution 
to \eqref{eq-expo}.
Then, one has  
\begin{equation*}
   \limsup_{\rho\to\infty} g(z_{\infty}(\rho)) - g(y_{\infty}(\rho))) \ge -4.
\end{equation*}
\end{lemma}

\begin{proof}
Assume by contradiction that 
\begin{equation*}
    g(z_{\infty}(\rho))- 
    g(y_{\infty}(\rho)) \le 
    -\frac{7}{2} \hspace{2mm}\text{for all $\rho>\rho_0$}. 
\end{equation*}
It follows from \eqref{eq-y24} and (G1) that
\begin{align}   \label{eq-y25-2}
\begin{split}
    \frac{1}{4}e^{g(z_{\infty}(\rho))-\rho}
    \le \frac{1}{4}e^{g(y_{\infty}(\rho))-\rho 
    - 7/2}
   &= \frac{1}{4} \exp\left[-2\log\rho +
    \log \left[\frac{g}{g'}(g^{-1}
    (\rho))\right]+\log 4 
    - \log q
    +o(1)-7/2\right]\\
    & \leq \frac{1}{qe^3\rho^2}\frac{g}{g'}(g^{-1}(\rho)).
  \end{split}
\end{align}

We derive a contradiction by dividing 
into (Case 1)\; $q = 1$ and 
(Case 2)\; $1<q$. 

We fist consider (Case 1)\; $q = 1$. 
Here, by Lemma \ref{basiclem}, we have 
$\frac{d z_{\infty}}{d \rho}(\rho) \ge 0$ and 
$\lim_{\rho \to \infty} 
\frac{d z_{\infty}}{d \rho}(\rho) = 0$. 
Thus, thanks to \eqref{eq-y25-2}, we have 
\begin{equation} \label{eq-y26}
    \frac{d z_{\infty}}{d \rho}(\rho) 
    = \int_{\rho}^{\infty}\frac{1}{4}e^{g(z_{\infty}(s))-s}\,ds
    \le e^{- 3}
    \int_{\rho}^{\infty}\frac{g}{g'}(g^{-1}(s))\frac{1}{s^2}\,ds.
\end{equation}
Here, we recall that $g(y_1(\rho)) 
= \rho$ (see \eqref{eq-y4}). Then, we see that $g'(y_1) 
\frac{d y_1}{d \rho}=1$. Thus, by integrating by parts, 
we obtain that
\begin{align*}
\int_{\rho}^{\infty}\frac{g}{g'}(g^{-1}(s))\frac{1}{s^2}\,ds&=[-(\frac{g}{g'})(y_1(s))s^{-1}]_{\rho}^{\infty}+\int_{\rho}^{\infty}(\frac{g}{g'})'(y_1)\frac{1}{g'(y_1)s}\,ds\\
&=\frac{g}{g'}(g^{-1}(\rho))\frac{1}{\rho}+\int_{\rho}^{\infty}(1-H(y_1))\frac{g}{g'}(y_1)s^{-2}\,ds. 
\end{align*}
Since it follows from 
(G1) with $q = 1$ that 
$H(y_1)-1=o(1)$, combining the above with \eqref{eq-y26}, we obtain that
\begin{equation*}\label{eq-y26-1}
 \frac{d z_{\infty}}{d \rho}(\rho)\le (e^{-3}+o(1))\frac{g}{g'}(g^{-1}(\rho))\frac{1}{\rho}\le e^{-2}
    \frac{1}{g'(g^{-1}(\rho))}
    = e^{-2}
    \frac{1}{g'(y_{1}(\rho))}. 
\end{equation*}
Therefore, it follows from \eqref{eq-y26} 
and $g'(y_1) 
\frac{d y_1}{d \rho}=1$ and 
$g(y_{1}(\rho)) = \rho$ that
\begin{equation*}
\frac{d}{d \rho}
(e^{z_{\infty}(\rho)-e^{-2}y_1(\rho)}) 
=e^{z_{\infty}(\rho)-e^{-2}y_1(\rho)}
(\frac{d z_{\infty}}{d \rho}(\rho) 
-e^{-2} \frac{d y_1}{d \rho} 
(\rho)) = e^{z_{\infty}(\rho)-e^{-2}y_1(\rho)}
\left(\frac{d z_{\infty}}{d \rho} 
(\rho) - e^{-2}\frac{1}{g'(y_1(\rho))} \right) \le 0.
\end{equation*}
Therefore, since $\lim_{\rho \to \infty} y_{1}
(\rho) = \infty$, 
we have 
$z_{\infty} (\rho) \le e^{-2}y_1(\rho) + 
C\le e^{-1}y_1(\rho)$ for 
sufficiently large $\rho > 0$. 
\normalsize
Thus, we have by \eqref{eq-y26}, 
\eqref{tec2} and $g(y_{1}(\rho)) = \rho$
that 
\begin{equation*}
    \frac{d z_{\infty}}{d \rho} (\rho) = \int_{\rho}^{\infty}\frac{1}{4}e^{g(z_{\infty}(s))-s}\,ds 
    \le \frac{1}{4}\int_{\rho}^{\infty}e^{g(e^{-1}y_1(s))-s}\,ds
    \leq \frac{1}{4}\int_{\rho}^{\infty}e^{e^{-1} g(y_1(s))-s}\,ds
    \le Ce^{-(1-e^{-1})\rho}. 
\end{equation*}
Thus, we have
\begin{equation*}
\sup_{\rho > \rho_0}z_{\infty}(\rho) 
\leq 
\sup_{\rho > \rho_0} \left\{
z_{\infty}(\rho_0)+\int_{\rho_0}^{\rho}Ce^{-(1-e^{-1})s} \right\} \le C
\end{equation*}
It is a contradiction because $z_{\infty}$ is 
a singular solution. 

Next, we consider (Case 2)\; $1<q$. Let $\e_0>0$ be sufficiently small.
Then, by Lemma \ref{basiclem}, we have $\frac{d z_{\infty}}{d\rho}=o(1)$. Hence, thanks to \eqref{eq-y25-2}, 
\eqref{eq-y13} and \eqref{eq-y29}, we have 
\begin{equation} \label{eq-y108}
\begin{split}
    \frac{d z_{\infty}}{d \rho}(\rho)
    = \int_{\rho}^{\infty} 
    \frac{1}{4} e^{g(z_{\infty}(s)) 
    - s} \, ds
    & \le \int_{\rho}^{\infty}
    \frac{g(y_1(s))}{g'(y_1(s))y_1(s)}
    \frac{y_1(s)}{qe^3 s^2}\,ds \\
    &
    = (1-\frac{1}{q}+o(1)) 
    \int_{\rho}^{\infty} 
    \frac{y_1(s)}{qe^3 s^2}\,ds
    \\
    & = (1-\frac{1}{q}+o(1)) 
    \int_{\rho}^{\infty} 
    \frac{y_1(s)}{g^{1 - 
    \frac{1}{q} + \e_0}(y_{1}(s))} 
    \frac{1}{qe^3
    g^{- 1 + 
    \frac{1}{q} - \e_0}(y_{1}(s))
    s^2}\,ds \\
    & \le (1-\frac{1}{q}+o(1))\frac{y_1(\rho)}{qe^3 g^{1-\frac{1}{q}+\varepsilon_0}(y_1(\rho))}
    \int_{\rho}^{\infty}\frac{1}{s^{1+\frac{1}{q}-\varepsilon_0}}\,ds\\
    &=\frac{(1-\frac{1}{q}+o(1))y_1(\rho)}{qe^3(\frac{1}{q}-\varepsilon_0) g^{1-\frac{1}{q}+\varepsilon_0}(y_1(\rho))\rho^{\frac{1}{q}-\varepsilon_0}}\\
    &=\frac{(1-\frac{1}{q}+o(1))y_1(\rho)}
    {e^3q(\frac{1}{q}-\varepsilon_0)g (y_{1}(\rho))} 
    \le  
    \frac{e^{-2}}{g'(y_1)}.
\end{split}
\end{equation}
Here, we have used \eqref{eq-y13} in 
the last equality. Thus, it follows 
from $g'(y_1) 
\frac{d y_1}{d \rho}=1$ and 
\eqref{eq-y108} that
\begin{equation*}
\frac{d}{d \rho} 
(e^{z_{\infty}-e^{-2}y_1}) = 
e^{z_{\infty}-e^{-2}y_1}
( \frac{d z_{\infty}}{d \rho} - 
e^{-2} \frac{d y_1}{d \rho}) 
\le e^{z_{\infty}-e^{-2}y_1} 
(\frac{d z_{\infty}}{d \rho} 
-e^{-2}\frac{1}{g'(y_1)})\le 0.
\end{equation*}
Therefore, we have $z\le e^{-2}y_1+C\le e^{-1}y_1$. 
Here, by using \eqref{tec2}. we have
\begin{equation*}
    z'(\rho)\le \frac{1}{4}\int_{\rho}^{\infty}e^{g(e^{-1}y_1(s))-s}\,ds\le Ce^{-(1-e^{-1})\rho}.
\end{equation*}
Thus, we have
\begin{equation*}
\sup_{\rho > \rho_0}z_{\infty}(\rho) \leq  
\sup_{\rho > \rho_0} \left\{
z_{\infty}(\rho_0)+\int_{\rho_0}^{\rho}Ce^{-(1-e^{-1})s} \right\} \le C.
\end{equation*}
It is a contradiction because $z_{\infty}$ is 
a singular solution. 
Thus, we obtain the desired result. 
\end{proof}
\subsection{Generalized Emden-type transformation}
\label{energysec}

\begin{proposition} \label{prop-alt}

Let
$z_{\infty}$ be a singular solution to
equation \eqref{ODE} and $y_{\infty}$ 
be the one obtained
by Theorem \ref{thm-sing}.   
Assume that $ \limsup_{\rho\to\infty} g(z_{\infty}(\rho)) - g(y_{\infty}(\rho))) 
< 1$. 
Then there exists a sequence
     $\{\rho_{n}\}$ with $\lim_{n \to 
\infty} \rho_n =\infty$ such that $\mathcal{F}(z_{\infty}(\rho_n))-\rho_n\ge -C$ for some $C>0$. 
\end{proposition}
In order to prove Proposition \ref{prop-alt}, 
we prepare the following lemma. 
\begin{lemma}
\label{lem4-5}
Let $\mathcal{F}$ be the function defined by \eqref{func-h}. 
Then, by taking $C > 0$ sufficiently large in \eqref{func-h}, 
we have 
    \begin{align}
    &\mathcal{F}(y_{\infty}(\rho))-\rho \leq 0 \hspace{4mm} 
    \mbox{for all $\rho$ sufficiently large}, 
     \label{eq-y34}\\ 
    & \mathcal{F}(y_{\infty}(\rho)) 
    = \rho+O(\rho^{-1+\varepsilon}) \hspace{4mm} 
    \mbox{as $\rho \to \infty$}.
    \label{eq-y69} 
    \end{align}
\end{lemma}
\begin{proof}
Observe from $g(y_{\infty}) 
= \rho \times (1 + \rho^{-1} 
(g(y_{\infty}) - \rho))$ and \eqref{eq-y24} that 
\begin{align} \label{eq-y35}
\log g(y_{\infty}) 
=\log \rho+\log (1+\rho^{-1}(g(y_{\infty}) 
-\rho))
= \log \rho + O(\rho^{-1}(g(y_{\infty}) - 
\rho))
=\log\rho+O(\frac{\log\rho}{\rho}).
\end{align}
Moreover, we see from our assumption (G1)  that 
\begin{equation} \label{eq-y35-1}
    (\frac{g}{g'}(s))' = 
    \frac{g'(s)^2 - 
    g(s)g''(s)}{g'(s)^2} 
    \to 1 - \frac{1}{q} 
    \hspace{4mm} \mbox{as $s \to \infty$}.   
\end{equation}
This together with the 
mean value theorem and $y_{\infty} 
= y_{1} + y_{2} 
+ o(\rho^{-1-\frac{1}{q}+ \e})$ yields that  
\begin{equation} \label{eq-y30-1}
    \frac{g}{g'}(y_{\infty}) = 
    \frac{g}{g'}(y_1) + 
    (\frac{g}{g'})'(\hat{y})(y_{\infty} 
    - y_1), 
\end{equation}
where $\hat{y} = \theta y_{\infty} 
+ (1 - \theta) y_{1} 
= y_{1} + \theta y_{2} + 
o(\rho^{- 1-\frac{1}{q}+\e})$ with some $\theta\in (0,1)$. 
Then, it follows from \eqref{eq-y35-1} and 
\eqref{eq-y30-1} that 
    \begin{equation} \label{eq-y30}
    \frac{g}{g'}(y_{\infty}) 
    = \frac{g}{g'}(y_1)+O(y_2).  
    \end{equation}
Moreover, it follows from $g(y_{1}) = \rho$, 
\eqref{eq-y77} and Lemma \ref{lem3-2} that
\begin{equation} \label{eq-y31}
\begin{split}
   \frac{g'(y_1)}{g(y_1)} y_2
    &= \frac{1}{g(y_1)} 
    (-2\log\rho - \log\frac{g'(y_{1})}{g(y_{1})} 
    +\log G(\rho) + \log 4)\\
    & = \frac{1}{\rho} 
    (-2\log\rho - \log\frac{g'(y_{1})}{g(y_{1})} 
    +\log G(\rho) + \log 4) \\
    & = O\left(\frac{\log \rho}{
    \rho} \right). 
\end{split}
\end{equation}
Therefore, by \eqref{eq-y30} and \eqref{eq-y31}, 
we have
\begin{align} \label{eq-y33}
\log \frac{g}{g'}(y_{\infty}) 
= \log \frac{g}{g'}(y_1) 
+ \log \left(1 + \frac{g'(y_1)}{g(y_1)} y_2 
\right)
= \log \frac{g}{g'}(y_1) + O(\frac{\log\rho}{\rho}). 
\end{align}
Moreover, we obtain by using a similar argument and the condition (G2) that
\begin{align*}
|H(y_{\infty})-H(y_1)|&\le |H'(\hat{y})||(y_{\infty}-y_1)|\le C 
[(\frac{g}{g'})(\hat{y})]^{-1}|y_2|\le C[(\frac{g}{g'})(y_1)+O(y_2)]^{-1}|y_2|\\
&\le C\frac{g'}{g}(y_1)[1+O(\rho^{-1}\log\rho)]^{-1}|y_2|\\
&\le O\left(\frac{\log \rho}{\rho} \right). 
\end{align*}
Hence it follows that
\begin{equation}
\label{eq-y33-1}
|\log H(y_{\infty})-\log H(y_{1})|=|\log (1+\frac{H(y_{\infty})-H(y_1)}{H(y_1)})|\le O\left(\frac{\log \rho}{\rho}\right). 
\end{equation}
Thus, it follows 
from \eqref{func-h}, 
\eqref{eq-y24}, \eqref{eq-y33}, \eqref{eq-y33-1} and 
\eqref{eq-y35} 
that for a large $C >0$, we obtain
\begin{equation} \label{eq-y34-1}
\begin{split}
    \mathcal{F}(y_{\infty}(\rho))-\rho
    & = 
    g (y_{\infty}(\rho)) + 
    2\log g(y_{\infty}(\rho))  
    -\log \frac{g}{g'}(y_{\infty}(\rho)) 
    - \log 4 H(y_{\infty}(\rho))
    -\frac{C}{g^{1- \varepsilon}(y_{\infty}(\rho))} 
    - \rho\\
    &=-2\log\rho+\log[\frac{g}{g'}(y_1)]+\log 4 H(y_1) 
    +2\log g(y_{\infty}(\rho)) \\
    &\hspace{6mm}-\log \frac{g}{g'}(y_{\infty}(\rho)) -\log 4H(y_{\infty})-
    \frac{C}{g^{1-\varepsilon}(y_{\infty}(\rho))}+o(\rho^{-1+\varepsilon})\\
    & = O(\frac{\log \rho}{\rho}) 
    - \frac{C}{g^{1-\varepsilon}(y_{\infty}(\rho))}+o(\rho^{-1+\varepsilon}) \\
    &\le  -\frac{C}{g^{1-\varepsilon}(y_{\infty}(\rho))} 
    + o(\rho^{-1+\varepsilon})\\
    &\le 0.
\end{split}
\end{equation}
This implies \eqref{eq-y34}. 
Moreover, we see 
from \eqref{eq-y34-1} and \eqref{eq-y24} 
that \eqref{eq-y69} holds.
\end{proof}
\begin{lemma}
\label{tec}
Let $z_{\infty}$ be a singular solution of \eqref{eq-me-o}. Then,
There exists $\rho_0>0$ such that if $g(z_{\infty}(\rho))-g(y_{\infty}(\rho))\le 5$ for some $\rho>\rho_0$, then we have 
$z_{\infty}(\rho)-y_{\infty}(\rho)\le C\frac{1}{g'(y_1)}$ with some $C>0$ independent of $\rho$. In addition to the above assumption, if $g(y_{\infty}(\rho))-g(z_{\infty}(\rho))\le 5$, then we have 
$|y_{\infty}(\rho)-z_{\infty}(\rho)|\le C\frac{1}{g'(y_1)}$.
\end{lemma}
\begin{proof}
We assume that $g(z_{\infty}(\rho))-g(y_{\infty}(\rho))\le 2$ for some $\rho$ sufficiently large. 
Since $g$ is convex (see Lemma \ref{lem2-2}), 
we obtain 
\begin{align}
  5 \ge g(z_{\infty}(\rho))-g(y_{\infty}(\rho))&=g'(y_{\infty}(\rho))(z_{\infty}(\rho)-y_{\infty}(\rho))+\frac{1}{2}g''(s)(z_{\infty}(\rho)-y_{\infty}(\rho))^2\notag\\
    &\ge g'(y_{1})(z_{\infty}-y_{\infty})\frac{g'(y_{\infty})}{g'(y_1)}\label{eq-imp-2}
\end{align}
with some $s=z_{\infty}(\rho)+\theta(y_{\infty}(\rho)-z_{\infty}(\rho))$. 
Here, by using \eqref{shape-y2}, we obtain
\begin{equation*}
    \frac{|g(y_1)-g(y_1+y_2)|}{g(y_1)}\le \frac{g'(y_1)|y_2|}{g(y_1)}= O(\frac{\log\rho}{\rho}).
\end{equation*}
Hence, it follows from (G1), \eqref{shape-y2}, $y_{\infty}=y_1+y_2+o(y_2) (< y_{1})$ and the convexity of $g$ that
\begin{equation}
\label{eq-imp-3}
    \frac{g'(y_{\infty})}{g'(y_1)}
    =  \frac{g'(y_{1}) 
    + g''(y_{1} + \tau y_{2} 
    + \tau o(y_{2})) (y_{2} + o(y_{2}))
    }{g'(y_1)}
    \ge 1-C\frac{g'^2(y_1)}{g(y_1+y_2)}|y_2|\ge 1-O(\rho^{-1}\log \rho),  
\end{equation}
where $\tau \in (0, 1)$. 
Combining the above with \eqref{eq-imp-2}, we deduce that
\begin{equation*}
z_{\infty}(\rho)-y_{\infty}(\rho)\le C\frac{1}{g'(y_1)}. 
\end{equation*}
In particular, $z_{\infty}\le y_1+y_2+o(y_2)$. Thus, by using the same argument with the proof of \eqref{eq-imp-3}, we have
\begin{equation*}
    \frac{g'(s)}{g'(y_1)}
    \ge 1-O(\frac{\log\rho}{\rho}) 
\end{equation*}
for any $s$ between $y_{\infty}(\rho)$ and $z_{\infty}(\rho)$. Therefore, if additionally $g(y_{\infty}(\rho))-g(z_{\infty}(\rho))\le 5$ is satisfied, we obtain 
\begin{align*}
  5 \ge g(y_{\infty}(\rho))-g(z_{\infty}(\rho))&=g'(s)(y_{\infty}(\rho)-z_{\infty}(\rho))\\
    &=g'(y_{1})(y_{\infty}-z_{\infty})\frac{g'(s)}{g'(y_1)}\\
    &\ge \frac{1}{2}g'(y_{1})(y_{\infty}-z_{\infty}),
\end{align*}
where $s=z_{\infty}(\rho)+\theta(y_{\infty}(\rho)-z_{\infty}(\rho))$ with some $\theta\in(0,1)$. Then, we can obtain the result.
\end{proof}

We now give the proof of 
Proposition \ref{prop-alt}. 
\begin{proof}[Proof of Proposition \ref{prop-alt}]

From our assumption $\limsup_{\rho \to \infty} 
    (g(z_{\infty}(\rho)) - 
    g(y_{\infty}(\rho))) \leq 1$ 
    and Lemma \ref{lower}, 
    there exists a 
sequence $\{\rho_{n}\}$ with 
$\lim_{n \to \infty} \rho_{n} = \infty$
such that 
\begin{equation} \label{eq-y36}
  2 \geq g(z_{\infty}(\rho_{n}))- 
  g(y_{\infty}(\rho_{n})) 
  \ge -5 \hspace{4mm}\text{for all 
  $n \in \N$}. 
\end{equation}
Then, we have by \eqref{eq-y24} that 
\begin{equation} \label{eq-y100}
    \log g(z_{\infty})\ge \log (g(y_{\infty})-5) 
    =\log g(y_{\infty}) + O(\rho^{-1}).
\end{equation}
In addition, thanks to Lemma \ref{tec}, 
one has 
\begin{equation} \label{eq-y70}
 |y_{\infty}-z_{\infty}| 
 \le \frac{C}{g'(y_1)}.
\end{equation}
Moreover, by \eqref{eq-y35-1}, \eqref{eq-y70}, $y_{\infty} = y_{1} + y_{2} + o(\rho^{-\frac{1}{q}-1+ \e})$ and \eqref{eq-y31}, using $g(y_{1}) = \rho$,
we obtain 
    \begin{equation} \label{eq-y99}
    \begin{split}
    \frac{g}{g'}(z_{\infty})
     =
     \frac{g}{g'}(y_1)+
     (\frac{g}{g'})'(\widetilde{y})(z_{\infty}-y_1)
     & =  \frac{g}{g'}(y_1)+
     (\frac{g}{g'})'(\widetilde{y})(z_{\infty}-y_{\infty}) 
     +  (\frac{g}{g'})'(\widetilde{y})
     (y_{\infty} - y_{1}) \\
     & =  \frac{g}{g'}(y_1)+
     O(\frac{1}{g'(y_{1})})
     +  O(y_{2}) \\
      & =  \frac{g}{g'}(y_1)(1+O(\frac{\log \rho}{g(y_1)})). 
    \end{split}
    \end{equation}
Finally, we have by 
\eqref{eq-y36}, 
\eqref{eq-y100}, 
 \eqref{eq-y99},  
\eqref{eq-y33-1}, 
\eqref{eq-y35}
and \eqref{eq-y24}
that 
    \[
    \begin{split}
     \mathcal{F}
(z_{\infty}(\rho_n))-\rho_n 
    & = 
    g (z_{\infty}(\rho_n)) 
    + 2\log g(z_{\infty}(\rho_n))  
    -\log \frac{g}{g'}(z_{\infty}(\rho_n)) 
    -\log 4H(z_{\infty}(\rho_n)) 
    -\frac{C}{g^{1-\varepsilon}(z_{\infty}(\rho_n))} - \rho_n 
    \\
    & \geq 
    g (y_{\infty}(\rho_n)) - 5 
    + 
    2\log g(y_{\infty}(\rho_n))  
    -\log \frac{g}{g'}(y_{1} (\rho_n)) 
    -\log 4H(y_{\infty}(\rho_n)) 
    -\frac{C}{g^{1-\varepsilon}(y_{\infty}(\rho_n))} - \rho_n + O(1)
    \\
    & \geq g (y_{\infty}(\rho_n)) - 5 
    + 
    2\log \rho_n  
    -\log \frac{g}{g'}(y_{1} (\rho_n)) 
    -\log 4H(y_{1}(\rho_n)) 
    -\frac{C}{g^{1-\varepsilon}
    (y_{1}(\rho_n))} - \rho_n +O(1)
    \\
    & \geq - C.
    \end{split}   
    \]
This completes the proof. 
\end{proof}
Next, we shall show the following: 
\begin{proposition}\label{prop-alt2}
    Let
$z_{\infty}$ be a singular solution to
equation \eqref{ODE} and $y_{\infty}$ 
be the one obtained
by Theorem \ref{thm-sing}. 
We put 
\begin{equation}\label{eq-w1}
    \eta(\rho):= \mathcal{F}(z_{\infty}(\rho)), 
    \hspace{4mm} 
    w(\rho) := \eta(\rho) -\rho,
    \end{equation}
Assume that $ \limsup_{\rho\to\infty} g(z_{\infty}(\rho)) - g(y_{\infty}(\rho))) \leq 1$ and $\limsup_{\rho \to \infty} w(\rho) 
\le 1$. 
Then, there exists a constant $C > 0$ such that 
$w(\rho) \geq - C$
for sufficiently large $\rho > 0$. 
\end{proposition}

We prepare the following lemma to prove 
Proposition \ref{prop-alt2}. 
\begin{lemma}\label{lem4-3}
  Let $\mathcal{F}(s)$ 
  be the function defined 
  by \eqref{func-h}. Then, 
  $\mathcal{F}$ is non-decreasing 
  and convex for sufficiently 
  large $s > 0$.  
  In addition, we have 
    \begin{equation}
    \label{eq-y105}
    \lim_{s \to \infty} 
    \frac{\mathcal{F}'(s)}{g'(s)} = 
     \lim_{s \to \infty} 
    \frac{\mathcal{F}''(s)}{g''(s)} 
    =
    1. 
    \end{equation}
\end{lemma}
\begin{proof}
We recall that
\begin{equation*}
    \mathcal{F}(s) 
    =g(s) +2\log g(s)-\log (g/g'(s)) - 
    \log (4H(s)) -\frac{C}{g^{1-\e}(s)}.
\end{equation*}
Then, it follows from the conditions 
(G1), (G2) and $g''(s) = H(s) 
(g'(s))^2/(g(s))$
that
\begin{equation} \label{eq-y101}
\begin{split}
    \mathcal{F}'(s) 
    & =g'(s)(1+\frac{2}{g(s)}) - 
    \frac{g'(s)(1-H(s))}{g(s)} 
    - \frac{H'(s)}{H(s)}
    +C(1-\e)g'(s) 
    g^{-2+\varepsilon}(s) \\
    & \geq g'(s)
    \left( 
    1+\frac{2}{g(s)}- 
    \frac{(1-H(s))}{g(s)}
    - \frac{C}{g(s)H(s)}
    + C (1-\e) 
    g^{-2+\varepsilon}(s)
    \right)
    \ge 0, 
\end{split}
\end{equation}
\begin{equation} \label{eq-y102}
\begin{split}
    \mathcal{F}''(s) 
    & = g''(s)(1+\frac{2}{g(s)}) - 
    \frac{2g'^2(s)}{g^2(s)}+\frac{g'(s)H'(s)}{g(s)}
    - \frac{(g''(s)g(s)-g'^2(s)) (1-H(s))}{g^2(s)} 
    \\
    & \hspace{4mm} 
    - \frac{H''(s)H(s) - (H'(s))^2}{H^{2}(s)}
    + C(1-\e) (g''(s)g^{-2+\e}(s)-(2-\e)g^{-3+\e}(s)g'^2(s)) \\
    & \geq 
    \frac{g'^2(s)}{q g(s)} 
    - 
    C \frac{g'^2(s)}{g^2(s)}
    \ge 0.
\end{split}
\end{equation}
We can easily see that \eqref{eq-y105} 
follows 
from the above. 
This completes the proof. 
\end{proof}
Using Lemma \ref{lem4-3}, we shall show 
Proposition \ref{prop-alt2}. 
\begin{proof}[Proof of Proposition 
\ref{prop-alt2}]
Since the proof is rather long, 
we divide the proof into three steps. 

\textbf{(Step 1).} 
We introduce an energy function $\mathcal{H}$ 
(see \eqref{eq-y75} below) and obtain some monotonicity of $\mathcal{H}$. 

We define 
\begin{equation} \label{eq-y110}
    L(s) :=\frac{g(s)g'(s)}
    {\mathcal{F}'(s)} 
    \exp\left[- \frac{C}{g^{1-\e}
    (s)}\right], 
\end{equation}
where $C > 0$ is the constant given in 
\eqref{func-h}. 
Then, by \eqref{eq-y101} and \eqref{eq-y102}, 
we obtain
\begin{equation} \label{eq-y111}
\begin{split}
L'(s) 
& = 
\exp\left[- \frac{C}{g^{1-\varepsilon}(s)}\right](\frac{(g(s) g''(s) + g'^2(s))\mathcal{F}'(s) - g(s)g'(s)\mathcal{F}''(s)}{\mathcal{F}'^2(s)}+\frac{C(1- \e)}
{g^{2-\varepsilon}(s)}\frac{g(s)(g'(s))^2}{\mathcal{F}'(s)}) \\
& \geq 
\frac{1}{\mathcal{F}'^2(s)}
\exp\left[- \frac{C}{g^{1-\varepsilon}(s)}\right]((g(s) g''(s) + g'^2(s))\mathcal{F}'(s) - g(s)g'(s)
\mathcal{F}''(s)) \\ 
& \geq \frac{1}{\mathcal{F}'^2(s)}
\exp\left[- \frac{C}{g^{1-\varepsilon}(s)}\right](g'^3(s) -C\frac{g'^3}{g}(s)) 
\ge 0,  
\end{split}
\end{equation}

\begin{equation*} \label{eq-y71}
    \frac{d^{2} \eta}{d \rho^2}-\frac{\mathcal{F}''(z_{\infty})}
    {\mathcal{F}'^2(z_{\infty})}
    \left(\frac{d \eta}{d \rho}\right)^2 +  
    \frac{\mathcal{F}'
    (z_{\infty})e^{g(z_{\infty})-\rho}}{4}=0. 
\end{equation*}
Here, we remark that
by \eqref{func-h}, we have 
    \[
    g(s) = 
    \mathcal{F}(s) 
    - \log g(s) - \log g'(s) 
    + \log (4 H(s))
    + C/g^{1-\varepsilon}(s), 
    \]
so that 
\begin{equation*} \label{eq-y72}
\frac{\mathcal{F}'(z_{\infty}) 
e^{g(z_{\infty}) - \rho}}{4}= 
\frac{\mathcal{F}'(z_{\infty})}
{g(z_{\infty})g'(z_{\infty})}
H(z_{\infty})
e^{C/g^{1-\e}(z_{\infty})}
e^{\mathcal{F}(z_{\infty}) 
-\rho} = 
\frac{H(z_{\infty}) 
e^{\mathcal{F}(z_{\infty}) - \rho}}{L(z_{\infty})}.   
\end{equation*}
Observe that      
$w$ defined by \eqref{eq-w1} satisfies
\begin{equation} \label{eq-y104}
 L(z_{\infty}) 
 \frac{d^{2} w}{d \rho^2} -\frac{\mathcal{F}''(z_{\infty}) 
 L(z_{\infty})}{
 (\mathcal{F}' (z_{\infty}))^2} 
 (\frac{d w}{d \rho} +1)^2 + 
 H(z_{\infty}) e^{w} =0.   
\end{equation}
Here, we set 
    \begin{equation*}\label{eq-W2}
    W(t)=w(\rho)
    \end{equation*}
with $d\rho=L^{1/2} (z_{\infty}(\rho))dt$.
Then, 
since 
    \[
    \frac{d w}{d \rho} = 
\frac{d \eta}{d \rho} - 1 
= \mathcal{F}'(z_{\infty}(\rho)) 
\frac{d z_{\infty}}{d \rho} - 1 
    \]
we have 
\begin{equation*}
    \frac{d W}{d t} = \frac{d w}{d \rho} 
    L^{1/2}(z_{\infty}), 
    \hspace{4mm} 
    \frac{d^2 W}{d t^{2}} = 
    L (z_{\infty}) \frac{d^2 w}{d \rho^2} 
    +\frac{1}{2} 
    L'(z_{\infty}) \frac{d w}{d \rho} \frac{d z_{\infty}}{d \rho}
    = L(z_{\infty}) \frac{d^2 w}{d \rho^2}  
    +\frac{L'(z_{\infty})}{2\mathcal{F}'(z_{\infty})}
    \frac{d w}{d \rho}(\frac{d w}{d \rho} + 1). 
\end{equation*}
Thus, we deduce from \eqref{eq-y104} that
\begin{equation*}
\frac{d^2 W}{d t^2} - 
\frac{L'(z_{\infty})}
{2\mathcal{F}'(z_{\infty})} 
(\frac{d w}{d \rho}+1) 
\frac{d w}{d \rho} -
\frac{\mathcal{F}''
(z_{\infty}) L(z_{\infty})}{
(\mathcal{F}'(z_{\infty}))^2}
\left(\frac{d w}{d \rho} + 1\right)^2
+ H(z_{\infty}) e^{W} = 0. 
\end{equation*} 
The above equation can be written by 
    \begin{equation}\label{eq-W}
    \frac{d^2 W}{d t^2} - 
\left[\frac{L'(z_{\infty})}
{2\mathcal{F}'(z_{\infty})} 
(\frac{d w}{d \rho}+1)+
\frac{\mathcal{F}''(z_{\infty}) 
L(z_{\infty})}
{(\mathcal{F}'(z_{\infty}))^2}
\left(\frac{d w}{d \rho} + 2\right)
\right]\frac{d w}{d \rho} + 
H(z_{\infty}) e^{W} - 
\frac{\mathcal{F}''(z_{\infty}) 
L(z_{\infty})}
{(\mathcal{F}'(z_{\infty}))^2} = 0. 
    \end{equation}
By the assumption $\limsup_{\rho \to \infty} 
(g(z_{\infty}) - g(y_{\infty})) < 1$, \eqref{shape-y2} and Lemma \ref{tec}, we have $|z_{\infty}-y_{\infty}|\le \frac{C}{g'(y_1)}$ and $z_{\infty}\le y_1$. Hence, 
using this, \eqref{eq-y34} and the convexity of $\mathcal{F}$ (see Lemma 
\ref{lem4-3}), we have
    \[
    W(t) = 
    \mathcal{F}(z_{\infty}(\rho)) - 
    \rho
    \leq 
    \mathcal{F}'(y_\infty(\rho)) 
    (z_{\infty}(\rho)- 
    y_{\infty}(\rho)) +C
    \le  C \hspace{4mm} 
    \mbox{for all $\rho$ sufficiently large}.  
    \]
This yields that 
    \[
    e^{W} \le  C 
    \hspace{4mm} 
    \mbox{for all $\rho$ sufficiently large}.
    \]
Thus, we have
\begin{equation*}
\left(\frac{1}{q}-H (z_{\infty})
\right)e^W=o(1) 
 \hspace{4mm} \mbox{as $\rho \to \infty$}
\end{equation*}
and by \eqref{eq-y110} and 
Lemma \ref{lem4-3} that 
\begin{equation*}
\begin{split}
    \frac{L (z_{\infty})
    \mathcal{F}''(z_{\infty})}
    {\mathcal{F}'^2(z_{\infty})}-\frac{1}{q}
    & = \frac{g(z_{\infty})g'(z_{\infty})
    \mathcal{F}''(z_{\infty})}{\mathcal{F}'(z_{\infty})^3} 
    e^{- C/g^{1-\e}(z_{\infty})} -
    \frac{1}{q} + o(1) \\
    & = 
    \frac{g''(z_{\infty})g(z_{\infty})}{g'^2(z_{\infty})}
     - \frac{1}{q} + o(1) \to 0 
    \hspace{4mm} \mbox{as $\rho \to \infty$}. 
\end{split}
\end{equation*}
Moreover, by Lemma \ref{lem4-3} and 
\eqref{eq-y111}, we can verify that
\begin{equation} \label{eq-y112} 
    \frac{L(z_{\infty})
    \mathcal{F}''(z_{\infty})}{\mathcal{F}'^2(z_{\infty})} 
    \ge 0,\hspace{4mm}
    \frac{L'(z_{\infty})}{2\mathcal{F}'(z_{\infty})} 
    \ge 0.
\end{equation}
Let $\e_0>0$ be a small constant. We put 
\begin{equation} \label{eq-y75}
    \mathcal{H}(\rho) := 
    \frac{1}{2}\left(\frac{d W}{d t}(t)\right)^2 
    +\frac{1}{q} e^{W(t)} - \frac{1}{q}(1-\varepsilon_0)W(t). 
\end{equation}
Then, if $\frac{d w}{d \rho}\ge 0$, 
we see from \eqref{eq-W} and 
\eqref{eq-y112} that 
    \begin{equation} \label{eq-y73}
    \begin{split}
    \frac{d \mathcal{H}}{d \rho} 
 & = \frac{d t}{d \rho} 
    \left\{\frac{d W}{d t} \frac{d^2 W}{d t^2} 
    + \frac{1}{q} e^{W} \frac{d W}{d t} 
    - \frac{1}{q}(1 - \e_0) 
    \frac{d W}{d t} 
    \right\} \\
    & =\frac{d t} {d \rho} 
    \frac{d W}{d t}
    \left\{ 
    \frac{d^2 W}{d t^2} 
    + qe^{W} 
    - \frac{1}{q}(1 - \e_0)  
    \right\} \\
    & =
    \frac{d w}{d \rho}\Big\{
    \left[\frac{L'(z_{\infty})}
{2\mathcal{F}'(z_{\infty})} 
(\frac{d w}{d \rho}+1)+
\frac{
L(z_{\infty})
\mathcal{F}''(z_{\infty})}
{(\mathcal{F}'(z_{\infty}))^2}
\left(\frac{d w}{d \rho} + 2\right)
\right]\frac{d w}{d \rho} \\
& \hspace{4mm}
+ (\frac{1}{q}-   H(z_{\infty}))e^{W} 
- \frac{1}{q}(1 - \e_0) + \frac{L(z_{\infty}) 
\mathcal{F}''(z_{\infty})}
{(\mathcal{F}'(z_{\infty}))^2}
\Big\} \\
  &\ge  \frac{d w}{d \rho}
((\frac{1}{q}-   H(z_{\infty}))e^{W} 
- \frac{1}{q}(1 - \e_0) + \frac{ L(z_{\infty}) \mathcal{F}''(z_{\infty})}
{(\mathcal{F}'(z_{\infty}))^2}) 
\geq \frac{\e_0}{2 q} \frac{d w}{d \rho}
\ge 0.
    \end{split}    
    \end{equation}

\textbf{(Step 2).}
We consider the case that the sign of $w'$ changes infinitely many times and obtain 
the desired result.
In this case, $w$ is non-decreasing in $[\rho^{-}_n, \rho^{+}_n]$, where $\rho^{\pm}_{n}$ are local max/min of $w$ so that $\rho^{\pm}_{n}<\rho^{\pm}_{n+1}$ for any $n\in\mathbb{N}$, $\rho^{\pm}_1$ are sufficiently large and $\rho^{\pm}_{n}\to\infty$ as $n\to\infty$.
We set $-\delta:=\log (1-\varepsilon_0)$. Then, 
we claim that $-\delta\le w(\rho^{+}_{n})$. 
Suppose to the contrary that $-\delta >  w(\rho^{+}_{n})$. 
Setting $\mathcal{P}(x) = e^{x} - (1 - \e_0)x$,   
we can easily verify that 
$\mathcal{H}(\rho) = 
\frac{1}{2} \left(d W/d t (t)\right)^2 
+ \frac{1}{q} \mathcal{P}(W(t))$, 
$\mathcal{P}(-\delta) 
= \min_{x \in \R} \mathcal{P}(x)$ and 
$\mathcal{P}'(x) < 0$ for $x < - \delta$. 
This together with our hypothesis 
$-\delta >  w(\rho^{+}_{n})$,  
$\rho^{-}_n <\rho^{+}_n$ and 
the monotonicity of $w$ 
in $[\rho^{-}_n, \rho^+_n]$
implies that 
    \begin{equation} \label{eq-y74}
    \mathcal{P}(w(\rho^{-}_{n})) > 
\mathcal{P}(w(\rho^{+}_{n})).
    \end{equation}
On the other hand, since 
$\frac{d w}{d \rho} \geq 0$ 
in $[\rho^{-}_n, \rho^{+}_n]$, 
we have by \eqref{eq-y73} that 
$\mathcal{H}(\rho^{-}_{n}) \leq 
\mathcal{H}(\rho^{+}_{n})$. 
Observe that 
$\frac{d w}{d \rho}(\rho^{\pm}_{n}) = 0$ 
because $\rho^{\pm}_{n}$ are local max/min points of $w$. This implies 
that $\mathcal{H}(\rho^{\pm}_{n}) 
= \mathcal{P}(\rho^{\pm}_n)$, so that 
$\mathcal{P}(\rho^{-}_{n}) \leq 
\mathcal{P}(\rho^{+}_{n})$ which contradicts 
\eqref{eq-y74}. Thus, our claim holds.

Then, since 
$w(\rho^{+}_n)\le 2$~\footnote{recall that we assume that 
 $\limsup_{\rho \to \infty} w(\rho) 
\le 1$.} and 
$\rho^{+}_n \geq - \delta$ for all $n \in \N$, we have by \eqref{eq-y73} that 
\begin{equation*}
e^{w(\rho^{-}_{n})}-(1-\varepsilon_0)w(\rho^{-}_{n}) = \mathcal{P}(w(\rho^{-}_{n})) 
\leq \mathcal{P}(w(\rho^{+}_{n}))
\le \mathcal{P}(2).
\end{equation*}
Therefore, we get 
$w(\rho) \ge 
- \frac{\mathcal{P}(2)} 
{1-\varepsilon}$ for all
$\rho > 0$ sufficiently large.

\textbf{(Step 3).}
Finally, we consider the case $w'\ge 0$ or $w'\le 0$ in $[\rho_{*}, \infty)$ for 
some $\rho_{*} > 0$. 
For the first case, we deduce that $\lim_{\rho\to\infty} w(\rho) 
\ge w(\rho_{*})$. Secondly, we consider the case $w'\le 0$. 
Observe from Proposition \ref{prop-alt} that 
there exists a sequence $\{\rho_{n}\}$ 
with $\lim_{n \to \infty} w(\rho_{n}) 
= \lim_{n \to \infty} 
(\mathcal{F}(z_{\infty}(\rho_{n})) 
- \rho_n)
\geq - C$ for some $C > 0$. 
This together with the monotonicity of 
$w$ that
    \[
    \inf_{\rho > \rho_{*}} w(\rho) 
    = \lim_{n \to \infty} w(\rho_{n}) 
    \geq - C.
    \]
Thus, in both cases, 
we have $w >-C$ in $[\rho_{*}, \infty)$.
This completes the proof. 
\end{proof}

We are now in position to prove 
Theorem \ref{thm4-1}. 
\begin{proof}[Proof of Theorem \ref{thm4-1}]
Clearly, we see that 
one of the following three cases occurs:
\begin{enumerate}
    \item[(Case 1).]
    $\limsup_{\rho\to\infty} g(z_{\infty}(\rho)) - g(y_{\infty}(\rho))) \leq 1$ and $\limsup_{\rho \to \infty} w(\rho) 
\le 1$. 
     \item[(Case 2).]
$ \limsup_{\rho\to\infty} g(z_{\infty}(\rho)) - g(y_{\infty}(\rho))) > 1$.
    \item[(Case 3).]
    $\limsup_{\rho \to \infty} w(\rho) 
> 1$.
\end{enumerate}
First, we consider (Case 1). 
Then, by  Proposition \ref{prop-alt2}, 
we see that \textrm{(ii)} holds. 
Next, suppose that (Case 2) occurs. 
Then, we see that \textrm{(i)} holds. 
Finally, if (Case 3) occurs, 
we see from  
\eqref{eq-y34} and 
\eqref{eq-w1} that \textrm{(iii)} holds. 
This completes the proof. 
\end{proof}

\section{Intersection result and proof 
of Theorem \ref{thm-bi}}
\label{sec-bif}
In this section, we obtain an intersection result and 
give a proof of Theorem \ref{thm-bi}. 
We first quote the following.
\begin{proposition}
\label{critical prop}
Let $R\ge 1$. Then, for any $\varphi\in C^{0,1}_{0}(B_1)$, we have
\begin{equation*}
    \frac{1}{4}\int_{B_1}\frac{\varphi^2}{|x|^2(\log(R/|x|))^2}\,dx \le \int_{B_1}|\nabla \varphi|^2\,dx.
\end{equation*}
Moreover, for any $\varepsilon>0$, there exist  sequences $\{\varphi_n\}_{n\in \mathbb{N}}\subset C^{0,1}_{0}(B_1)$ and $r_i \downarrow 0$ such that
$\mathrm{supp}(\varphi_i)=[r_{i+1},r_{i}]$ for all $i\in \mathbb{N}$ and
\begin{equation}
\label{reversecritical}
    \frac{1+\varepsilon}{4}\int_{B_1}\frac{\varphi^2_{i}}{|x|^2(\log(R/|x|))^2}\,dx > \int_{B_1}|\nabla \varphi_{i}|^2\,dx \hspace{4mm} \text{for any $i\in \mathbb{N}$}.
\end{equation}
\end{proposition}
This proposition is 
obtained by the proof of \cite[Proposition 
2.5]{kuma2025-2}.

Next, we introduce a useful intersection result.
We say that the solution $u$ of \eqref{singulareq-expo} is  \textit{stable} if
    \[ 
    \int_{\R^2} f'(u) \xi^2 dx 
    \leq \int_{\R^2} |\nabla \xi|^2 dx 
    \hspace{4mm} \mbox{for all $\xi \in 
    C_{0}^{\infty}(\R^2)$}. 
    \]
Otherwise, we say that $u$ is unstable. 
Then, we show the following
\begin{lemma}
\label{unstable-singular}
Assume that $g$ satisfies (G1) and (G2) with some $q\ge 1$.
Let $V_{\infty}\in C^2(0, \infty)$ be a singular solution of \eqref{singulareq-expo} satisfying $\mathcal{F}(V_{\infty})+2\log r>-C_0$ with some $C_0>0$ for any $r>0$ sufficiently small. Then, we obtain 
\begin{equation*}
f'(V_{\infty})\ge \frac{2+o(1)}{qe^{C_0}r^2 (-\log r)}
\end{equation*}
as $r\to 0$. In particular, $V_{\infty}$ is unstable in $B_r$ for any $r>0$.
\end{lemma}
\begin{remark}
\label{rem5-1}
    \rm{From Lemma \ref{lem4-5} and Lemma \ref{unstable-singular}, we see that $U_{\infty}$ also satisfies the above assertions.}
\end{remark}

\begin{proof}[Proof of Lemma \ref{unstable-singular}.]
By the definition of $\mathcal{F}$ 
(see \eqref{func-h}) 
and Lemma \ref{basiclem}, we obtain 
    \begin{align*}
f'(V_\infty)&= e^{g(V_\infty)}
    g'(V_\infty)\\
    &= \exp\left\{
    \mathcal{F}(V_\infty) 
    - \log g(V_\infty) 
    - \log g'
    (V_\infty) 
    + \log (4 H(V_{\infty})) + 
    \frac{C}{g^{1-\e}
    (V_\infty)}\right\}
    g'(V_\infty)\\
    & \ge 
    4 H(V_{\infty})
    \frac{\exp\{\mathcal{F}(V_\infty)\}}
    {g(V_{\infty})}
    \exp\left\{\frac{C}{g^{1-\e}
    (V_\infty)} \right\}
    \\
    &\ge (\frac{4}{q}+o(1))\frac{1}{(-2\log r)}\exp\{\mathcal{F}(V_\infty)\} 
    \\&\ge (\frac{2}{q}+o(1))\frac{1}{e^{C_0} r^2(-\log r)}.
\end{align*}
Moreover, by Proposition \ref{critical prop}, we see that $V_{\infty}$ is unstable in $B_{r}$
for any $r>0$.
\end{proof}

\begin{lemma}
\label{unstablelem}
Let $U$ and $V$ be a regular or 
singular solution to 
\eqref{singulareq-expo}. We assume that $U$ is unstable in $B_{r_2}\setminus B_{r_1}$. Then, $V<U$ in $B_{r_2}\setminus B_{r_1}$ or $U(r)=V(r)$ for some $r\in [r_1,r_2]$. In particular, if $U$ and $V$ are \textit{unstable} in $[r_1,r_2]$, the latter case occurs.
\end{lemma}

\begin{proof}
We assume by contradiction that $W:=V-U>0$ in $B_{r_2}\setminus B_{r_1}$. 
Then, by the convexity of $f$ (see 
Lemma \ref{lem2-2}), we have
\begin{equation*}
    -\Delta W=f(V)-f(U)\ge f'(U)W.
\end{equation*}
Let $\phi\in C^{\infty}_{c}
(B_{r_2}\setminus B_{r_1})$. 
Then, we have by the Young inequality 
that 
\begin{align*}
    \int_{B_{r_2}\setminus B_{r_1}}f'(U)\phi^2\,dx&\le   \int_{B_{r_2}\setminus B_{r_1}} \nabla W\cdot \nabla(\phi^2/W)\,dx\\
    &=\int_{B_{r_2}\setminus B_{r_1}}-\frac{|\nabla W|^2}{W^2}\phi^2+2\frac{\phi}{W}\nabla\phi\cdot\nabla W\,dx\\
    &\le \int_{B_{r_2}\setminus B_{r_1}} |\nabla \phi|^2\,dx.
\end{align*}
Thus, we deduce that $U$ is stable, which is a contradiction.
\end{proof}

\begin{proof}[Proof of Theorem \ref{thm-bi}]
We define $U_{\infty}(r)
:=y_{\infty}(\rho)$, where 
$y_{\infty}$ is the solution to 
\eqref{ODE} obtained by Proposition 
\ref{prop-es}. Let $v(r,\alpha)=u(\lambda(\alpha)^{-1/2}r,\alpha)$. Then, $v(r,\alpha)$ is a solution of \eqref{singulareq-expo} such that $v(0,\alpha)=\alpha$. Moreover,
by arguing exactly the same with that of the proof of \cite[Lemma 5]{Mi2015}, 
we deduce that 
in order to prove Theorem \ref{thm-bi}, 
it is sufficient to obtain 
the following: there exists a sequence 
$\{\alpha_m\}$ 
such that $\alpha_m\to\infty$ and $v_m(r):=v(r,\alpha_m)$ satisfies $Z_{[0,r_0]}(U_\infty-v_m)\to\infty$ as $m\to\infty$ for any fixed $r_0$, where
\[Z_{[0,r_0]}(U_\infty-v_m) := 
    \left\{
    r \in (0, r_{0}) \colon 
    (U_\infty - v_{m})(r) = 0
    \right\}. 
    \]
We remark that 
\cite[Theorem (i), (ii)]
{kuma2025} can be proved by only imposing $f\in C^1[0,\infty)$, $f>0$ in $[0,\infty)$, the convexity and 
monotonicity of $f$ for large $u$ and 
\eqref{supercritical}. 
Thus, thanks to Lemma \ref{lem2.6}, there exists $\alpha_m \to \infty$ such that $v_m=v(r,\alpha_m)\to V_{\infty}$ in $C^{2}_{\mathrm{loc}}(0,\infty)$ as $m\to\infty$ to some singular solution $V_{\infty}$~\footnote{
Here, 
we essentially use the condition 
$1 \leq q < 2$. 
Indeed,  the sequence 
$\{v_{m}\}$ does not converge 
to any singular solution in general 
if $q > 2$ (see 
\cite[Lemma 1]{MR823114} or \cite[Theorem 1.1]{MR1319012}).}. 

Now, we fix $r_0>0$ and prove that 
\begin{equation} \label{eq-y106}
Z_{[0,r_0]}(U_\infty-v_m)\to\infty
\hspace{4mm} \mbox{as $m\to\infty$}. 
\end{equation}
Thanks to Theorem \ref{thm4-1} and Remark \ref{rem4-1}, 
it is only possible for case $(i)$ or 
case $(ii)$ in Theorem \ref{thm4-1} to occur.

We first consider the case $(i)$ 
in Theorem \ref{thm4-1}
and $V_\infty \not\equiv U_\infty$. 
Then, thanks to Remark \ref{rem4-1} and Proposition \ref{critical prop},
we can choose $I_{n}=[r_{i(n)+1},r_{i(n)}]$ with $i(n)\to\infty$
such that $V_\infty(r_n)>U_\infty(r_n)$ for some $r_n\in I_n$ \footnote{Indeed, if 
$V_\infty(r_n)=U_\infty(r_n)$, 
we can deduce that 
$V_\infty'(r_n)\neq U_\infty'(r_n)$ 
or $V_\infty=U_\infty$ 
by the classical ODE argument.} and \eqref{reversecritical} is satisfied for some
$\varphi_n:=\varphi_{i_n}\in C^{0,1}_{0}(I_n)$. Let $M>0$. Then, there exists some $m_0$ such that 
$v_m(r_n)>U_\infty(r_n)$ for all $n\le M$ and $m>m_0$. Moreover, 
observe from Lemma \ref{unstable-singular} and Remark \ref{rem5-1} that
\begin{align*}
    \int_{I_n}f'(U_\infty)
    \varphi_n^2 r\,dr
     \ge \frac{1}{qe^{C_0}}\int_{I_n}\frac{1}{r^2(-\log r)}\varphi_n^2 r\,dr > \int_{I_n}|\nabla \varphi_n|^2 r\,dr.
\end{align*}
This implies that 
$U_\infty$ is unstable for all $I_n$. Thus, by Lemma \ref{unstablelem} and the fact that
$v_m(r_n)>U_\infty(r_n)$ for all $n\le M$, we deduce that $U_\infty$ intersects with $v_m$ in $I_n$ for all $m>m_0$ and $n\le M$. Since $M$ is arbitrary, we get \eqref{eq-y106}

Next, we consider the case 
$U_\infty \equiv V_\infty$ or 
case $(ii)$ in Theorem \ref{thm4-1}
occurs. In this case, we choose $I_n=[r_{n+1}, r_n]$ and $\varphi_n\in C^{0,1}_{0}(I_n)$ such that \eqref{reversecritical} is satisfied. Therefore, 
in this case, we have 
by Lemma \ref{unstable-singular} 
that 
\begin{align*}
    \int_{I_n}f'(V_\infty) 
    \ge \frac{1}{qe^{C_0}}\int_{I_n}\frac{1}{r^2(-\log r)}\varphi_n^2 r\,dr > \int_{I_n}|\nabla \varphi_n|^2 r\,dr.
\end{align*}
Let $M>0$. Then, thanks to the above discussion,  
we can deduce that $v_m$ is unstable in $I_n$ for all $n\le M$ if $m$ is sufficiently large. Thus, 
since for any $n \leq M$, 
both of the functions $U_{\infty}$ and 
$v_{m}$ are unstable in $I_{n}$, \eqref{eq-y106} follows from Lemma \ref{unstablelem}.
\end{proof}
\appendix

\section{The examples of nonlinearities} \label{ex-non}
Let $f_1$ and $f_2$ be those of \eqref{ex1-1}. Then, by a direct  computation, we see that the following examples of $g$ satisfy (G1), (G2) and \eqref{addition}. Moreover, $q$ and $H$ are computed as follows 
\begin{enumerate}
\item[\textrm{(i)}] $H(s)=\frac{1}{B}+\frac{-B'(B'-1)(2B'-1)\log s +\log (4/BB')+(2B'-1)^2-(2B'-1)^2 s^{-B'}\log s +[(2B'-1)\log (4/BB')-\frac{1}{B}(2B'-1)^2]s^{-B'}}{s^{B'}(B'-(2B'-1)s^{-B'})^2}$\\
\hspace{0.73cm}
$=\frac{1}{B}+O(\frac{\log s}{s^{B'}})$
and $q=B$ \hspace{2mm}
    if $g(s)=\log f_{1}=s^{B'}-(2B'-1)\log s+\log (4/BB')$;
\item[\textrm{(ii)}] $H(s)=1+\frac{i}{(1+i) s^{(1+i)}} 
= 1 + O(s^{-(1 + i)})$ 
    and $q = 1$
    \hspace{2mm}
    if $g(s) =e^{s^{i+1}}$ with $i>-1$;
    \item[\textrm{(iii)}] $H(s)=\frac{p-1}{p}+\frac{(p-2r)(p-1)+r(r-1)+r(r-1-\frac{(p-1)r}{p})s^{r-p}}{s^{p-r}(p+rs^{r-p})^2}$ and $q=\frac{p}{p-1}$ \\
    \hspace{0.73cm}
    $= \frac{p-1}{p} + O(s^{-(p-r)})$
    \hspace{2mm} if $g=s^{p}+s^{r}$ with $p>r>0$ and $p>1$; 
    \item[\textrm{(iv)}] $H(s)=\frac{p-1}{p}+\frac{i+(i(i-1)-\frac{p-1}{p}i^2)(\log s)^{-1}}{\log s
    (p+i(\log s)^{-1})^2} 
    = \frac{p-1}{p}+ O((\log s)^{-1})$
    and $q = \frac{p}{p-1}$
    \hspace{2mm}
    if 
    $g=s^p(\log s)^i$;
\item[\textrm{(v)}] $H(s)=1+\frac{(-2s+4+\log 4)-4e^{-s}}{e^s(1-2e^{-s})^2} 
= 1 + O(\frac{s}{e^{s}})$ 
and $q=1$ \hspace{2mm} if $g=\log f_2=e^{s}-2s+\log 4$;
    \item[\textrm{(vi)}] $H(s) 
    =1+\frac{1}{e^{s}}$
    and $q = 1$
    \hspace{2mm}
    if $g(s)=e^{e^s}$;  
    \item[\textrm{(vii)}] $H(s) =
    1 + \frac{1}{e^{e^{s}}} 
    + \frac{1}{e^{s} e^{e^{s}}}$ 
    and $q = 1$
    \hspace{2mm}
    if $g(s)=e^{e^{e^s}}$. 
\end{enumerate}
Moreover, by the direct computation, we see that $g$ satisfies \eqref{equiv-apendix} 
below for the cases (i) and (v); while $g$ does not satisfy \eqref{equiv-apendix}  
for the cases (vi) and (vii). For the cases (ii) and (iv), $g$ satisfies \eqref{equiv-apendix} if and only if $i=0$. For the case (iii), $g$ satisfies \eqref{equiv-apendix} if and only if $p>2r$.

\section{Relation between the asymptotic behavior
$g(\Tilde{u})$ and $g(U_{\infty})$}
Here, we study the relations 
between 
the asymptotic behavior
$g(\tilde{u})$ and 
$g(U_{\infty})$. 
Here, $\tilde{u}$ is given 
by \eqref{ex1-3}. 
We also study the relationship between the conditions 
(F1)--(F3) 
assumed in \cite{MR4876902} 
and our conditions 
(G1) and (G2). 

We first quote the following
\begin{lemma}
\label{hennalem}
Assume that $g\in C^3(u_0,\infty)$ satisfies (G1) and (G2).
Then, $f(s)=e^{g(s)}$ satisfies
\begin{equation} \label{asy-eq2}
     F(s)=\frac{1}{f(s)g'(s)}(1-(1+o(1))\frac{H(s)}{g(s)})\quad\text{as $s\to\infty$}, 
\end{equation}
where $F(s)$ is defined by 
\eqref{asf1}. 
In particular, we have $f'(s)F(s) 
\to 1$ as $s\to\infty$. 
Moreover, if $g$ satisfies the following additional condition
\begin{equation}
\label{addition}
\frac{g^{k}(s)|H^{(k)}(s)|}
{g'^{k}(s)}=o(1)\quad\text{as $s\to\infty$ for $k=1,2$,}
\end{equation} 
we obtain 
\begin{equation} \label{asy-eq3}
    F(s)=\frac{1}{f(s)g'(s)}
     \left[
     1-\frac{H(s)}{g(s)}- 
   \left(\frac{H'(s)}{g(s)g'(s)} 
    -\frac{H(s)}{g^2(s)} - 
    \frac{H^2(s)}{g^2(s)}
    \right) + 
    O(g^{-3}(s))
    \right] \quad\text{as $s\to\infty$}.
\end{equation}
\end{lemma}
\begin{proof}
First, we will prove \eqref{asy-eq2}. 
By  the assumptions (G1), (G2) 
and \eqref{eq-y52-2}, 
we have 
\begin{equation*}
    \frac{1}{g(s)}\log g'(s)
    \to 0, \hspace{4mm} 
    \frac{H(s)}{g(s)}\to 0\hspace{4mm}\text{and} \hspace{4mm} 
    \frac{g(s)}{H(s)g'(s)}(\frac{H(s)}{g(s)})'\to 0 
    \quad\text{as $s\to\infty$}
    .
\end{equation*}
Moreover, Lemma \ref{lem2-2} implies 
that $\lim_{s \to \infty} g(s) 
= \infty$ and $g'(s)>0$ 
for sufficiently large $s > 0$.
Thus, it follows from the proof of \cite[Lemma 2.2] {kumagai2025classificationbifurcationstructuresemilinear} 
~\footnote{
We apply 
$b(s)$ and $c(s)$ defined there 
as $b(s) = g(s)$ and $c(s) = H(s)/g(s)$.
}
that
\begin{equation} \label{asy-eq1}
 \begin{split}
    F(s)
    & =e^{-g(s)}g'^{-1}(s)
    (1 - \frac{H(s)}{g(s)})-\int_{s}^{\infty}e^{-g(t)}(\frac{H(t)}{g(t)g'(t)})'\,dt.
\end{split}
\end{equation}
Since 
\begin{equation*}
(\frac{H(t)}{g(t)g'(t)})'=  \frac{H'(s)}{g(s)g'(s)} -
    \frac{H(s)}{g^2(s)} - 
    \frac{H^2(s)}{g^2(s)}=O(g^{-2})=o(H/g),
\end{equation*}
we get \eqref{asy-eq2}. 

Next, we shall show \eqref{asy-eq3}. 
If \eqref{addition} is additionally 
satisfied, we have by 
(G1) that 
\begin{equation*}
    \frac{H'(s)}{g(s)g'(s)} -
    \frac{H(s)}{g^2(s)} - 
    \frac{H^2(s)}{g^2(s)} 
    =-(\frac{q+1}{q^2}+o(1))
    g^{-2}(s) 
    \quad \mbox{as $s \to \infty$}.
\end{equation*} 
Thus, by \eqref{asy-eq1} and 
integrating by parts, 
we obtain that 
\begin{align*}
    F(s)
    & = e^{-g(s)}g'^{-1}(s)
    (1 - \frac{H(s)}{g(s)})
    + 
    \int_{s}^{\infty}
    (e^{-g(t)})' g'^{-1}(t)
    \left(
    \frac{H'(t)}{g(t)g'(t)} -
    \frac{H(t)}{g^2(t)} - 
    \frac{H^2(t)}{g^2(t)} \right)
    \,dt \\
    & = 
    e^{-g(s)}g'^{-1}(s)
    (1-\frac{H(s)}{g(s)} - 
    \frac{H'(s)}{g(s)g'(s)}+\frac{H(s)}{g^2(s)}+\frac{H^2(s)}{g^2(s)}) 
    -
    \int_{s}^{\infty}e^{-g(t)}
    \varphi(t)\,dt
    \\
    &=e^{-g(s)}g'^{-1}(s)
    (1-\frac{H(s)}{g(s)} - 
    \frac{H'(s)}{g(s)g'(s)} + 
    \frac{H(s)}{g^2(s)} + 
    \frac{H^2(s)}{g^2(s)}-\varphi(s)) 
    -\int_{s}^{\infty}e^{-g(s)}(\varphi(t) g'^{-1}(t))'\,dt,
\end{align*}
where 
    \[
    \varphi(s)
   := \left(g'^{-1}(s)
    \left(
    \frac{H'(s)}{g(s) g'(s)} -
    \frac{H(s)}{g^2(s)} - 
    \frac{H^2(s)}{g^2(s)} \right)
    \right)'\\
    =(\frac{H'(s)}{g(s)g'^2(s)} - 
    \frac{H(s)}{g'(s)g^2(s)} - 
    \frac{H^2(s)}{g'(s)g^2(s)})'. 
    \]
Note that 
    \[
    \left(\frac{1}{g'(s)}\right)' 
    = - \frac{H(s)}{g(s)}, \qquad 
    \frac{1}{g(s)} 
     \left(\frac{1}{g^{2}(s)}\right)' 
     = - \frac{2}{g^{3}(s)}. 
    \]
Using this, 
with a little bit of effort, we find that  
\begin{align*}
     \varphi(s) 
     &=\frac{H''(s)}{g(s)g'^2(s)} 
   -\frac{H'(s)}{g^2(s)g'(s)} - 
   \frac{2H'(s)g''(s)}{g(s)g'^{3}(s)} 
   -\frac{H'(s)+2H(s)H'(s)}{
   g'(s) g^2(s)}+\frac{(H(s)+H^2(s))g''(s)}{g'^2(s)g^2(s)}+\frac{2(H(s)+H^2(s))}
   {g^3(s)}\\
    &=\frac{H''(s)}{g(s)g'^2(s)} 
    -\frac{H'(s)}{g^2(s)g'(s)} 
    -\frac{2H'(s)H(s)}{g^2(s)g'(s)} 
    -\frac{H'(s)+2H(s)H'(s)}{g'(s) g^2(s)}+\frac{H(s)(H(s)+H^2(s))}{g^3(s)}+\frac{2(H(s)+H^2(s))}{g^3(s)}\\
&=\frac{1}{g(s)g'^2(s)}H''(s) - 
\frac{2H'(s)}{g^2(s)g'(s)}(1+2H(s)) + 
\frac{1}{g^3(s)}(H^3(s)+3H^2(s)+2H(s))\\
    &=g^{-3}(s)(H^3(s) +3H^2(s) +2H(s) +o(1)). 
\end{align*}
We see from the condition (G2) that
\begin{align*}
(\varphi g'^{-1})'&=\frac{1}{gg'^3}H'''-\frac{1}{g^2g'^2}(2H''+4H'^2+4HH'')+\frac{H'}{g^3g'}(3H^2+6H+2)\\
&\qquad-H''(\frac{1}{g^2g'^2}+\frac{3g''}{gg'^4})+4(\frac{1}{g^3g'}+\frac{g''}{g'^3 g^2})H'(1+2H)-
(H^3+3H^2+2H)(\frac{3}{g^4}+\frac{g''}{g^3g'^2})\\
&=O(g^{-4})=o(\varphi).
\end{align*}
Therefore, we can deduce that
\begin{equation*}
     F(s)=e^{-g}g'^{-1}(1-\frac{H}{g}-(\frac{H'}{gg'}-\frac{H}{g^2}-\frac{H^2}{g^2})+O(g^{-3}))\quad\text{as $s\to\infty$.}
\end{equation*}
Thus, \eqref{asy-eq3} holds. 
\end{proof}
From the above lemma, we can get the following
\begin{lemma}
\label{bnolem}
We assume that $g$ satisfies $(G1)$, $(G2)$ and \eqref{addition}. Then, $f$ satisfies (F1) and (F2) 
(see Section \ref{intro}). In particular, we have $B=q$, where 
$B = \lim_{s \to \infty}B_{2}[f]$.
Moreover, $f$ satisfies (F3) if and only if 
\begin{equation}
\label{equiv-apendix}
g^{1/2}(s)(H(s)-\frac{1}{B})\to 0 
\quad\text{and}\quad g^{1/2}
(s)\frac{H'g}{g'}(s)\to 0 
\hspace{4mm} \mbox{as $s \to 
\infty$}.
\end{equation}
\end{lemma}
\begin{proof}
We first remark that (F1) follows from Lemma \ref{hennalem} and Lemma \ref{lem2-2}. In addition, Since $f'(s)=f(s) g'(s)$ and $f''(s)=f(s)^{-1}f'(s)^2+f(s)g''(s)$ are satisfied, it follows from the direct computation that
    \begin{align*}
    & f'(s) F(s) 
    = 1-\frac{H(s)}{g(s)}- 
   \left(\frac{H'(s)}{g(s)g'(s)} 
    -\frac{H(s)}{g^2(s)} - 
    \frac{H^2(s)}{g^2(s)}
    \right) + 
    O(g^{-3}(s)), \\
    & \frac{f(s)f''(s)}
    {f'(s)^2}=1+\frac{g''(s)}
    {g'(s)^2}= 1+\frac{H(s)}
    {g(s)}.
    \end{align*}
Moreover, by Lemma \ref{hennalem}, we have 
\begin{equation}\label{hyouji-b1}
\begin{split}
  (B_{1}[f](s))^{-1}
  &=(-\log F)(f'(s)F(s)-1) \\
  & =(g(s)+\log g'(s)+o(1))(\frac{H(s)}{g(s)}+\frac{H'(s)}{g(s)g'(s)} 
  -\frac{H(s)}{g^2(s)}-\frac{H^2(s)}{g^2(s)}+O(g^{-3}(s)))\\
  &=H(s)+O(\frac{\log g(s)}{g(s)}) \qquad\text{(by Lemma \ref{lem2-1})}
\end{split}
\end{equation}
and
\begin{equation}\label{hyouji-b2}
\begin{split}
(B_{2}[f](s))^{-1}=&f'(s)F(s) (-
\log F(s))^2 
(\frac{f(s)f''(s)F(s)}
{f'(s)}-1)\\
&\qquad=(g(s)+\log 
g'(s)+o(1))^2[(1+\frac{H(s)}
{g(s)})(1-\frac{H(s)}{g(s)} 
-\frac{H'(s)}{g(s)g'(s)} 
+ (\frac{H(s)}{g^2(s)} + 
\frac{H^2(s)}{g^2(s)})+ 
O(g^{-3}(s)))-1]\\
&\qquad=(1+O(\log g(s)/g(s)))
(H(s)-\frac{H'(s)g(s)}{g'(s)}+O(g^{-1}(s)))\\
&\qquad=H(s) -\frac{H'(s)g(s)}{g'(s)}+O(\log g(s)/g(s)).
\end{split}
\end{equation}
Therefore, the limit $B$ exists and $B=q$.
Moreover, we observe from \eqref{asy-eq3} and 
\eqref{ex1-3} 
that 
    \[
    g(\Tilde{u}) = -(1+o(1))\log F(\Tilde{u})=-(2+o(1))\log r
    \]
    for any $f$ such that $g=\log f$ satisfies (G1) and (G2). In particular, since $g_j:=\log f_{j}$ satisfies (G1), (G2) and \eqref{addition} (see Subsection 2.1), we have $g_{j}(u_{j,\infty})=-(2+o(1))\log r$ for $j=1,2$, where $f_1$, $f_2$, $u_{1,\infty}$ and $u_{2,\infty}$ are those in \eqref{ex1-1} and \eqref{ex1-2}. Therefore, since $g_j$ satisfies \eqref{equiv-apendix} (see Subsection 2.1) and $B=q$, we have
\begin{equation} \label{equiv-apendix3}
\begin{split}
(-\log r)^{1/2}&\sum_{i=1}^{2}|(B_i[f_j](u_{j,\infty}))^{-1}-\frac{1}{B}|\\
&\le Cg_{j}^{1/2}(u_{j,\infty})\left(|\frac{1}{q}-H(u_{j,\infty})|+|\frac{1}{q}-H(u_{j,\infty})+\frac{H'g_{j}}{g_{j}'}(u_{j,\infty})|\right)+o(\frac{\log g_{j}}{g^{1/2}_{j}}(u_{j,\infty}))\\
&\to 0  \qquad\text{as $r\to 0$ for any $j=1,2$.}
\end{split}
\end{equation}

We now prove \eqref{equiv-apendix}
implies (F3). 
Indeed, by \eqref{hyouji-b1} and 
\eqref{hyouji-b2}, we obtain 
    \[
    \begin{split}
    (- \log r)^{1/2}
    & \sum_{i = 1}^{2}
    \biggl|
    (B_{i}[f_{j}](u_{j, \infty}))^{-1} 
    - (B_{i}[f](\tilde{u}))^{-1}
    \biggl| \\
    & \leq 
     (- \log r)^{1/2}
    \sum_{i = 1}^{2}
    \biggl|
    (B_{i}[f_{j}](u_{j, \infty}))^{-1} 
    - \frac{1}{B}
    \biggl| +  (- \log r)^{1/2}
    \sum_{i = 1}^{2}
    \biggl|
   H(\tilde{u}) 
    - \frac{1}{B}
    \biggl| \\ 
    & \qquad + (- \log r)^{1/2}
    \sum_{i = 1}^{2}
    \biggl|
   H(\tilde{u}) 
    -  (B_{i}[f](\tilde{u}))^{-1}
    \biggl| \\
    & \leq 
     (- \log r)^{1/2}
    \sum_{i = 1}^{2}
    \biggl|
    (B_{i}[f_{j}](u_{j, \infty}))^{-1} 
    - \frac{1}{B}
    \biggl| +  C
    \biggl|
    g^{1/2}(\tilde{u})\left(
   H(\tilde{u}) - \frac{1}{B}\right)
    \biggl| \\ 
    & \qquad + 
    C (\log g(\tilde{u})/g^{1/2}(\tilde{u})) 
    + C g^{1/2}(\tilde{u}) 
    \biggl|\frac{H'(\tilde{u}) g(\tilde{u})}
    {g'(\tilde{u})}\biggl|. 
    \end{split}
    \]
From \eqref{equiv-apendix} and 
\eqref{equiv-apendix3}, we find that (F3) 
holds.

Next, we will show that if $f$ satisfies 
(F3), then \eqref{equiv-apendix} holds. 
Since $\tilde{u}$ is monotone decreasing 
in $r$, it suffices to show that  
\begin{equation*}
g^{1/2}(\tilde{u})(H(\tilde{u}) - 
\frac{1}{B})\to 0 
\quad\text{and}\quad g^{1/2}
(\tilde{u})\frac{H'g}{g'}(\tilde{u})\to 0 
\hspace{4mm} \mbox{as $r \to 0$}.
\end{equation*}
Since $B = q$ and 
$g(\tilde{u}) = - (2 + o(1)) \log r$, 
we have by \eqref{hyouji-b1} that 
    \begin{equation}\label{equiv-apendix4}
    \begin{split}
    g^{1/2}(\tilde{u}) 
    \biggl|(H(\tilde{u}) - 
\frac{1}{B}) \biggl|
    & \leq 3(- \log r)^{\frac{1}{2}}
       \biggl|
       (B_{1}[f](\tilde{u}))^{-1} - 
       O(\frac{\log g(\tilde{u})}
       {g(\tilde{u})})
       - \frac{1}{B})\biggl| \\
     &  \leq 3(- \log r)^{\frac{1}{2}}
       \biggl|
       (B_{1}[f](\tilde{u}))^{-1} 
       - (B_1[f_j](u_{j,\infty}))^{-1}
       \biggl| \\
     & \quad + 3(- \log r)^{\frac{1}{2}}
       \biggl|B_1[f_j](u_{j,\infty}))^{-1} 
       - \frac{1}{B}\biggl|
       + O(\frac{\log g(\tilde{u})}
       {g^{1/2}(\tilde{u})})
    \end{split}
    \end{equation}
This together with \eqref{equiv-apendix3} 
and (F3), we see that 
$\lim_{r \to 0}g^{1/2}(\tilde{u})(H(\tilde{u}) - \frac{1}{B}) = 0$. 
In addition, it follows from 
\eqref{hyouji-b2} that 
    \[
    \begin{split}
    g^{1/2}(\tilde{u})
    \biggl|\frac{H'g}{g'}(\tilde{u})\biggl|
& \leq g^{1/2}(\tilde{u})
\biggl|(B_{2}[f](\tilde{u}))^{-1} 
- H(\tilde{u}) + O(\log g(\tilde{u})
/g(\tilde{u}))
\biggl| \\
&  \leq 3 (- \log r)^{1/2}
\biggl|(B_{2}[f](\tilde{u}))^{-1} 
- (B_{2}[f_{j}](u_{j, \infty}))^{-1}
\biggl| \\
& \quad + 3(- \log r)^{1/2}
\biggl|(B_{2}[f_{j}](u_{j, \infty}))^{-1}
- \frac{1}{B}\biggl|
+ g^{1/2}(\tilde{u})
\biggl|H(\tilde{u}) - \frac{1}{B}\biggl|
+ O(\log g(\tilde{u})
/g^{1/2}(\tilde{u})). 
    \end{split}
    \]
Thus, we see from \eqref{equiv-apendix3}, 
\eqref{equiv-apendix4} and (F3) that 
$\lim_{r \to 0} g^{1/2}
(\tilde{u})\frac{H'g}{g'}(\tilde{u}) = 0$.

Therefore, we see that the condition (F3) is equivalent to \eqref{equiv-apendix}.
\end{proof}
We now prove the following
\begin{lemma}
\label{apenlem}
We assume that $g$ satisfies (G1), (G2) and \eqref{addition}. Then, $\Tilde{u}$ given in \eqref{ex1-3} with $B=1/q$ satisfies 
\begin{equation} \label{asy-eq8}
g(\Tilde{u})=-2\log r-2\log(-2\log r)+\log[\frac{g}{g'}(g^{-1}(-2\log r))]+\log 4-\log q+o((-\log r)^{-1+\e})\qquad\text{as $r\to 0$}
\end{equation}
for any $\e>0$. Moreover, if additionally (F3) is satisfied, 
the singular solution $u$ provided by \cite{MR4876902} (see \eqref{eq1-FI}) satisfies
\begin{equation}
\label{apeneq-1}
g(u)=g(\Tilde{u})+O(R)=g(\Tilde{u})+o((-\log r)^{-1/2})\qquad\text{as $r\to 0$},
\end{equation}
where $R$ is 
the function given in \eqref{eq1-FI}. In particular, $u$ corresponds with $U_{\infty}$ in Theorem \ref{thm-sing}.

\end{lemma}
\begin{remark}
\label{apenrem}
\rm{Lemma \ref{apenlem} shows that the singular solution $U_{\infty}$ in Theorem \ref{thm-sing} satisfies
\begin{equation*}
    g(U_{\infty})=g(\Tilde{u})+\log qH(g^{-1}(-2\log r)) 
    + o((-\log r)^{-1/2}),
\end{equation*} 
Moreover, Lemma \ref{bnolem} 
tells us that under the condition $(F3)$, one has
\begin{equation*}
    \log qH(g^{-1}(-2\log r))=\log (q(H(g^{-1}(-2\log r))
    -\frac{1}{q})+1)=O(H(g^{-1}(-2\log r))
    -\frac{1}{q})=o((-\log r)^{-1/2})
\end{equation*}
and thus the second term 
can be regarded 
as the reminder term. 
}

\end{remark}

\begin{proof}
We remind that
\begin{equation*}
    F(\Tilde{u})=\Phi(u_{i,\infty})= \frac{B}{4} e^{-\rho}(\rho+1)\quad\text{for $i\in \{1,2\}$ with $\rho=-2\log r$},
\end{equation*}
which implies that
\begin{equation} \label{asy-eq4}
  w :=-\log F(\Tilde{u})=\rho-\log \rho-\log\frac{B}{4}+O(1/\rho). 
\end{equation}
Therefore, we obtain by 
\eqref{asy-eq2} and Lemma 
\ref{lem2-1} that 
\begin{equation}
\label{nagaisiki}
    -\log F(\Tilde{u}) =g(\Tilde{u})+\log g(\Tilde{u})-\log (g(\Tilde{u})/g'(\Tilde{u}))+\log (1-(1+o(1))\frac{H(\Tilde{u})}{g(\Tilde{u})})=g(\Tilde{u})+(\frac{1}{q}+o(1))\log g(\Tilde{u}).
\end{equation}
Thus, we first obtain
\begin{equation} \label{asy-eq5}
    g(\Tilde{u})=\rho-(1+\frac{1}{q}+o(1))\log \rho,\hspace{4mm}\log g(\Tilde{u})=\log \rho+O(\log\rho/\rho). 
\end{equation}
Indeed, set $\eta\log \rho =g(\Tilde{u})-\rho+(1+\frac{1}{q})\log \rho$. Then, 
by \eqref{asy-eq4} and 
\eqref{nagaisiki},
$\eta$ satisfies
\begin{equation*}
\begin{split}
\rho-\log\rho -\log\frac{B}{4}+O(1/\rho)
& = g(\Tilde{u})+(\frac{1}{q}+o(1))\log g(\Tilde{u}) \\
& =\rho+ \left(-(1+\frac{1}{q})+\eta+(\frac{1}{q} + o(1))
\frac{\log(\rho+(-(1+\frac{1}{q})+\eta)\log\rho)}{\log\rho} \right)
\log \rho \\
& = \rho+ \left(-(1+\frac{1}{q})+\eta+(\frac{1}{q} + o(1))
\frac{\log\rho+O((-(1+\frac{1}{q})+\eta)\frac{\log\rho}{\rho})}
{\log\rho} \right)
\log \rho \\
& = \rho+ \left(\eta -1 + o(1)
+ O((1+|\eta|)\frac{\log\rho}{\rho}(\log \rho)^{-1}) \right)
\log \rho,
\end{split}
\end{equation*}
which implies that
\begin{equation*}
\eta+o(1)+O(\frac{(1+|\eta|)}{\rho})
=o(1).
\end{equation*}
Therefore, we have $\eta=o(1)$ as 
$\rho \to \infty$.

We define 
$y_1$ as $g(y_1)=\rho$. 
Observe from 
$g(\tilde{u}) < g(y_1)$ 
(see \eqref{asy-eq5}) 
that $\tilde{u} < y_{1}$. 
Then, since
\begin{equation*}
  g'(y_1)(y_1-\tilde{u}) \ge g(y_1)-g(\Tilde{u})\ge g'(\Tilde{u})(y_1-\Tilde{u}),
\end{equation*}
we have by \eqref{asy-eq5} that 
\begin{equation} \label{asy-eq6}
 \frac{C\log\rho}{g'(y_1)} \le |y_1-\Tilde{u}|\le \frac{C\log\rho}{g'(\Tilde{u})}
\end{equation}
By Lemmas \ref{lem2-1}, 
\eqref{eq-y52-2}, 
\eqref{asy-eq6} and 
the convexity of $g$, we obtain 
\begin{align*}
 &\frac{g'(y_1)-g'(\Tilde{u})}{g'(y_1)}\le C \frac{g''(\hat{u})}{g'(y_1)g'(\Tilde u)} 
 \log \rho
 =C \frac{g'(y_1)^2\log\rho}{g'(y_1)g^{1/q+o(1)+1}(\Tilde{u})}\le C \frac{g'(y_1)\log\rho}{g^{1/q-\delta+1}(y_1)}\le C \frac{C\log\rho}{g^{1-\delta}(y_1)}\\
 &(g(s)/g'(s))'=\frac{g'^2-g''g}{g'(s)^2}\to 1-\frac{1}{q},
\end{align*}
where $\hat{u} = \theta y_{1} 
+ (1 - \theta) \tilde{u}$ for 
some $\theta \in (0, 1)$ and $\delta>0$ is sufficiently small fixed constant. 
Therefore, it follows 
from the mean value theorem, \eqref{asy-eq5} 
and \eqref{asy-eq6} that 
\begin{equation} \label{asy-eq7}
    \frac{g(\Tilde{u})}{g'(\Tilde{u})}=\frac{g(y_1)}{g'(y_1)}(1+\frac{g'(y_1)}{g(y_1)}(g(\hat{u})/g'(
    \hat{u}))'(y_1-\Tilde{u}))= \frac{g(y_1)}{g'(y_1)}\left(1+
    O(\frac{g'(y_1)}{g'(\Tilde{u})g(y_1)}
    \log\rho)
    \right)=\frac{g(y_1)}{g'(y_1)}(1+O(\log\rho/\rho)). 
\end{equation}
From the above estimates, we obtain
\begin{align*}
&\log g(\Tilde{u})=\log (\rho-(\frac{1}{q}+o(1))\log \rho)=\log \rho +O(\log \rho/\rho), \\
&\log (\frac{g(\Tilde{u})}{g'(\Tilde{u})})=\log (\frac{g(y_1)}{g'(y_1)}(1+O(\log\rho/\rho)))=\log(\frac{g(y_1)}{g'(y_1)})+O(\log\rho/\rho), \\
&\log(1-(1+o(1))\frac{H(\Tilde{u})}{g(\Tilde{u})})=O(\frac{1}{g(\Tilde{u})})=O(1/\rho).
\end{align*}
Therefore, by using \eqref{asy-eq4} and \eqref{nagaisiki}, we obtain 
\begin{equation*}
    \rho-\log\rho-\log\frac{B}{4}+O(1/\rho)=g(\Tilde{u})+\log\rho-\log(g(y_1)/g'(y_1))+O(\log\rho/\rho).
\end{equation*}

Thus, we obtain  
\begin{equation*} 
\begin{split}
g(\Tilde{u})
=\rho-2\log\rho+\log(g(y_1)/g'(y_1))+\log 4-\log q+O(\log\rho/\rho), 
\end{split}
\end{equation*}
which yields \eqref{asy-eq8}. 

In particular, if (F3) is assumed, then 
$R = o(\rho^{-1/2})$ as $\rho \to 
\infty$.  
We see from 
\eqref{asy-eq2} and \eqref{asy-eq8} that 
the solution $u$ of \eqref{eq1-FI} satisfies
\begin{equation*}
u=\Tilde{u}+O(F(\Tilde{u})
f(\Tilde{u})R(\rho)) =\Tilde{u}+O(\frac{R(\rho)}{g'(\Tilde{u})})\le y_1
\end{equation*}
where
\begin{equation*}
   R(\rho):=\sup_{|y|\le |x|}(R_1(y)+R_2(y)).
\end{equation*}
Moreover, since $g(\tilde{u})=(1+o(1))g(y_1)$ as $\rho\to\infty$, it follows from \eqref{asy-eq7}, (G1) and the monotonicity of $g'$ that  
\begin{align*}
\frac{g''(\hat{u})}
{g'^2(\Tilde{u})}
&\le C\frac{g'^2(\hat{u})}{g(\hat{u})g'^2(\Tilde{u})}\le C\frac{g'^2(y_1)}{g(\Tilde{u})g'^2(\Tilde{u})}\\
&=C\frac{g^2(\Tilde{u})}{g'^2(\Tilde u)}\cdot\frac{g'^2(y_1)}{g^3(\Tilde{u})}\le C \frac{g^2(y_1)}{g^3(\Tilde{u})}\le 
\frac{C}{\rho} \quad\text{for any $\hat{u}$ between $\Tilde{u}$ and $y_1$}.
\end{align*}
Hence, by \eqref{asy-eq8}, 
we have
\begin{equation} \label{asy-eq9}
\begin{split}
g(u)&=g(\Tilde{u})+O(g'(\Tilde{u})f(\Tilde{u})F(\Tilde{u})R(\rho))+O\left(g''(\hat{u})
\left(f(\Tilde{u})F(\Tilde{u})R(\rho)\right)^2\right)\\
    &=g(\Tilde{u})+O(R(\rho))
    +O\left(\frac{g''(\hat{u})}{g'^2(\Tilde{u})}R(\rho)^2\right)\\
    &=g(\Tilde{u})+O(R(\rho))+O(\rho^{-1}R(\rho)^2)\\
    &
    =\rho-2\log\rho+\log(g(y_1)/g'(y_1))+\log 4-\log q+O(R(\rho))\\
    &
    =\rho-2\log\rho+\log(g(y_1)/g'(y_1))+\log 4-\log q+o(\rho^{-1/2}).
\end{split}
\end{equation}
Therefore, we have \eqref{apeneq-1}.

Finally, we prove that their singular solution $u$ corresponds with our singular solution $U_{\infty}$ if (F3) is satisfied. By \eqref{eq-asy}, \eqref{asy-eq9} and 
Lemma \ref{bnolem}, we obtain 
\begin{equation} \label{asy-eq11}
    |g(U_{\infty})-g(u)|=|\log qH|+o(\rho^{-1/2})=|\log(1+q(H-\frac{1}{q}))|+o(\rho^{-1/2})=o(\rho^{-1/2}).
\end{equation}
Moreover, by \eqref{shape-y2} and the monotonicity and convexity of $g$, we have
\begin{equation} \label{asy-eq10}
C\log \rho\ge C g'(y_1)|y_2|\ge|g(y_1)-g(y_1+2y_2)|\ge 2|y_2| g'(y_1+2y_2).
\end{equation}
Hence, by the condition (G1) and 
\eqref{asy-eq10}, we have
\begin{align*}
    |g'(y_1+2y_2)-g'(y_1)|\le g''(y_1+\theta(2y_2-y_1))|y_2|&\le C\frac{g'^2(y_1+\theta(2y_2-y_1))}{g(y_1+\theta(2y_2-y_1))}|y_2|\\
    &\le C\log\rho g'(y_1)g(y_1+2y_2)^{-1}
    \le C g'(y_1)\frac{\log\rho}{\rho}
\end{align*}
with some $0<\theta<1$. From the two estimates above and \eqref{shape-y2}, we can deduce that
\begin{equation*}
    g(y_1+2y_2)\le g(y_1)-2(2+o(1))\log\rho\le \min\{g(u),g(U_{\infty})\}.
\end{equation*}
Therefore, by the monotonicity of $g$ and the asymptotics of $U_{\infty}$ and $u$,
we see that $U_{\infty}>y_1+2y_2$ and $u>y_1+2y_2$. Hence, by the convexity of $g$, \eqref{asy-eq11} and 
the mean value formula, we have
\begin{align*}
|u-U_{\infty}|
&= |g^{-1}\left(g(U_{\infty}) 
- (g(U_{\infty}) - g(u))\right)-U_{\infty}|
=|g^{-1}(g(U_{\infty})+o(\rho^{-1/2}))-U_{\infty}|\\
&\le\frac{o(\rho^{-1/2})}{g'(y_1+2y_2)}\le o((g'(y_1))^{-1}\rho^{-1/2}).
\end{align*}
By Proposition \ref{prop-es}, we get the result.
\end{proof}

\section{Computation of 
derivative $y_{2}$}
\label{sec-y2}
In this appendix, we shall give a proof of 
\eqref{derivative-y2}. 
It follows from \eqref{eq-y59} that 
    \begin{equation} \label{eq-y95}
    \frac{d }{d \rho} (G(\rho)) = 
    H'(y_{1}(\rho)) \frac{d y_{1}}{d \rho}. 
    \end{equation}
By \eqref{eq-y76} and \eqref{eq-y52-2}, 
we have 
    \begin{equation} 
    \label{eq-y85}
    |\frac{d y_{1}}{d \rho}| 
    = \frac{1}{g'(y_{1})} 
    \leq  C \rho^{- \frac{1}{q} + o(1)}. 
    \end{equation}
By \eqref{eq-y4}, \eqref{eq-y76} and \eqref{eq-y59}, we have 
\begin{align} 
\label{eq-y79}
\begin{split}
& \frac{d}{d\rho}(\log g'(y_1))=\frac{g''(y_1)}{g'(y_1)} 
\frac{d y_1}{d \rho} 
=\frac{g''(y_1)}{g'(y_1)^2}=\frac{G(\rho)}{\rho}, \\
& \frac{d}{d\rho}(\frac{1}{g'(y_1)}) 
= - \frac{g''(y_{1})}{g'(y_1)^2} 
\frac{d y_{1}}{d \rho} 
=- \frac{1}{g'(y_1)^2} 
\frac{G(\rho) g'(y_{1})^2}{g(y_{1})}
\frac{d y_{1}}{d \rho}
=- \frac{G(\rho)}{\rho} 
\frac{d y_{1}}{d \rho}.
\end{split}
\end{align}
Observe from \eqref{eq-y76} and   
\eqref{eq-y79} 
that 
    \[
    \frac{d^{2} y_{1}}{d \rho^{2}} 
    = \frac{d}{d\rho} 
    \left(\frac{1}{g'(y_{1})} \right)
    = - \frac{G(\rho)}{\rho} \frac{d y_{1}}{d 
    \rho},  
    \]
which together with \eqref{eq-y85} 
and the condition (G1) imply 
    \begin{equation} \label{eq-y89}
     |\frac{d^2 y_{1}}{d \rho^2}| 
    \leq  C \rho^{-1 - \frac{1}{q} + o(1)}. 
    \end{equation} 
In addition, since 
$\rho \frac{d^{2} y_{1}}{d \rho^{2}} 
= - G(\rho) \frac{d y_{1}}{d \rho}$, 
one has by \eqref{eq-y95} that  
    \begin{equation} \label{eq-y92}
    \frac{d^{2} y_{1}}{d \rho^{2}} + 
    \rho \frac{d^{3} y_{1}}{d \rho^{3}}
    = - H'(y_{1}) 
    \left(\frac{d y_{1}}{d \rho} \right)^2 
    - G(\rho) \frac{d^{2} y_{1}}{d \rho^{2}}. 
    \end{equation}
We have by 
\eqref{eq-y77}, 
\eqref{eq-y79}, $\frac{d y_{1}}{d \rho} 
= \frac{1}{g'(y_{1})}$ (see \eqref{eq-y76})
\eqref{eq-y79} and \eqref{eq-y95} that 
\begin{align} \label{eq-y81}
\frac{d y_2}{d \rho}=\frac{G(\rho)}{\rho g'(y_1)}(\log \rho+\log g'(y_1)-\log G-\log 4)-\frac{1}{g'(y_1)}(\frac{1}{\rho}+\frac{G(\rho)}{\rho} 
- \frac{H'(y_1)}{G(\rho)g'(y_1)})
\end{align}
By the condition (G2), \eqref{eq-y85}, 
\eqref{eq-y52-2} and 
\eqref{eq-y89}, we have 
    \[
    |H'(y_{1}) 
    \left(\frac{d y_{1}}{d \rho} \right)^2|, 
    |G(\rho) \frac{d^{2} y_{1}}{d \rho^{2}}| 
    \leq C \rho^{-1 - \frac{1}{q} + o(1)}. 
    \]
This together with \eqref{eq-y92} implies 
    \begin{equation}\label{eq-y91}
    |\frac{d^3 y_{1}}{d \rho^3}| 
    \leq  C \rho^{-2 - \frac{1}{q} + o(1)}. 
    \end{equation}
It follows from \eqref{eq-y77},  
\eqref{eq-y81} and $\frac{d y_{1}}{d \rho} 
= \frac{1}{g'(y_{1})}$ that 
    \begin{equation} \label{eq-y84}
    \begin{split}
    \rho \frac{d y_{2}}{d \rho} 
    & = - G(\rho) y_{2}
    - \frac{d y_{1}}{d \rho} 
    (1 + G(\rho) - 
    \rho K(\rho)
    \frac{d y_{1}}{d \rho} ) \\
    & = - G(\rho) y_{2}
    - \frac{d y_{1}}{d \rho} 
    - G(\rho) \frac{d y_{1}}{d \rho}  - 
    \rho K(\rho)
    \left(\frac{d y_{1}}{d \rho}\right)^2,
    \end{split}
    \end{equation}
where 
    \[
    K(\rho) := 
    \frac{H'(y_{1}) }{G(\rho)}. 
    \]
By the condition (G2) and \eqref{eq-y52-2}, 
we obtain
    \begin{equation} \label{eq-y86}
    |K(\rho)| \leq C\frac{g'(y_{1})}{g(y_{1})} 
    \leq C\rho^{-1 + \frac{1}{q} + o(1)}.  
    \end{equation}
It follows from \eqref{eq-y51-2}, 
\eqref{eq-y85}, 
and  \eqref{eq-y86} that
    \[
    \begin{split}
    |G(\rho) y_{2}| \leq C \rho^{- 
    \frac{1}{q} +o(1)} 
    (\log \rho), 
    \qquad 
    |\frac{d y_{1}}{d \rho}|,  
    |\frac{d y_{1}}{d \rho} G(\rho)|, 
    |\rho K(\rho)
    \left(\frac{d y_{1}}{d \rho}\right)^2| 
    \leq C \rho^{- \frac{1}{q} 
    + o(1)},
    \end{split}
    \]
This together with \eqref{eq-y84} 
yields that 
    \begin{equation} \label{eq-y88}
    |\frac{d y_{2}}{d \rho}| 
    \leq C \rho^{-1 - 
    \frac{1}{q} +o(1)} 
    (\log \rho)
    \end{equation}
Since $K(\rho) G(\rho) = H'(y_{1})$, 
we obtain 
    \begin{equation} \label{eq-y93}
    K'(\rho) G(\rho) 
    + K(\rho) H'(y_{1}) \frac{d y_{1}}{d \rho} 
    = H''(y_{1}) \frac{d y_{1}}{d \rho}. 
    \end{equation}
It follows from \eqref{eq-y95}, \eqref{eq-y85}, 
\eqref{eq-y52-2} and \eqref{eq-y86} that 
    \begin{equation}\label{eq-y87}
    |K'(\rho)| 
    \leq  |K(\rho) H'(y_{1}) \frac{d y_{1}}{d \rho}| + |H''(y_{1}) \frac{d y_{1}}{d \rho}|
    \leq C \rho^{- 2 + \frac{1}{q} + o(1)}. 
    \end{equation}
Differentiating \eqref{eq-y93}, 
one has 
    \begin{equation} 
    \label{eq-y113}
    K''(\rho) G(\rho) 
    + 2 K'(\rho) H'(y_{1}) 
    \frac{d y_{1}}{d \rho} 
    + K(\rho) H''(y_{1}) 
    \left(\frac{d y_{1}}{d \rho} 
    \right)^{2} + 
    K(\rho) H'(y_{1}) \frac{d^2 y_{1}}{d \rho^2} 
    = H'''(y_{1})  \left(\frac{d y_{1}}{d \rho} 
    \right)^{2} 
    + H''(y_{1}) \frac{d^{2} y_{1}}{d \rho^{2}}. 
    \end{equation}
Then, it follows from the condition (G2), 
\eqref{eq-y85}, \eqref{eq-y89} and 
\eqref{eq-y87}
that 
    \[
    \begin{split}
    |K'(\rho) H'(y_{1}) 
    \frac{d y_{1}}{d \rho}|, 
    |K(\rho) H''(y_{1}) 
    \left(\frac{d y_{1}}{d \rho} 
    \right)^{2}|, 
    |K(\rho) H'(y_{1}) \frac{d^2 y_{1}}{d 
    \rho^2}|,
    |H'''(y_{1})  \left(\frac{d y_{1}}{d \rho} 
    \right)^{2}|, 
    |H''(y_{1}) \frac{d^{2} y_{1}}{d \rho^{2}}|
    \leq C \rho^{-3 + \frac{1}{q} + o(1)}.
    \end{split}
    \]
This together with \eqref{eq-y113} yields that 
    \begin{equation}\label{eq-y94}
    |K''(\rho)| \leq C 
    \rho^{-3 + \frac{1}{q} + o(1)}.    
    \end{equation}
Differentiating \eqref{eq-y84}, 
we obtain
    \begin{equation} \label{eq-y82}
    \begin{split}
\frac{d y_{2}}{d \rho} 
+ \rho \frac{d^2 y_{2}}{d \rho^2} 
& = -  H'(y_{1}) \frac{d y_{1}}{d \rho} 
y_{2}
- G(\rho) \frac{d y_{2}}{d \rho} 
- \frac{d^{2} y_{1}}{d \rho^2}
- G(\rho) 
\frac{d^{2} y_{1}}{d \rho^2} 
- H'(y_{1}) \left(\frac{d y_{1}}{d 
\rho}\right)^2  \\
& \quad 
- K(\rho) \left(\frac{d y_{1}}{d 
\rho}\right)^2 
- \rho  K'(\rho) \left(\frac{d y_{1}}{d 
\rho}\right)^2
- 2 \rho K(\rho) 
\frac{d y_{1}}{d 
\rho} 
\frac{d^2 y_{1}}{d 
\rho^2}. 
    \end{split}
    \end{equation}
By the condition (G2), \eqref{eq-y76},  
\eqref{eq-y51-2}, \eqref{eq-y89}, 
\eqref{eq-y86}, 
\eqref{eq-y88} and \eqref{eq-y87}, one has 
    \[
    \begin{split}
    & 
    |\frac{d y_{2}}{d \rho}|, 
    |H'(y_{1}) \frac{d y_{1}}{d \rho} 
y_{2}|, | G(\rho) \frac{d y_{2}}{d \rho}|, 
\leq C \rho^{-1 - \frac{1}{q} + o(1)} 
(\log \rho), \\
& 
|\frac{d^{2} y_{1}}{d \rho^2}|, 
|G(\rho) \frac{d^{2} y_{1}}{d 
\rho^2}|, 
|H'(y_{1}) \left(\frac{d y_{1}}{d 
\rho}\right)^2|, 
|K(\rho) \left(\frac{d y_{1}}{d 
\rho}\right)^2|, 
|\rho K(\rho) 
\frac{d y_{1}}{d 
\rho} 
\frac{d^2 y_{1}}{d 
\rho^2}|, |\rho \left(\frac{d y_{1}}{d 
\rho}\right)^2 K'(\rho)| \leq 
C \rho^{-1 - \frac{1}{q} + o(1)}. 
    \end{split}
    \]
From this and \eqref{eq-y82}, we have 
    \begin{equation} \label{eq-y90}
    |\frac{d^2 y_{2}}{d \rho^2}| \leq 
    C \rho^{- 2 - \frac{1}{q} 
    + o(1)}\log \rho.  
    \end{equation}
Differentiating \eqref{eq-y82}, 
we obtain 
     \begin{equation} \label{eq-y83}
    \begin{split}
2\frac{d^2 y_{2}}{d \rho^2} 
+ \rho \frac{d^{3} y_{2}}{d \rho^{3}}
& = -  H''(y_{1}) 
\left(\frac{d y_{1}}{d \rho}\right)^{2}y_{2}
- H'(y_{1}) \frac{d^{2} y_{1}}{d \rho^2} y_{2}
- H'(y_{1}) \frac{d y_{1}}{d \rho} 
\frac{d y_{2}}{d \rho} \\
& \quad 
- H'(y_{1}) \frac{d y_{1}}{d \rho} 
\frac{d y_{2}}{d \rho} 
- G(\rho) \frac{d^{2} y_{2}}{d \rho^{2}} 
- \frac{d^{3} y_{1}}{d \rho^3} 
- H'(y_{1}) \frac{d y_{1}}{d \rho} 
\frac{d^2 y_{1}}{d \rho^2}
- G(\rho) \frac{d^{3} y_{1}}{d \rho^3} 
\\
& \quad 
- H''(y_{1}) 
\left(\frac{d y_{1}}{d \rho} \right)^3
-2 H'(y_{1})\frac{d y_{1}}{d \rho} 
\frac{d^{2} y_{1}}{d \rho^{2}} 
- K'(\rho) \left(\frac{d y_{1}}{d \rho} \right)^2 
- 2 K(\rho) \frac{d y_{1}}{d \rho} 
\frac{d^{2} y_{1}}{d \rho^2}
\\
& \quad 
- K'(\rho) \left(\frac{d y_{1}}{d \rho} \right)^2
- \rho K''(\rho) \left(\frac{d y_{1}}{d \rho} \right)^2
- 2 \rho K'(\rho) \frac{d y_{1}}{d \rho} 
\frac{d^{2} y_{1}}{d \rho^2} \\
& \quad 
- 2 K(\rho) \frac{d y_{1}}{d \rho} 
\frac{d^{2} y_{1}}{d \rho^2} 
- 2 \rho K'(\rho) \frac{d y_{1}}{d \rho} 
\frac{d^{2} y_{1}}{d \rho^2} 
- 2 K(\rho) \left(\frac{d^{2} y_{1}}{d \rho^2} \right)^2 
-2 \rho K(\rho)\frac{d y_{1}}{d \rho} 
\frac{d^{3} y_{1}}{d \rho^3}  \\
& = 
-  H''(y_{1}) 
\left(\frac{d y_{1}}{d \rho}\right)^{2}y_{2}
- H'(y_{1}) \frac{d^{2} y_{1}}{d \rho^2} y_{2}
- 2 H'(y_{1}) \frac{d y_{1}}{d \rho} 
\frac{d y_{2}}{d \rho}
- G(\rho) \frac{d^{2} y_{2}}{d \rho^{2}}
- \frac{d^{3} y_{1}}{d \rho^3}
 \\
& \quad 
- 3H'(y_{1}) \frac{d y_{1}}{d \rho} 
\frac{d^2 y_{1}}{d \rho^2}
- G(\rho) \frac{d^{3} y_{1}}{d \rho^3}
- H''(y_{1}) \left(\frac{d y_{1}}{d \rho} \right)^3 
\\
& \quad 
 - 2 K'(\rho) \left(\frac{d y_{1}}{d \rho} \right)^2
- 4 K(\rho) \frac{d y_{1}}{d \rho} 
\frac{d^{2} y_{1}}{d \rho^2} 
- \rho K''(\rho) \left(\frac{d^{2} y_{1}}{d \rho^2} \right)^2 
 \\
& \quad    
- 4 \rho K'(\rho) \frac{d y_{1}}{d \rho} 
\frac{d^{2} y_{1}}{d \rho^2}
- 2 K(\rho) \left(\frac{d^{2} y_{1}}{d \rho^2} \right)^2
-2 \rho K(\rho)\frac{d y_{1}}{d \rho} 
\frac{d^{3} y_{1}}{d \rho^3}
    \end{split}
    \end{equation}
By the conditions (G2), \eqref{eq-y51-2}, 
\eqref{eq-y85}, 
\eqref{eq-y89}, \eqref{eq-y91}, 
\eqref{eq-y88} and \eqref{eq-y90}, one has 
    \[
    \begin{split}
&
|H''(y_{1}) 
\left(\frac{d y_{1}}{d \rho}\right)^{2}y_{2}|, \;
|H'(y_{1}) \frac{d^{2} y_{1}}{d \rho^2} y_{2}|, \;
|H'(y_{1}) \frac{d y_{1}}{d \rho} 
\frac{d y_{2}}{d \rho}|,  \;
|G(\rho) \frac{d^{2} y_{2}}{d \rho^{2}}|
\leq C \rho^{-2 - \frac{1}{q} + o(1)} \log \rho, 
\\
& 
|\frac{d^{3} y_{1}}{d \rho^3}|, \;
|H'(y_{1}) \frac{d y_{1}}{d \rho} 
\frac{d^2 y_{1}}{d \rho^2}|, \;
|G(\rho) \frac{d^{3} y_{1}}{d \rho^3}|, \;
|H''(y_{1}) \left(
\frac{d y_{1}}{d \rho} \right)^3|  
\leq C \rho^{-2 - \frac{1}{q} + o(1)}.
    \end{split}
    \]
By \eqref{eq-y85}, 
\eqref{eq-y89}, \eqref{eq-y91}, \eqref{eq-y86}, 
\eqref{eq-y88}, \eqref{eq-y87}, \eqref{eq-y94} 
and \eqref{eq-y90}, 
one has 
    \[
    \begin{split}
    & 
    |K'(\rho) \left(\frac{d y_{1}}{d \rho} \right)^2|,
    \;
    |K(\rho) \frac{d y_{1}}{d \rho} 
\frac{d^{2} y_{1}}{d \rho^2}|, \;
 |\rho K'(\rho) \frac{d y_{1}}{d \rho} 
\frac{d^{2} y_{1}}{d \rho^2}|, \;
|\rho K(\rho)\frac{d y_{1}}{d \rho} 
\frac{d^{3} y_{1}}{d \rho^3}|
\leq C \rho^{-2 - \frac{1}{q} + o(1)}, \\
    & |K(\rho) \left(\frac{d^{2} y_{1}}{d \rho^2} 
    \right)^2|
    \leq C \rho^{-3 - 
    \frac{1}{q} + o(1)}, \\
    & 
    |\rho K''(\rho) \left(\frac{d^{2} y_{1}}{d \rho^2} \right)^2|, \;
    |K'(\rho) \left(\frac{d^{2} y_{1}}{d \rho^2} \right)^2|
    \leq C \rho^{-4 - \frac{1}{q} + o(1)}
    \end{split}
    \]
Thus, we see from \eqref{eq-y83} that 
    \[
    |\frac{d^{3} y_{2}}{d \rho^{3}}| 
    \leq C \rho^{-3 - \frac{1}{q} + o(1)} 
    \log \rho.
    \]

\begin{thank}
H.K. was supported by JSPS 
KAKENHI 
Grant Number JP25H01453 and 
JP25K07089. K.K was supported by 
JSPS KAKENHI Grant Number 
23KJ0949 
\end{thank}

 \subsection*{Data availability}
Data sharing not applicable to this article as no datasets were generated or analyzed during the current study.

\subsection*{Conflict of interest}
The authors declare that they have no conflict of interest.

\bibliographystyle{plain}
\bibliography{singular_solution_2}

\noindent
Hiroaki Kikuchi
\\[6pt]
Department of Mathematics,
\\[6pt]
Tsuda University,
\\[6pt]
2-1-1 Tsuda-machi, Kodaira-shi, Tokyo 187-8577, JAPAN
\\[6pt]
E-mail: hiroaki@tsuda.ac.jp

\vspace{0.5cm}

\noindent
Kenta Kumagai
\\[6pt]
Department of Mathematics,
\\[6pt]
Institute of Science Tokyo,
\\[6pt]
2-12-1 Ookayama, 
Meguro-ku, Tokyo 152-8551, JAPAN
\\[6pt]
E-mail:kumagai.k.ah@m.titech.ac.jp

\end{document}